\documentclass[a4paper,11pt]{article}

\usepackage[english]{babel}
\usepackage[utf8]{inputenc}
\usepackage{amsmath}
\usepackage{amsfonts}
\usepackage{graphicx}
\usepackage[colorinlistoftodos]{todonotes}
\usepackage{tikz-cd}
\usepackage{amsthm}
\usepackage{lipsum}
\usepackage{epigraph}
\theoremstyle{definition}
\newtheorem{definition}{Definition}[subsection]
\newtheorem{example}{Example}[subsection]
\newtheorem{remark}{Remark}[subsection]
\theoremstyle{plain}
\newtheorem{proposition}{Proposition}[subsection]
\newtheorem{theorem}{Theorem}[subsection]
\newtheorem{lemma}{Lemma}[subsection]
\newtheorem{corollary}{Corollary}[subsection]
\DeclareMathOperator{\lt}{LT}
\DeclareMathOperator{\lc}{LC}
\DeclareMathOperator{\dom}{dom} 
\DeclareMathOperator{\codom}{codom}
\DeclareMathOperator{\obj}{obj}
\DeclareMathOperator{\Hom}{Hom}
\DeclareMathOperator{\GL}{GL}
\DeclareMathOperator{\Ima}{Im}
\DeclareMathOperator{\Tor}{Tor}
\DeclareMathOperator{\Ext}{Ext}
\DeclareMathOperator{\id}{id}
\DeclareMathOperator{\lm}{LM}

\title{Gr\"{o}bner Bases: Connecting Linear Algebra with Homological and Homotopical Algebra}
\author{Soutrik Roy Chowdhury}
\date{}

\begin{document}
\maketitle
\begin{abstract}
The main objective of this thesis is to connect the theory of Gr\"{o}bner bases to concepts of homological algebra. Gr\"{o}bner bases, an important tool in algebraic system and in linear algebra help us to understand the structure of an algebra presented by its generators and relations by constructing a basis of its set of relations. In this thesis we mainly deal with graded augmented algebras. Given a graded augmented algebra with its generators and relations, it is possible to construct a free resolution from its Gr\"{o}bner basis, known as Anick's resolution. Though rarely minimal, this resolution helps us to understand combinatorial properties of the algebra. The notion of a $K_2$ algebra was recently introduced by Cassidy and Shelton as a generalization of the notion of a Koszul algebra. We compute the Anick's resolution of $K_2$ algebra which shows several nice combinatorial properties. Later we compute some derived functors from Anick's resolution and outline how to construct $A_\infty$-algebra structure on $\Ext$-algebra from a minimal graded projective resolution in which we can use Anick's resolution as a tool. Thus this thesis provides an unique platform to connect concepts of linear algebra with homological algebra and homotopical algebra with the help of Gr\"{o}bner bases.
\end{abstract}
\pagenumbering{gobble}% Remove page numbers (and reset to 1)
\clearpage
\thispagestyle{empty}
\section*{Introduction} 
In this thesis, we discuss two major topics in algebra and algebraic system: Gr\"{o}bner bases [1][2] and Anick's resolution [15][14][12].\\ 
 Given polynomials $f, g \in \mathbb{Q}[x]$ with $g \neq 0$, there exist uniquely determined polynomials $q$ and $r$ in $\mathbb{Q}[x]$ such that $f = qg +r$ and $\deg r < \deg g$. The polynomial $r$ is called the remainder of $f$ with respect to $g$.\\ 
 This fundamental fact, known from high school, can be generalized to polynomial algebras over an arbitrary field. Indeed, there is an algorithm, known as Euclidean algorithm, to compute $q$ and $r$. The algorithm works as follows: if $\deg f < \deg g$, then $q=0$ and $r=f$. If $\deg f \geq \deg g$, let $ax^n$ with $a \in \mathbb{K}$ be the leading term of $f$ and $bx^m$ with $b \in \mathbb{K}$ the leading term of $g$. Then the degree of $r_1 = f - (a/b)x^{n-m}g$ is less than $\deg f$. If $\deg r_1 < \deg g$, then $q= (a/b)x^{n-m}$ and $r=r_1$. Otherwise one applies the same reduction to $r_1$ and arrives, after a finite number of steps, to the desired presentation.\\
 We would like to have similar division algorithm for polynomials in several variables. In section 3.6 we will give such an algorithm. Apparently there are two problems to overcome. The first is, that it is not clear what the leading term of such a polynomial should be. For example, there will be a problem to find the leading term of the polynomial $f = x_1^2x_2 + x_1x_2^2$. To define a leading term we must have a total order on the monomials. The second is, that one has to say what it means that the remainder is `small' compared with the divisors. We will deal with these problems with the concepts \textit{reductions} and \textit{S-polynomials}.\\
 The above desired division algorithm for polynomials with multi-variables comes under the notion of Gr\"{o}bner bases which is a type of basis defined on an ideal $I$ of an algebra. In 1965, Buchberger introduced the notion of Gr\"{o}bner bases for a polynomial ideal and an algorithm (known as Buchberger's algorithm) for their computations [1]. Since the end of seventies, Gr\"{o}bner bases have been an essential tool in the development of computational techniques for the symbolic solution to the systems of equations and in the development of effective methods in Algebraic Geometry and Commutative Algebra.\\ 
 Bergman [11] was the first to extend the Gr\"{o}bner bases and Buchberger algorithm to ideals in the free non-commutative algebra; later, [16] made precise in which sense Gr\"{o}bner bases can be computed in the free non-commutative algebra: there is a procedure which halts if and only if the ideal has a finite Gr\"{o}bner basis, in which case it returns such a basis; though it has to be noted that the property of having finite Gr\"{o}bner basis is undecidable. Gr\"{o}bner bases have also been generalized to various non-commutative algebras, of interest in Differential Algebra (e.g. Weyl algebras, enveloping algebras of Lie algebras).\\ 
 We begin this thesis with some background materials used to understand Gr\"{o}bner bases and Anick's resolution which we will describe in the later sections. In section 3.3, we define Gr\"{o}bner bases for both commutative and non-commutative algebras and introduce the notion of \textit{reduction} and \textit{S-polynomial} for both cases. We mention the Diamond's lemma which plays the pivotal role in the construction of Buchberger's algorithm.\\ 
 Let $D$ be an associative augmented $\mathbb{K}$-algebra. For many purposes one would like to have a projective resolution of $\mathbb{K}$ as a $D$-module. The bar resolution is always easy to define, but it is often too large to use in practice. On the other hand, minimal resolution [17] may exist, but they are often hard to write down in a way that is amenable to calculations. In 1986, David J Anick, in his paper `\textit{On the homology of associative algebras}' [15] described a new free resolution of $\mathbb{K}$ by $D$-modules which he has constructed from the Gr\"{o}bner basis of $D$. In section 4.2, we mention a theorem which presents the resolution. Though rarely minimal, it is small enough to offer some bounds but explicit enough to facilitate calculations, it is best suited for analyzing otherwise tricky algebras given via generators and relations.\\ 
 The notion of a $K_2$-algebra was recently introduced by Cassidy and Shelton in [10] as a generalization of the notion of a Koszul algebra. Conner and Goetz described an algebra $B_n$ in [8]. Theorem 3.10 of [8] states that for any $n \in \mathbb{N}$, $B_n$ is a $K_2$-algebra. We compute Anick's resolution of algebra $B_n$ (and hence of algebra $K_2$) in section 4.4.\\
 Later we mention some introductory category theory [4][5] and some concepts of homological algebra [6] which is used to show how we can use Anick's resolution to construct derived functors and $A_\infty$-algebra structures associated to $\Ext$-algebras. Derived functor is an important concept in homological algebra and $A_\infty$ algebra structure, which was invented in early sixties by James Stasheff, provides an important tool in homotopical algebra with a wide range on applications in mathematical physics, mirror symmetry, topology etc. This paper thus presents a unique bridging between Gr\"{o}bner bases, an important tool in algebraic system and theoretical computer science with homological algebra and homotopical algebra by constructing a free resolution from Gr\"{o}bner basis and later show the construction of derived functors and finally constructing $A_\infty$ structures from minimal free resolution through Merkulov's construction.
\clearpage
\section*{Acknowledgement} 
\epigraph{"As are the crests on the heads of peacocks, as are the gems on the hoods of cobras, so is mathematics at the top of all sciences."}{The Yajurveda (c. 600 BCE)}
Mathematics is a form of art which is used to draw several scientific tools. It depends on the artist to make the art beautiful and meaningful in front other people. If one devotes his time to understand mathematics, he will surely fall in love with mathematics. Though paying all my love and attention to learn mathematics and make this thesis an unique place to bridge linear algebra with homological algebra as well as with homotopical algebra (with a brief outline), this work is not possible without the help of the following people, both academically and non-academically. They are my \textit{angels}. \\
At first, I would like to thank my thesis advisor, \textit{ Prof. Vladimir Dotsenko}. Thank you very much for all your support and love to make my work fruitful. Whenever I was in trouble, you were always there to help me understanding different topics and from time to time  giving me motivations to learn various topics which plays a pivotal part in my learning. I would also like to say thanks for your help in my non-academic life also. Whenever I faced difficulties with rules and regulations in Trinity, you were always there to provide suggestions.\\ 
I would like to thank \textit{Fran}, \textit{Breda} and the \textit{family} for providing me such beautiful hospitality and a family care during my stay in Ireland. As an international student, adjusting in a new country is always a problem. I never face such problem and that is because of you. \textit{Fran} and \textit{Breda}, you both are truly my Irish parents.\\
 And finally thanks to my beloved mom and dad. Thanks for having faith on me. I must say, this is the beginning of devoting myself to mathematics, still long way to go. \textit{Mani} and \textit{Bapi}, I love you very much.
\clearpage 
\section{Notations}
\begin{itemize}
\item Let $A$ be a $\mathbb{K}$-algebra where $\mathbb{K}$ is a field. For any $a, b \in A$, $a \cdot b$ and $ab$ are same. 
\item By $\mathbb{N}, \mathbb{Q}$ we denote the set of natural numbers and the set of rational numbers respectively.
\item Let $I \subset A$ be an ideal of algebra $A$. By $I = (a)$, we mean $I$ is generated by an element $a \in I$.
\item For a word $f$, $\overline{f}$ denotes the reduction of $f$ to its normal form.
\end{itemize}
\clearpage 
\pagenumbering{arabic}
\tableofcontents 
 
\section{Background Materials}
\subsection{Algebra}
\begin{definition}
An \textbf{algebra} is a vector space $A$ (over a field $\mathbb{K}$) equipped with a multiplication $\mu: A \otimes A \rightarrow A$, given by $x, y \mapsto \mu(x,y) = xy$, with the following properties:
\begin{itemize}
\item $(x + y)z = xz + yz$ for $x,y,z \in A$
\item $x(y +z) = xy +xz$ for $x,y,z \in A$
\item $(ab)(xy) = (ax)(by)$ where $x,y \in A$ and $a,b \in \mathbb{K}$.
\end{itemize}
\end{definition}
\begin{example}
Let us consider the vector space of all commutative polynomials in $n$ variables over some field $\mathbb{K}$ with usual polynomial addition and multiplication, let us  denote it by $\mathbb{K}[x_1,x_2,\dots,x_n]$. We take the multiplication as usual polynomial multiplication, then the vector space $\mathbb{K}[x_1,x_2,\dots,x_n]$ form an algebra over $\mathbb{K}$. 
\end{example}
\begin{definition}
Let $A$ be an algebra over a field $\mathbb{K}$. Then $I \subset A$ is a \textbf{left ideal} of $A$ if:
\begin{enumerate}
\item  $x-y \in I$ for any $x, y \in I$.
\item $ax \in I$ for any $a \in A$ and any $x \in I$.
\end{enumerate}
Similarly $I \subset A$ is a \textbf{right ideal} of $A$ if condition 1 holds and $xa \in I$ for any $a \in A$ and $x \in I$. $I$ is a \textbf{two-sided ideal} (more generally 'an ideal') if it is a left as well as a right ideal.
\end{definition}
\begin{example}
For algebra $\mathbb{K}[x]$, the set generated by $x^2$ (usually denoted by $(x^2)$) is an ideal of $\mathbb{K}[x]$.
\end{example}
\subsection{Associative algebra}
\begin{definition}
An algebra is \textbf{associative} if for the multiplication $\mu : A \otimes A \rightarrow A$ and for any $a_1, a_2, a_3 \in A$ we have the equality
\begin{equation}
((a_1a_2)a_3) = (a_1(a_2a_3)).
\end{equation}
An associative algebra is called \textbf{commutative} if for any $a_1,a_2 \in A$ we have
\begin{equation}
(a_1a_2) = (a_2a_1).
\end{equation}
\end{definition}
\begin{example}
Algebra of commutative polynomials with either single or multivariables over a field $\mathbb{K}$ is an example of commutative associative algebra.We usually denote the algebra of commutative polynomials with $n$ variables by $\mathbb{K}[x_1,x_2,\dots,x_n]$.\\ 
 However algebra of non-commutative polynomials are examples of \\  non-commutative associative algebra and we usually denote a non-commutative polynomial algebra with $n$ variables by $\mathbb{K}\langle x_1,x_2,\dots,x_n\rangle$.
\end{example}
\subsection{Augmentation (Algebra)}
\begin{definition}
An \textbf{augmentation} of an associative algebra $A$ over a field $\mathbb{K}$ is an algebra homomorphism $ \epsilon : A \rightarrow \mathbb{K}$. An algebra together with an augmentation is called an \textbf{augmented algebra}.
\end{definition}
\begin{example}
Let $A = \mathbb{K}[G]$ be group algebra[18] of a group $G$, then $\epsilon : A \rightarrow \mathbb{K}$ given by,
\[  \sum a_ix_i \mapsto \sum a_i  \]
where $a_i's  \in \mathbb{K}$ and $x_i's \in G$, is an augmentation.
\end{example}
\subsection{Graded algebra}
\begin{definition}
An algebra $A$ is \textbf{graded} if:
\begin{itemize}
\item $ A = \bigoplus_{i=0}^{\infty}A_i$ where $A_i$'s are subspaces of $A$.
\item $ A_iA_j \subseteq A_{i+j}$.
\end{itemize}
\end{definition}
\begin{example}
The polynomial algebra is graded by degree.
\end{example}
\begin{remark}
Throughout this paper, we deal with finite dimensional subspaces of an algebra. It is possible to superimpose some homogeneous relations of a given degree so that the algebra remains infinite dimensional regardless of the form of its relations, stated by the \textbf{Golod-Shafarevich} theorem [3]. 
\end{remark}
\subsection{Hilbert Series}
One of our main goal is to find tools to investigate properties of algebras. One of the main properties we are interested in is the size of an algebra. Most algebras we will be considering are infinite dimensional, so the question needs to be better defined. In the graded case at least, this is simple - we can ask the dimension of each graded component. This gives a sequence of numbers, which describes the size of the entire algebra. When we have a sequence of numbers, we can talk about the generating function. This is called \textit{Hilbert series} for our graded algebra.
\begin{definition}
Let $A = \bigoplus_{n=0}^{\infty}A_n$ be a graded algebra. Then the formal series
\[ H_A = \sum_{n=0}^{\infty}\dim(A_n)t^n \]
is called the \textbf{Hilbert series} of the algebra $A$.
\end{definition}
\begin{example}
Let us give an example - how to find the Hilbert series of the free associative algebra $A = \mathbb{K}\langle x,y,z\rangle$.\\
One can easily see that this algebra is graded - the graded components are homogeneous subspaces of each degree. So the subspace of degree $0$ has dimension $1$. The subspace of degree $1$ is generated by $x,y,z$. So that has dimension $3$. The subspace of degree $2$ is generated by $x^2,xy,xz,yx,y^2,yz,zx,$\\$zy,z^2$, which is of dimension $9$.\\
In general, the subspace of degree $n$ has dimension $3^n$ (as there are $3$ independent choices for each of the $n$ positions). So
\[ H_A = \sum_{n=0}^{\infty}3^nt^n = \frac{1}{1-3t}   \]
\end{example}
\begin{theorem}
If we have two graded algebras $U$ and $V$, then we grade the algebra $U \oplus V$ with components $U_n \oplus V_n$ (except when $n=0$, in which case $(U \oplus V)_n = K$). Similarly we can define the graded components of $U \otimes V$ as $\sum_{i=0}^{n}U_i \otimes V_{n-i}$. Then
\[ H_{U \oplus V} = H_U + H_V  \]
and
\[  H_{U \otimes V} = H_UH_V   \]
\end{theorem}
\begin{proof}
For the first equality, we simply note that:
\[ \dim(U \oplus V)_n = \dim U_n + \dim V_n.   \]
For the second,
\[  \dim(U \otimes V)_n = \sum_{i=0}^{n}\dim U_i \dim V_{n-i}  \]
this exactly matches the co-efficient of $t^n$ in $H_UH_V$.
\end{proof}
\begin{corollary}
We call a subspace $V$ homogeneous if $V = \oplus A_n \cap V$. If $U$,$V$ are homogeneous subspaces of some algebra, then $H_{U+V} \leq H_U + H_V$ and $H_{UV} \leq H_UH_V$ (where the inequality is co-efficient wise, and the sum and product are set sum and multiplications).
\end{corollary}
\begin{proof}
This directly follows from theorem 2.5.1. We have equality if every element of $H_{U+V}$ can be written uniquely as $u+v$.
\end{proof}
\begin{example}
(Hilbert series of polynomial algebra):\\
Let us give another example - how to find the Hilbert series of the polynomial algebra $A = \mathbb{K}[x,y]$.\\
Again we can take grading by degree. So the subspace of degree $0$ has dimension $1$. The subspace of degree $1$ is generated by $x,y$, so it has dimension $2$. The subspace of degree $2$ is generated by $x^2,xy,y^2$, which has dimension $3$.\\
In general, the subspace of degree $n$ has dimension $n+1$ (as the word is determined by the position where the $x$'s stop and the $y$'s begin; there are $n+1$ such places). So
\[  H_A = \sum_{n=0}^{\infty}(n+1)t^n  \]
we note that
\[  \frac{1}{1-t} = \sum_{n=0}^{\infty}t^n.  \]
Taking derivatives on both sides give:
\[ \frac{1}{(1-t)^2} = \sum_{n=0}^{\infty}(n+1)t^n    \]
So
\[  H_A = \frac{1}{(1-t)^2}  \]
\end{example}
\begin{theorem}
The Hilbert series of the polynomial algebra $\mathbb{K}[x]$ is computed by the formula
\[ H_{\mathbb{K}[x]} = \prod_{x \in X}(1- t^{|x|})^{-1}.   \]
The Hilbert series of the exterior algebra $\bigwedge \mathbb{K}[x]$ is calculated by the formula
\[ H_{\bigwedge \mathbb{K}[x]} = \prod_{x \in X}(1 + t^{|x|}).   \]
In particular, in case of natural graduation and a finite set of generators $d$, we have:
\[ H_{\mathbb{K}[x]}^{-1} = (1-t)^d;\hspace{5mm} H_{\bigwedge \mathbb{K}[x]} = (1+t)^d     \]
\end{theorem}
\begin{proof}
In the case of one generator, the Hilbert series in the power of $n = |x|$ is computed straightforwardly: it is equal to $1 + t^n + t^{2n} + \dots = (1-t^n)^{-1}$ in case of polynomial ring and $1+t^n$ in the case of exterior algebra[19]. The case of finite number of generators reduces to this one, with the help of theorem 2.5.1. Finally, in case of infinite number of generators, the degree of generators must increase, for if not, we do not get finite-dimensionality of graded subspaces. Consequently, for every $n$, the segment of the Hilbert series up to the exponent $n$ depends only on finite number of generators with the degree not exceeding $n$, thus everything reduces to the finite case.
\end{proof}
\begin{remark}
If we apply the formula of Hilbert series for polynomial algebra of theorem 2.5.2 to our previous computed example 2.5.2 we will get the exact result.
\end{remark}
\subsection{Hilbert series for free product} 
\begin{definition}
We define the \textbf{free product} of two algebra $A,B$ as the disjoint union of their generators, with both sets of relations. We usually denote it by $A \ast B$.
\end{definition}
\begin{example}
The free product of the algebra $A = \langle x\hspace{1mm} |\hspace{1mm} x^3 + 2x^2 \rangle$ and the algebra $B = \langle x,y\hspace{1mm} |\hspace{1mm} 2x^2 = y^2 \rangle$ is given by the algebra $ A \ast B = \langle x,y,z\hspace{1mm} |\hspace{1mm} x^3 + 2x^2 , 2y^2 = z^2 \rangle$.
\end{example}
\begin{theorem}[Hilbert series of free product]
If $A,B$ are graded algebras, then
\[ (H_{A \ast B})^{-1} = H_A^{-1} + H_B^{-1} - 1.   \]
\end{theorem}
\begin{proof}
Any word/monomial in $A \ast B$ is either begins with a (non-scalar) element of $A$, or an element of $B$ (excluding terms that belong to the underlying field). This follows from the fact that, because the two sets of generators have no overlap, the Gr\"{o}bner basis of $A \ast B$ will be the union of the Gr\"{o}bner basis of $A$ and the Gr\"{o}bner basis of $B$. So take any monomial in $A \ast B$. It begins with a generator either from $A$ or $B$. Without loss of generality say its from $A$. Then take the longest prefix of this word that consists of generators from $A$. This must be a word in $A$, otherwise it would contain a leading term of the Gr\"{o}bner basis for $A$, and hence a leading term in the Gr\"{o}bner basis for $A \ast B$.\\ 
So we know that words in $A \ast B$ start with a word from either $A$ or $B$. That suggests the following decomposition:
\[  H_{A \ast B} = H_{V_1} + H_{V_2} +1     \]
where $V_1$ are the elements that begin with a non-scalar element of $A$, and $V_2$ are the elements that begin with a non-scalar element of $B$. So we have
\[  H_{V_1} = (H_A -1)(H_{V_2} + 1)    \]
\[  H_{V_2} = (H_B -1)(H_{V_1} + 1).     \]
The above imply that
\[  H_{V_1}  = \frac{H_AH_B - H_B}{H_A + H_B - H_AH_B}    \]
\[  H_{V_2}  = \frac{H_AH_B - H_A}{H_A + H_B - H_AH_B}    \]
So
\[  H_{A \ast B} = \frac{H_AH_B - H_B}{H_A + H_B - H_AH_B} + \frac{H_AH_B - H_A}{H_A + H_B - H_AH_B} + 1      \]
\[  = \frac{2H_AH_B - H_A - H_B - H_AH_B + H_A + H_B}{H_A + H_B - H_AH_B}    \]
\[  = \frac{H_AH_B}{H_A + H_B - H_AH_B}.     \]
Hence
\[  (H_{A \ast B})^{-1} = H_A^{-1} + H_B^{-1} -1.   \]
\end{proof}
\newpage 
\section{Gr\"{o}bner Bases}
\subsection{Motivation}
 Let $\mathbb{K}[x_1,x_2, \dots, x_n]$ be a polynomial algebra. Let $I$ be an ideal of \\  $\mathbb{K}[x_1,x_2, \dots, x_n]$. Our aim is to study the structure of $\mathbb{K}[x_1,x_2, \dots, x_n]/I$ in a constructive way. In other words, we need to investigate given $f \in \mathbb{K}[x_1,x_2, \dots, x_n]$, whether it belongs to $I$ or not. For single variable polynomial algebra it is easy to determine as $\mathbb{K}[x]$ is an Euclidean domain, so we can perform Euclidean algorithm to know whether $f$ is in $I$ or not. But in case of multivariable polynomial algebra ($\mathbb{K}[x_1,x_2, \dots, x_n]$) it is difficult to say as $\mathbb{K}[x_1,x_2, \dots, x_n]$ is neither an Euclidean domain nor a principal ideal domain.  
\begin{lemma}
The polynomial algebra  $\mathbb{K}[x_1,x_2,\dots,x_n]$ is not a principal ideal domain for $n >1$.
\end{lemma}
\begin{proof}
Take an ideal generated by $\{x_1,x_2\}$.\hspace{1mm}If $f$ generates this ideal, then $f$ divides both $x_1$ and $x_2$,\hspace{1mm}so $f$ is a constant term. So our ideal must be the entire algebra. But $1$ is in the algebra,\hspace{1mm}but not in the ideal. Contradiction.
\end{proof}
To solve this problem we have the concept of Gr\"{o}bner basis which is a type of basis defined carefully which tells that if we replace our generators $f_i$ of the ideal $I$ with a Gr\"{o}bner basis $g_j$ of the same ideal then we have the property that the remainder of $f$ on division by the polynomials $g_j$ is $0$ if and only if $f$ is in the ideal. Later we also observe that monomials not divisible by leading terms of the Gr\"{o}bner basis $G$ form a basis of $\mathbb{K}[x_1,x_2, \dots, x_n]/I$, which is sufficient to describe the structure of the algebra modulo $I$. \\
So we understand that to study the structure of $ \mathbb{K}[x_1,x_2,\dots,x_n]/I$ in a constructive way we require the concept of Gr\"{o}bner bases.The original definition was given in Bruno Buchberger's PhD thesis in 1965[1]. Before moving to the core concepts of Gr\"{o}bner bases we require some preliminary tools to describe it.

\subsection{Admissible ordering, Leading terms, Leading co-efficients}

\begin{theorem}[Hilbert basis theorem]
$I \subset \mathbb{K}[x_1,x_2,\dots,x_n]$ is always finitely generated,\hspace{1mm}so there exist $f_1,f_2,\dots,f_m \in \mathbb{K}[x_1,x_2,\dots,x_n]$ such that $I=  (f_1,f_2,\dots,f_m)$.\\
\end{theorem}
\begin{proof}
For proof, we refer to [2].
\end{proof}

\begin{definition}
An \textbf{admissible ordering}\hspace{1mm} $`<'$\hspace{1mm} of monomials is a total ordering of all monomials in $\mathbb{K}[x_1,x_2,\dots,x_n]$ such that
\begin{itemize}
\item it is a well ordering. Equivalently we say that there is no infinite decreasing sequence.
\item $m_1 < m_2 \Rightarrow m_1m_3 < m_2m_3 $ for any monomial $m_3$,\hspace{1mm}where $m_1,m_2 \in \mathbb{K}[x_1,x_2,\dots,x_n]$.
\end{itemize}
\end{definition}

\begin{lemma}
There is only one admissible ordering of monomials  in $\mathbb{K}[x]$ i.e.\\
\[ x^k < x^l\hspace{2mm} \text{if and only if}\hspace{2mm} k < l\]
\end{lemma}
\begin{proof}
This is easy to verify. We take an ordering $1 < x$ which implies $x < x^2$, $x^2 < x^3$, $\dots$. so we get an well ordering. Also take $x^m < x^l$ for some $m, l \in \mathbb{N}$, then for any $x^k$ for some $k$, we always have $x^{m+k} < x^{l+k}$. These all together imply that the ordering we choose is admissible. Now suppose we take $x < 1$, which implies $x^2< x$, $x^3 < x^2$, $\dots$,\hspace{1mm}then this implies existence of an infinite decreasing sequence,\hspace{1mm}hence contradiction.
\end{proof}
%\begin{remark} 
%For $n \geq 2$,\hspace{1mm} there are infinitely many admissible orderings. 
%\end{remark}
\begin{example}
\textbf{LEX}(lexicographic ordering),
\[ x_1^{i_1}x_2^{i_2}x_3^{i_3}\dots x_n^{i_n} < x_1^{j_1}x_2^{j_2}x_3^{j_3}\dots x_n^{j_n} \]
if
\[ i_1 < j_1 \hspace{3mm}\text{or} \hspace{1mm} \]
 \[  i_1=j_1 , i_2 < j_2 \hspace{3mm}\text{or} \hspace{1mm} \]
 \[  i_1=j_1 , i_2=j_2 , i_3 < j_3 \hspace{3mm} \text{or} \hspace{1mm}\]
  \[ \vdots \]
\end{example}
\begin{example}
\textbf{DEGLEX}(degree-lexicographic ordering),\\
a little difference with \textbf{LEX} is that here first we consider the degree then the \textbf{LEX} ordering.
\[ x_1^{i_1}x_2^{i_2}x_3^{i_3}\dots x_n^{i_n} < x_1^{j_1}x_2^{j_2}x_3^{j_3}\dots x_n^{j_n}  \]
if
\[ i_1+i_2+i_3+\dots+i_n < j_1+j_2+j_3+\dots+j_n \hspace{3mm} \text{or}\]
\[ i_1+i_2+i_3+\dots+i_n = j_1+j_2+j_3+\dots+j_n \hspace{3mm} \text{and}\]
\[ x_1^{i_1}x_2^{i_2}x_3^{i_3}\dots x_n^{i_n}\hspace{1mm} <_\textbf{LEX}\hspace{1mm} x_1^{j_1}x_2^{j_2}x_3^{j_3}\dots x_n^{j_n}.\]
\end{example}
\newpage 
Let us fix an admissible ordering. Let $I \subset \mathbb{K}[x_1,x_2,\dots,x_n]$ be an ideal. 
\begin{definition}
By leading term of $I$ ( we usually denote it by $\lt(I)$) we mean space of linear combinations of monomials $m$ over $\mathbb{K}$ which are leading terms of elements of $I$. We say $m \in \lt(I)$ if there exists $f \in I$ such that
\[  f = cm + \sum c_im_i, \]
where $m_i$'s are monomials with $m_i < m$ and $c_i , c \in \mathbb{K}$ with $c \neq 0$.
\end{definition}
\begin{definition}
By $\lt(f)$, we denote the leading monomial in $f$. 
\end{definition}
\begin{example}
Let us fix DEGLEX ordering with $x > y$. Let $f$ be $x^3-xy^2$. Then according to the ordering $\lt(f) = x^3$. Now let us take the same $f$ and now we fix the DEGLEX order with $y>x$, then $\lt(f) = xy^2$.
\end{example}
\begin{definition}
By $\lc(f)$ (leading co-efficient), we denote the co-efficient of the leading monomial in $f$.
\end{definition}
\begin{example}
Let us consider the same $f$ from example 3.2.3. In case of ordering $x>y$, $\lc(f) = 1$. In case of ordering $y>x$, $\lc(f) = -1$.  
\end{example}
\begin{lemma}
$\lt(I)$ is itself an ideal in $\mathbb{K}[x_1,x_2,\dots,x_n]$.
\end{lemma}
\begin{proof}
Let $m \in \lt(I)$ be a monomial. Then according to the definition 3.2.2 there exists $f \in I$ such that $f = cm + \sum c_im_i$ with $m_i < m$ for all $i$ and $c_i , c \in \mathbb{K}$, $c \neq 0$. We form $m' = m''m$. Multiplying $m''$ with $f$ gives,
\[  m''f = cm''m + \sum c_i m''m_i.  \]
As $I$ is an ideal hence $m''f \in I$. $m_i < m$ implies $m''m_i < m''m$ for all $i$. These together imply that $m''m$ is the leading monomial of $m''f$. Hence $m' \in \lt(I)$.
\end{proof}

\begin{lemma}
Cosets of monomials $m \notin \lt(I)$ form a basis in $R = \mathbb{K}[x_1,x_2,\dots,x_n]/I$.
\end{lemma}
\begin{proof}
Let us first prove the linear independence. Let $m_1,m_2,\dots,m_l \notin \lt(I)$,\hspace{1mm}without loss of generality assume $m_1 < m_2 < \dots < m_l$ \hspace{1mm}(where $<$ is our fixed admissible ordering),\hspace{1mm}then we have $c_1m_1 + c_2m_2 +\dots + c_lm_l = 0$ \hspace{1mm} in $R$,\hspace{1mm}where $c_i 's \in \mathbb{K}$. Let $f = c_1m_1 + c_2m_2 +\dots + c_lm_l \in I $. Then $\lt(f) \in \lt(I)$,\hspace{1mm} a contradiction unless $f = 0$ implies $c_1 = c_2 = \dots = c_l=0 $.\\
Next we need to check the spanning set property i.e. we need to show that if $m \in \lt(I)$,\hspace{1mm} then $m$ is a linear combination in $R$ of cosets of monomials not present in $\lt(I)$. We will prove this with the help of contradiction. Let's take the smallest $m \in \lt(I)$ for which such a combination doesn't exist. Now by definition, there exists $f \in I$,\hspace{1mm} such that
$0 = f = cm + \sum c_im_i$,\hspace{1mm}with $m_i < m$,\hspace{1mm}and each of $m_i's$ is not in $\lt(I)$. Then we have $m = -\sum (c_i/c)m_i$ which can be represented as a combination of cosets of elements outside $\lt(I)$. Contradiction.
\end{proof}
\subsection{Gr\"{o}bner bases and Diamond lemma}
\begin{definition}
A subset $G \subset I$ is called a \textbf{Gr\"{o}bner basis} of $I$ if $\{\lt(g) |\hspace{1mm}g \in G \}$ generate the ideal $\lt(I)$ i.e. for each $f \in I$,\hspace{1mm}$\lt(f)$ is divisible by $\lt(g)$ for some $g \in G$.
\end{definition}
\begin{lemma}
$(G) = I$.
\end{lemma}
\begin{proof}
We know that $(G) \subset I$, assume that $(G) \neq I$. Let $f \in I\backslash(G)$ with smallest possible leading monomial. Then $\lt(f)= m \lt(g)$ for some $g \in G$,\hspace{1mm} $m \in I$. Let $\hat{f} = f - (\lt(f)/\lt(g))mg$. Then we have $\lt(\hat(f)) < \lt(f)$,\hspace{1mm}it implies $\hat{f} \in I \hspace{2mm} \text{ which implies}\hspace{2mm} \hat{f} \in (G)$,\hspace{1mm}then 
\[f = \hat{f} + (\lc(f)/\lc(g) )mg \in (G),\]
which is a contradiction.
\end{proof}
\begin{remark}
We have already proved that monomials not divisible by $\lt(G)$ form a basis of  $\mathbb{K}[x_1,x_2,\dots,x_n]/I$.
\end{remark}
\begin{definition}
 Given a Gr\"{o}bner basis $G \subset I$,\hspace{1mm}\textbf{normal monomials} with respect to $G$ are those monomials which are not divisible by $\lt(g)$,\hspace{1mm}for $g \in G$.\\
 We sometime call normal monomials as \textbf{normal words}. The space spanned by the normal words is usually denoted by $N$.
\end{definition}
\begin{lemma}
Cosets of normal monomials form a basis of \\  $\mathbb{K}[x_1,x_2,\dots,x_n]/I$.
\end{lemma}
\begin{proof}
This is proved earlier in form of lemma 3.2.3
\end{proof}
We are now going to define \textit{reduction} and \textit{S-polynomial} for the commutative case. These two definitions play an important role for computation of Gr\"{o}bner bases.
\begin{definition}
Suppose $f_1$,$f_2 \in \mathbb{K}[x_1,x_2, \dots, x_n]$,\hspace{1mm}and there exists a monomial $m$ with
\[ \lt(f_1) = m \lt(f_2).\]
Then
\begin{equation}
R_{f_2}(f_1) = f_1 - \frac{\lc(f_1)}{\lc(f_2)}mf_2
\end{equation}
is called \textbf{reduction} of $f_1$ with respect to $f_2$.
\end{definition}
\begin{definition}
Suppose we have two polynomials $f_1$ and \\  $f_2$ $\in \mathbb{K}[x_1,x_2,\dots,x_n]$,\hspace{1mm}suppose there exist monomials $m_1$,$m_2$ such that\\
\begin{equation} m_1\lt(f_2) = m_2\lt(f_1) \hspace{2mm}\text{and}\hspace{2mm} \deg(m_1) < \deg \lt(f_1) 
\end{equation}
then
\begin{equation}
 S(f_1,f_2) = \frac{1}{\lc(f_1)}m_2f_1 - \frac{1}{\lc(f_2)}m_1f_2
\end{equation}
is called \textbf{S-polynomial} with respect to a small common multiple (4).
\end{definition}
\begin{example}
Let us give an example to show how reduction and S-polynomial work, suppose we have a commutative polynomial algebra in 2 variable i.e. $\mathbb{K}[x,y]$.\hspace{1mm}We will pick \textbf{DEGLEX} ordering. Now let $f_1 = x^3 - y^2$ and $f_2 = x^3 -x +1$.\hspace{1mm}Then one can see for both $f_1$ and $f_2$, $\lm$ is $x^3$. So $\lt(f_1) = \lt(f_2) = x^3$ implies $m = 1$ so that $x^3 = 1\cdot x^3$,\hspace{1mm}our reduction of $f_1$ w.r.t $f_2$ will be then $x^3-y^2-(x^3-x+1) = x-y^2-1$.\\
And while computing S-polynomial with respect to a small common multiple we have 3 choices,
\[ 1 \cdot x^3 = 1 \cdot x^3\]
\[ x \cdot x^3 = x \cdot x^3 \]
\[ x^2 \cdot x^3 = x^2 \cdot x^3\]
hence for each cases we can compute the S-polynomial using our formula (5). For an example if we consider $x \cdot x^3 = x \cdot x^3  $ then our S-polynomial will be $x^2(x^3-y^2)-x^2(x^3-x+1) = x^3 - x^2y^2- x^2$.
\end{example}
\begin{definition}
A polynomial $f$ can be reduced to \textbf{$0$ modulo $G$},\hspace{1mm}if there exists $g_1,g_2,\dots,g_n \in G$ such that
\[ R_{g_m}(\dots\dots R_{g_2}(R_{g_1}(f))\dots) = 0\]
\end{definition}
\begin{lemma}
\textbf{Diamond lemma}:\\
A subset $G \subset I$ forms a Gr\"{o}bner basis if and only if for each $g_1,g_2 \in G$\\
$R_{g_2}(g_1)$ (if defined) can be reduced to $0$ modulo $G$.\\
And also for each $g_1,g_2 \in G$ and each small common multiple of\hspace{1mm} $\lt(g_1),\lt(g_2)$;\hspace{1mm}the corresponding S-polynomial can be reduced to $0$ modulo $G$.
\end{lemma}
We write another lemma which is equivalent to the \textit{diamond lemma}.
\begin{lemma}
Assume $(G) = I$,\hspace{1mm}then the following statements are equivalent:\\
(i) $G$ is a Gr\"{o}bner basis of $I$.\\
(ii) All reductions and all S-polynomials of pair of elements of $G$ can be reduced to $0$ modulo $G$.\\
(iii) For every $f \in I$,\hspace{1mm}$f$ admits a representation\\
\[ f=h_1g_1 + h_2g_2+\dots+h_ng_n \hspace{1mm};\hspace{1mm} g_i \in G\]
with \[ \lt(f) = \max(\lt(h_ig_i))\]
\end{lemma}
\begin{proof}
We need to show that
\[ (i) \hspace{2mm}\text{implies}\hspace{2mm}(ii) \hspace{2mm}\text{implies}\hspace{2mm}(iii)\hspace{2mm}\text{implies}\hspace{2mm} (i) \]
The implication of (iii) to (i) is trivial from the definition of Gr\"{o}bner bases as for each $f \in I$, $\lt(f) = \max(\lt(h_ig_i))$ of (iii) implies that $\lt(f)$ is divisible by $\lt(g_i)$ for some $g_i \in G$ which exactly the definition of Gr\"{o}bner basis tells. We will at first prove (i) implies (ii).\\
Suppose $G$ is a Gr\"{o}bner basis, then every $f \in I$ can be reduced to $0$ modulo $G$, ( $f \in I \Rightarrow \lt(f) = m\lt(g)$ for some $g \in G$,\hspace{1mm}then $R_g(f)$ has smaller leading term and we proceed with that and proceed further until get $0$. Similarly for $f \in I$ we have $m_1\lt(f)=m_2\lt(g)$ for some $g \in G$,\hspace{1mm}with $\deg(m_1) < \deg \lt(g)$,\hspace{1mm}so we have $S(f,g)$ with smaller leading term and proceed like this until get $0$.). This is exactly what statement (2) says.\\
Next we will show (ii) implies (iii).\\
As $I = (G)$,\hspace{1mm}lets take $f \in I$,
\[ f=h_1g_1 + h_2g_2 + \dots + h_ng_n; \hspace{2mm}\text{with}\hspace{1mm} \lt(f) < \max(\lt(h_ig_i)). \]
Our main objective is to show how to replace this combination by another one with smaller $\max(\lt(h_ig_i))$,\hspace{1mm}if still bigger than $\lt(f)$,\hspace{1mm}continue until it becomes $\lt(f)$. Without loss of generality,
\[ \lt(h_1g_1) = \lt(h_2g_2) = \dots  = \lt(h_kg_k) = \max \lt(h_ig_i)  \]
\[ \text{and} \hspace{1mm}\lt(h_pg_p) < \max \lt(h_ig_i)\hspace{2mm}\text{for} \hspace{1mm} p > k.  \]
We will use induction on $k$ to prove our desired result.\\
if $k = 1$,\hspace{1mm}a contradiction as there is no possibility to replace the combination with smaller $\max \lt(h_ig_i)$, so $k \geq 2$.\hspace{1mm}Now
\[ \lt(h_1g_1) = \lt(h_1)\lt(g_1)   \]
\[ \lt(h_2g_2) = \lt(h_2)\lt(g_2).  \]
If $\lt(h_1g_1) = \lt(h_2g_2)$, then one of the following three conditions holds:\\
(a) $\lt(g_1)$ is divisible by $\lt(g_2)$.\\
(b) $\lt(g_1),\lt(g_2)$ have a small common multiple.\\
(c) $\lt(g_1),\lt(g_2)$ have no common divisors.\\
Let us first consider case (a) ,
\[ \lt(g_1) = m\lt(g_2)  \]
\[ \text{then} \hspace{2mm} R_{g_2}(g_1) = g_1 - \frac{\lc(g_1)}{\lc(g_2)}mg_2  \]
\[ \text{now} \hspace{2mm}h_1g_1 + h_2g_2 = h_1\left( R_{g_2}(g_1) + \frac{\lc(g_1)}{\lc(g_2)}mg_2\right) + h_2g_2 \]
 \[                       = h_1R_{g_2}(g_1) + \left(\frac{\lc(g_1)}{\lc(g_2)}mh_1 + h_2\right)g_2  \]
 \[ \text{with} \hspace{3mm} \lt(R_{g_2}(g_1)) < \lt(g_1)  \]
 \[ \text{so} \hspace{2mm} h_1R_{g_2}(g_1) = \sum_{i}\tilde{h_i}g_ih_1  \]
 \[ \lt(g_1) > \lt(R_{g_2}(g_1)) = \max ( \lt(\tilde{h_i}g_i)) \]
 So the condition (a) replaces the combination by another one with smaller $k$.\hspace{1mm}Let us move to condition (b),
 \[ \lt(h_1)\lt(g_1) = \lt(h_2)\lt(g_2)  \]
 suppose that,
 \[ \lt(g_1) = m_1d  \]
 \[ \lt(g_2) = m_2d  \]
 where $m_1,m_2$ have no common factors,
 \[ \lt(h_1)m_1d = \lt(h_2)m_2d    \]
 cancelling $d$, we get,
 \[ \lt(h_1)=em_2  \]
 \[ \lt(h_2)=em_1,\hspace{3mm}\text{ for some}\hspace{1mm} e,  \]
 \[ h_1g_1 + h_2g_2 = \lc(h_1)\lt(h_1)g_1 + ( h_1 - \lc(h_1)\lt(h_1))g_1   \]
 \[ + \lc(h_2)\lt(h_2)g_2 +( h_2 - \lc(h_2)\lt(h_2))g_2  \]
 as\hspace{1mm} $( h_1 - \lc(h_1)\lt(h_1))g_1 +( h_2 - \lc(h_2)\lt(h_2))g_2$ \hspace{1mm} have smaller leading term so we proceed with the remaining terms,\hspace{1mm}i.e. we deal with
 \begin{equation} \lc(h_1)\lt(h_1)g_1 + \lc(h_2)\lt(h_2)g_2 . \end{equation}
 We have
 \[ m_2\lt(g_1) = m_1\lt(g_2)  \]
 so the S-polynomial with respect to small common multiple,
 \[ S = \frac{1}{\lc(g_1)}m_2g_1 - \frac{1}{\lc(g_2)}m_1g_2   \]
 Putting the values of $\lt(h_1)$ and $\lt(h_2)$ in equation (6) we get,
 \[ \lc(h_1)em_2g_1 + \lc(h_2)em_1g_2   \]
 \[ = \lc(h_1)\left(\lc(g_1)S + \frac{\lc(g_1)}{\lc(g_2)}m_1g_2 \right)e + \lc(h_2)em_1g_2 \]
 \[ = \lc(h_1)\lc(g_1)S + \left( \lc(h_1)\frac{\lc(g_1)}{\lc(g_2)}m_1e + \lc(h_2)m_1e \right)g_2  \]
 So we have again replaced the combination with a smaller $k$.\\
 Finally we have condition (c) in hand,
 \[ \lt(h_1)\lt(g_1) = \lt(h_2)\lt(g_2)  \]
 \[ \Rightarrow \lt(h_1) = \lt(g_2)e   \]
 \[ \text{and} \hspace{3mm} \lt(h_2) = \lt(g_1)e, \hspace{2mm}\text{for some}\hspace{1mm}e \]
 Now
 \[ h_1g_1 + h_2g_2  \]
 \[ = \lc(h_1)\lt(h_1)g_1 + \lc(h_2)\lt(h_2)g_2 + \hspace{2mm}\text{lower terms}. \]
 We proceed with 
 \[ \lc(h_1)\lt(h_1)g_1 + \lc(h_2)\lt(h_2)g_2 .  \]
 Replacing the values of $\lt(h_1)$ and $\lt(h_2)$ in the above term we get,
 \begin{equation}
 \lc(h_1)\lt(g_2)eg_1 + \lc(h_2)\lt(g_1)eg_2.
 \end{equation}
 We have 
 \[ \lt(g_2) = \frac{1}{\lc(g_2)}(g_2 - \bar{g_2}),   \]
 replacing this in equation (7) we get
 \[ \frac{\lc(h_1)}{\lc(g_2)}(g_2 - \bar{g_2})eg_1 + \lc(h_2)\lt(g_1)eg_2   \]
 which again makes either leading term or $k$ smaller.\\
 This is how we proceed using induction and prove statement (iii) from (ii).
\end{proof}
\subsection{Reduced Gr\"{o}bner basis}
\begin{definition}
A Gr\"{o}bner basis $G$ of $I$ is \textbf{reduced}[2] if for each $g \in G$,
\begin{itemize}
\item $\lc(g) = 1$.
\item $g - \lt(g)$ is a linear combination of normal monomials.
\end{itemize}
\end{definition}
\begin{theorem}
Let us fix an admissible ordering. Then every $I$ has a unique reduced Gr\"{o}bner basis.
\end{theorem}
\begin{proof}
Let us take some Gr\"{o}bner basis $G \subset I$.\\
First condition of reduced Gr\"{o}bner basis is easy to satisfy as we just divide each $g$ by it's leading co-efficient,\hspace{1mm}i.e. $ g \rightarrow g/\lc(g)$.\\
The reduction and S-polynomial suggests that remaining terms of $g$ is not divisible by the leading term of any terms in $G$ which implies that $g - \lt(g)$ is a linear combination of normal monomials. Now we will prove the uniqueness.\\
Let $\{ f_1,f_2,\dots,f_s \}$ and $\{g_1,g_2,\dots,g_s \}$ be two reduced and ordered Gr\"{o}bner bases so that $\lt(f_i)=\lt(g_i)$ for each $i$.\hspace{1mm}Consider $f_i - g_i \in I$,\hspace{1mm}if it's not $0$,\hspace{1mm}then its leading term must be a term that appeared either in $f_i$ or in $g_i$.\hspace{1mm}In either case,\hspace{1mm}this contradicts the fact that the bases being reduced,\hspace{1mm}so in fact we get our required $f_i = g_i$.
\end{proof}

\subsection{Terminology for Gr\"{o}bner bases in non-commutative algebras}
We will deal with non-commutative polynomial algebra $\mathbb{K}\langle x_1,x_2,\dots,x_n \rangle$. For definitions of admissible ordering in non-commutative cases and more terminology we refer [13][16]. 
\begin{definition}
We have two polynomials $f,g \in \mathbb{K}\langle x_1,x_2,\dots,x_n \rangle$ , with there exists monomials $m_1,m_2$ such that,
\[  \lt(f) = m_1\lt(g)m_2 \]
then
\begin{equation} R_g(f) = f - \frac{\lc(f)}{\lc(g)}m_1gm_2   \end{equation}
is called \textbf{reduction} of $f$ with respect to $g$.
\end{definition}
\begin{definition}
For $f, g \in \mathbb{K}\langle x_1,x_2,\dots,x_n \rangle$ and for any small common multiple of $\lt(f),\lt(g)$;\hspace{1mm}there exists two monomials $m_1,m_2$ with
\[   \lt(f)m_2 = m_1\lt(g) \hspace{2mm}with \hspace{1mm} \deg m_1 < \deg \lt(f) \]
then
\begin{equation}
S(f,g) = \frac{1}{\lc(f)}fm_2 - \frac{1}{\lc(g)}m_1g
\end{equation}
is called the \textbf{S-polynomial} with respect to small common multiples.
\end{definition}
\begin{example}
In 3.6, while computing Gr\"{o}bner basis in non-commutative polynomial algebra we give examples to show how reduction and S-polynomial work.
\end{example}
\begin{lemma}
Let $ I \subset \mathbb{K}\langle x_1,x_2,\dots,x_n \rangle$ be a two-sided ideal. Then $\lt(I)$ which is linear span of $\lt(f)$ with $f \in I$ is also a two-sided ideal.
\end{lemma}
\begin{definition}
A subset $ G \subset I$ is a \textbf{Gr\"{o}bner basis} of $I$ if for every $f \in I$,\hspace{1mm}we have $\lt(f) = m_1\lt(g)m_2$ for some $g \in G$.
\end{definition}
\begin{lemma}[Diamond lemma for non-commutative case]
$G \in I$ forms a Gr\"{o}bner basis of $I$ if and only if for each $g_1,g_2 \in G$;\hspace{1mm}$R_{g_2}(g_1)$(if defined) can be reduced to $0$ modulo $G$. Also for each $g_1,g_2 \in g$ and each small common multiple of $\lt(g_1),\lt(g_2)$;\hspace{1mm}the corresponding S-polynomial can be reduced to $0$ modulo $G$.
\end{lemma}
\begin{proof}
The proof is similar as the proof of diamond lemma in case of commutative algebra. Here we need some modification while dealing with reduction terms and continue with the same strategy of the proof given in lemma 3.3.4, so we skip this proof. For a proof, we direct to [13].
\end{proof}
\subsection{Computation of Gr\"{o}bner bases}
 In this section we show how to compute Gr\"{o}bner basis for an ideal $I$ of a polynomial algebra. At first we present the Buchberger's algorithm of computation of Gr\"{o}bner basis in a commutative algebra. The algorithm is similar in case of non-commutative algebra. 
 
 \subsubsection{Buchberger's algorithm}
 We start with an ideal $I$ generated by a set $G$. The Buchberger's algorithm,\hspace{1mm}which is a simple consequence of lemma 3.3.3,\hspace{1mm}is the following:\\
 Step 1: If the leading term of any element $u$ of $G$ occurs inside the leading term of another element $v$ of $G$,\hspace{1mm}then we reduce $v$ by subtracting off the required multiple of $u$. In general we will perform the reduction mentioned in definition 3.3.3.\\
 Step 2: For each pair of distinct elements of $G$ we compute the S-polynomial (mentioned in definition 3.3.4) and a remainder of it.\\
 Step 3: If the remainder can be reduced further then we will follow step 1 or we will add that term in our set $G$. If all S-polynomials reduce to $0$,\hspace{1mm}then the algorithm ends and $G$ is the Gr\"{o}bner basis of $I$. If not then we will continue further with our 3 steps.
 \begin{remark}
 For commutative cases the algorithm ends in a finite number of stages. However for a non-commutative case there is no guarantee of the termination of the algorithm after a finite number of stages. In that case we start adding all elements which can't be reduced further in our set $G$ and in most cases we have seen a combinatorial interpretation for our terms in $G$. 
\end{remark}
 \subsubsection{Commutative example}
\begin{example}
 Let us take $\mathbb{K}[x_1,x_2]$ as our commutative polynomial algebra in two variables $x_1,x_2$. Suppose there are two polynomials
 \[  h_1(x_1,x_2) = x_1^2 + x_2^2  \]
 \[  h_2(x_1,x_2) = x_1^3 + x_2^3 \]
 belonging to our polynomial algebra $\mathbb{K}[x_1,x_2]$. We compute the Gr\"{o}bner basis for $I=(h_1,h_2) \subset \mathbb{K}[x_1,x_2]$.\\
 Let us fix an admissible ordering. Usually we take \textbf{DEGLEX} ordering. So here we consider $x_1 > x_2$. So we get $\lt(h_1)=x_1^2$ and $\lt(h_2)=x_1^3$. So initially our set is $G=\{ h_1,h_2 \}$. But we see that $h_2$ can be reduced further. So we have $x_1^3= x_1\cdot x_1^2$,
 \[ R_{h_1}(h_2) = (x_1^3 +x_2^3) - x_1(x_1^2 +x_2^2)  \]
 \[  = x_2^3 -x_1x_2^2.     \]
So we have obtained a new term $x_2^3 - x_1x_2^2$ which cannot be reduced further,\hspace{1mm}we add this to our set $G$ which is now $\{ h_1,R_{h_1}(h_2) \}$. We call $R_{h_1}(h_2)$ as $h_3$. We see that the leading term of $h_3$ is $x_1x_2^2$. We have also found that $ x_1\cdot x_1x_2^2 = x_1^2 \cdot x_2^2$. So we will compute the S-polynomial between $h_1,h_3$.
\[ S(h_1,h_3) = -x_1(x_2^3 - x_1x_2^2) - (x_1^2 + x_2^2)x_2^2  \]
\[  = -x_1x_2^3 - x_2^4.   \]
The term $-x_1x_2^3 - x_2^4$ has $x_1x_2^3$ as the leading term which can be reduced further through $\lt(h_3)$. We get $x_1x_2^3 = (x_1x_2^2)\cdot x_2$. hence the reduction yields
\[ -x_1x_2^3 - x_2^4 -(x_2^3 -x_1x_2^2)x_2   \]
\[ = -2x_2^4    \]
which cannot be reduced further and also one cannot compute more S-polynomial. Hence we add $-2x_2^4$ in our set $G$ and the final set $G$ is our Gr\"{o}bner basis for $I$,\hspace{1mm}the set is precisely as follows
\[ \{ x_1^2+x_2^2,\hspace{1mm}x_2^3 -x_1x_2^2,\hspace{1mm}-2x_2^4 \}.   \]
\end{example}
\begin{remark}
The reduced Gr\"{o}bner basis of $I$ of our previous example is given by $\{ x_1^2+x_2^2,\hspace{1mm}x_1x_2^2-x_2^3,\hspace{1mm}x_2^4 \}$. It is not very difficult to obtain this reduced Gr\"{o}bner basis from our computed Gr\"{o}bner basis. If we recall the definition of reduced Gr\"{o}bner basis we will see that all leading co-efficients of the reduced basis should be $1$. So we just divide terms $-2x_2^4, x_2^3 - x_1x_2^2$ of $\{ x_1^2+x_2^2,\hspace{1mm}x_2^3 -x_1x_2^2,\hspace{1mm}-2x_2^4 \}$ by $-2$ and $-1$ respectively to obtain $\{ x_1^2+x_2^2,\hspace{1mm}x_1x_2^2-x_2^3,\hspace{1mm}x_2^4 \}$. Indeed $x_2^2$ and $x_2^3$ are normal monomials. So we have obtained the reduced Gr\"{o}bner basis of $I$.  
\end{remark}

\subsubsection{Non-commutative examples}
\begin{example}
Let us consider $\mathbb{K}\langle x,y \rangle$ as non-commutative polynomial algebra in two variables. We are going to compute Gr\"{o}bner basis for $I= (x^2 -xy)$. So we begin with our set $G$ as $\{ x^2-xy \}$ whose leading term is $x^2$ (we consider the \textbf{DEGLEX} for this case i.e. here $x>y$). We need to compute S-polynomial between polynomial $f_1= x^2-xy$ and $f_2=x$. We have $x^2 \cdot x = x \cdot x^2$. Then
\[ S(f_1,f_2) = (x^2-xy)x - x(x^2-xy)   \]
\[ = xxy - xyx.   \]
We see that $xxy-xyx$ whose leading term is $xxy$ can be reduced further. So we have $xxy = x^2 \cdot y$. So the reduction is
\[ (xxy-xyx) - (x^2-xy)y  \]
\[ = xxy - xyx - x^2y + xyy    \]
\[ = xyy - xyx   \]
which cannot be reduced further. so we include $xyy-xyx$ in our set $G$ which is now $\{ x^2-xy,xyy-xyx \}$. We see that $x^2 \cdot yx = x \cdot xyx$,\hspace{1mm}hence we compute S-polynomial between those elements of our set $G$ and we get
\[ (x^2-xy)yx + x(xyy - xyx)  \]
\[ = x^2yx - xyyx + xxyy - xxyx   \]
\[ = xxyy- xyyx    \]
this element with leading term $xxyy$ can be reduced further and we get $xxyy = x^2 \cdot yy$. So the reduction gives
\[ (xxyy - xyyx) - (x^2 -xy)yy   \]
\[ = xxyy - xyyx - x^2yy + xyyy  \]
\[ = xyyy - xyyx \]
which cannot be reduced further and we add this term in our existing set $G$ and obtain $\{ x^2-xy,xyy-xyx,xyyy-xyyx \}$.\\
Now we claim that the Gr\"{o}bner basis for $I= (x^2-xy)$ is given by
\[ \{ x^2 - xy \} \cup \hspace{1mm} \bigcup_{i=2}^{\infty} \{ xy^i - xy^{i-1}x \}.  \]
This is indeed very easy to prove. We will construct elements of $\bigcup_{i=2}^{\infty} \{ xy^i - xy^{i-1}x \}$ by method of induction. We have already shown for $i=2,3$.\\ Suppose up to $i=k$ steps the set $G$ is 
\[ \{ x^2 - xy \} \cup \hspace{1mm} \bigcup_{i=2}^{k} \{ xy^i - xy^{i-1}x \}.  \]
But none of $\{xy^i - xy^{i-1}x \}$ for $2 \leq i \leq k$ have S-polynomial between them,\hspace{1mm}so $\lt(xy^k-xy^{k-1}x)= xy^{k-1}x$ and we have $x^2 \cdot y^{k-1}x = x \cdot xy^{k-1}x$ which leads us to compute the s-polynomial between $x^2-xy$ and $xy^k-xy^{k-1}x$ and we get
\[ (x^2-xy)y^{k-1}x + x(xy^k - xy^{k-1}x)   \]
\[ = x^2y^{k-1}x - xy^kx + x^2y^k - x^2y^{k-1}x    \]
\[ = x^2y^k - xy^kx    \]
which has the leading term $x^2y^k$ which can be reduced further. We get $x^2y^k = x^2 \cdot y^k$. Hence the reduction is 
\[ (x^2y^k - xy^kx) - (x^2 -xy)y^k  \]
\[ = xy^{k+1} - xy^kx   \]
which cannot be reduced further and so we add the term $xy^{k+1}- xy^kx$ in our existing set $G$. Hence the induction is complete. Therefore
\[ \{ x^2 - xy \} \cup \hspace{1mm} \bigcup_{i=2}^{\infty} \{ xy^i - xy^{i-1}x \}.  \]
is our Gr\"{o}bner basis.
\end{example}
 
\begin{example}
Let us give another example to compute Gr\"{o}bner basis for a non-commutative polynomial algebra. Consider $\mathbb{K}\langle x,y,z \rangle$ as our non-commutative polynomial algebra. We are going to compute Gr\"{o}bner basis for $I =(x^2, xy-zx)$. We will consider \textbf{DEGLEX} ordering,\hspace{1mm}so we consider $x > y > z$. The leading terms of $x^2$ and $xy-zx$ are $x^2$ and $xy$ respectively and both of them cannot be reduced further. So our initial set $G$ is $\{ x^2, xy-zx \}$. However we can compute S-polynomial between them based on $x^2 \cdot y = x \cdot xy$ and obtain
\[ x^2y - x(xy-zx)  \]
\[ = xzx     \]
which cannot be reduced further and so include it in our set $G$. So now our set $G$ is $\{ x^2, xy-zx, xzx \}$. We can compute S-polynomial between $xy-zx$ and $xzx$. We have $ xzx \cdot y = xz \cdot xy$ and so our S-polynomial is 
\[ (xzx)y - xz(xy-zx)  \]
\[ = xzzx  \]
which cannot be reduced further so we add it to $G$ and get $\{ x^2, xy-zx, xzx, xzzx \}$. This is how we proceed and claim that Gr\"{o}bner basis for $I$ is given by
\[ \{x^2, xy-zx \} \cup \hspace{1mm} \bigcup_{i=1}^{\infty}\{ xz^ix \}.  \]
We can prove this using the similar argument we have used in the previous example by method of induction.
\end{example}
\newpage 
\section{Anick's Resolution and Its Computations}

\subsection{Chain and their examples}
 Let us begin the discussion of Anick's resolution[3] with some background materials:

\subsubsection{Calculating Hilbert series through chains}
\begin{lemma}
If $A$ is an algebra with space spanned by normal words $N$, then $H_A = H_N$.
\end{lemma}
\begin{proof}
$A$ is the free associative algebra, modulo some ideal $I$. The free associative algebra itself decomposes into $N \oplus I$, so it's Hilbert series can be written as $H_N + H_I$. Taking modulo the ideal $I$ to the free algebra turns the free algebra into our algebra $A$, and eliminates the $H_I$ term. So we obtain $H_A = H_N$ as required.
\end{proof}
Let $A = \mathbb{K}\langle X \hspace{1mm}|\hspace{1mm}R \rangle$ be an algebra where $X$ is the set of generators and $R$ is the set of relations. Let the set of normal words be $N$, and say we have a reduced Gr\"{o}bner basis $G$. Let $F$ be the set of leading terms of $G$ - we call the elements of $F$ \textit{leading terms}. A word is normal with respect to $G$ if and only if it does not contain any of the elements of $F$ as a subword. Thus a word $s$ is normal with respect to $G$ if and only if it's normal with respect to $F$. Therefore the algebra $\hat{A} = \mathbb{K}\langle X \hspace{1mm}|\hspace{1mm}F \rangle$ has the same normal words as $A$. And so $H_A = H_{\hat{A}}$. The advantage of dealing with $F$ instead of $G$ is that $F$ consists of monomials only, which makes it much easier to deal with.

\begin{definition}[Chains]
We will define chain[3] inductively: A $(-1)$ - chain is the empty word and is its own tail. The $0$-chains are the elements of the generating set $X$, and are also their own tails. A $n$-chain is a word $f$ of the form $gt$, with some conditions on $g$ and $t$. Firstly, $g$ must be a $(n-1)$-chain and $t$ is a normal word. Secondly, if $r$ is the tail of $g$ then $\deg_F rt = 1$; that is, the word $rt$ contains exactly one element of $F$ as a subword. This subword must occur at the end of $rt$. The tail of $gt$ is defined to be $t$.\\ 
We denote the space spanned by $n$-chains by $C_n$.
\end{definition}
\begin{example}
Let $F = \{ x^3 \}$. The unique 1-chain is $x^3= x \cdot x^2$ and its tail is $x^2$. Then the unique 2-chain is $\overline{xxxx} = x^3 \cdot x$. The word $x^3 \cdot x^2$ is not a 2-chain, since $\deg_F x^2x^2 = 2$. The unique 3-chain is the word $x^6=x^4x^2$. The word $x^5 = x^4x$ is not a 3-chain because $\deg_F x\cdot x =0$, regardless of the fact that it can be represented ($xxxxx$) as a link of three leading terms $\overline{xxx}$ (the fact is that the first one intersects with the last one).\\ 
In general the $n$-chain is given by $x^{3m+1}$ if $n =2m$ and $x^{3m}$ if $n=2m+1$. We see that in this case for every $n$ there exists only one $n$-chain.
\end{example}
\begin{theorem}
Let
\[ \begin{tikzcd}
  \dots \arrow{r}{d_{n+1}} & A_n \arrow{r}{d_n} & A_{n-1} \arrow{r}{d_{n-1}} & \dots \arrow{r}{d_{k+1}} & A_k \arrow{r}{d_k} & K \arrow{r} & 0 
  \end{tikzcd}\]
  be an exact sequence of graded spaces (i.e $\ker d_i = \Ima d_{i+1}$). Then
  \[  \sum_{i=k}^{\infty} (-1)^i H_{A_i} = (-1)^k     \]
  if the sum is well defined.
\end{theorem}
\begin{proof}
We know that if $f: V \rightarrow W$ is any linear transformation such that $f(V_n) \subset f(W_n)$, then $H_V = H_{\ker f} + H_{\Ima f}$. So
\[  H_{A_i} = H_{\ker {d_i}} + H_{\Ima {d_i}} = H_{\Ima {d_{i+1}}} + H_{\Ima {d_i}},  \]
where the second equality comes from exactness. Thus taking the alternating sum of these equalities, we get our required result.
\end{proof}
\begin{theorem}[Hilbert series from chains]
Let $A$ is an algebra and $C_n$ be the linear span of $n$-chains. We point out that
\[ H_{C_{-1}} = H_K = 1;\hspace{4mm} H_{C_{0}} = H_X; \hspace{4mm} H_{C_1}= H_F.   \]
Then
\[  H_A = ( H_{C_{-1}} - H_{C_0} + H_{C_1} - H_{C_2} + \dots )^{-1}.    \]
\end{theorem}
\begin{proof}
Let $\hat{A} = \mathbb{K}\langle X \hspace{1mm}|\hspace{1mm}F \rangle$ where $F$ is our set of leading terms. Let $d_n: C_n \otimes \hat{A} \rightarrow C_{n-1} \otimes \hat{A}$ be defined by $d_n(gt \otimes a) = g \otimes ta$. Then $d_n(d_{n+1}(gt \otimes a)) = d_n( g \otimes ta) = g' \otimes t'ta$, where $g \in C_n$, $g' \in C_{n-1}$. By the definition of $n$-chains it follows that there must be a leading term in $t'ta$, so this is $0$.\\ 
Also $d_n$ is a surjection, as for every $ g \otimes a \in C_{n-1} \otimes \hat{A}$, we can decompose $a$ as $a_1a_2$ such that $ga_1 \otimes a_2 \in C_n \otimes \hat{A}$. So we have an exact sequence:
\[ \begin{tikzcd}
 \dots \arrow{r}{d_{n+1}} & C_n \otimes \hat{A} \arrow{r}{d_n} & C_{n-1} \otimes \hat{A} \arrow{r}{d_{n-1}} & \dots \arrow{r}{d_0} & C_{-1} \otimes \hat{A} \arrow{r} & K \arrow{r} & 0
\end{tikzcd}\]
So by theorem 4.1.1
\[  \sum_{i=-1}^{\infty} (-1)^i H_{C_i \otimes \hat{A}} = -1.   \]
Now $H_{C_i \otimes \hat{A}} = H_{C_i}H_{\hat{A}} = H_{C_i}H_A$. Hence
\[  H_A = ( H_{C_{-1}} - H_{C_0} + H_{C_1} - H_{C_2} + \dots )^{-1}.    \]
\end{proof}
\begin{example}
Let $A = \mathbb{K}\langle x,y \hspace{1mm}|\hspace{1mm}x^2+y^2\rangle$. The leading terms of $A$ are $x^2$ and $xy^2$ (the Gr\"{o}bner basis of $A$ is $\{ x^2+y^2, xy^2 - y^2x \}$). We have found that the $n$-chains are given by $x^ny^2$ and $x^{n+1}$, for $n > 0$. So we have,
\[ H_A^{-1} = ( 1- 2t +(t^2+t^3) -(t^3+t^4) + \dots ) = 1-2t+t^2.    \]
Hence
\[  H_A = \frac{1}{1-2t+t^2}.   \]
\end{example}
\subsection{Anick's resolution}
In the proof of theorem 4.1.2, we have constructed a resolution for the normal words of our algebra. This allows us to find the Hilbert series for our normal words and hence for our original algebra. However this resolution only depends on normal words, so the rest of the structure of our algebra is lost. If we want more properties of the algebra then this leads to a construction of a resolution of the algebra itself. Along with this fact, that sequence and isomorphism of spaces with graduations $C_n \otimes \hat{A}$ and $C_n \otimes A$ leads us to think about the existence of a corresponding free resolution which was constructed by Anick in 1986 [15].\\ 
We will now state the resolution in form of a theorem. Though the resolution itself looks complicated but it will be clear with computations in various examples.
\begin{remark}
Let $f$ be a word. By $\overline{f}$ we denote the reduction of the word to a normal form. If a word $f$ cannot be reduced to a normal word and contain a leading term as a subword inside it then we write $\overline{f} = 0$.
\end{remark}
\begin{theorem}[Anick's resolution]
Let $A$ be a graded algebra with augmentation (i.e. there exists an augmentation map $\epsilon : A \rightarrow K$), and let $C_n$ be the set of $n$-chains. We have a resolution of $A$ of the following form:
\[ \begin{tikzcd}
\dots \arrow{r}{d_{n+1}} & C_n \otimes A \arrow{r}{d_n} & C_{n-1} \otimes A \arrow{r}{d_{n-1}} & C_{n-2} \otimes A \arrow{r}{d_{n-2}} & \dots \\
 \dots \arrow{r}{d_2} & C_1 \otimes A \arrow{r}{d_1} & C_0 \otimes A \arrow{r}{d_0} & C_{-1} \otimes A \arrow{r}{\epsilon} & K \arrow{r} & 0
\end{tikzcd}\]
with splitting inverse maps $i_n : \ker d_{n-1} \rightarrow C_n \otimes A$ (which, unlike $d_n$ need not to be homomorphisms of modules). Where:
\begin{itemize}
\item $d_0( x \otimes 1) = 1 \otimes x$.
\item $i_{-1}(1) = 1 \otimes 1$.
\item $i_0( 1 \otimes x_{i_1}x_{i_2}\dots x_{i_n}) = x_{i_1} \otimes x_{i_2}\dots x_{i_n}$.
\item $d_{n+1}(gt \otimes 1) = g \otimes t - i_nd_n( g \otimes t)$ for all ($n+1$)-chains $gt$, with tail $t$.
\item $i_n(u) = \alpha g \otimes c + i_n( u - \alpha d_n(g \otimes c))$ for all $u \in \ker d_{n-1}$ with leading term $f \otimes s$, where $\overline{fs} = \overline{gc}$, and $f$ is a ($n-1$)-chain and $g$ is a $n$-chain. The bar over $fs$ and $gc$ denotes reduction of them to normal forms.
\end{itemize}
\end{theorem}
\begin{remark}
We define $d_n$ on free generators $f \otimes 1$ where $f$ is a $n$-chain and equals to $gt$.
\end{remark}
\begin{proof}
We will construct ($d_{n+1}, i_n$) by induction. Note that we only need to define $i_n$ on the kernel of $d_{n-1}$. Our induction hypothesis follows in 2 parts:
\begin{enumerate}
\item $d_{n-1}(d_n(u)) = 0$ for all $u$.
\item $d_{n-1}(i_{n-1}(u)) = 0$ for all $u \in \ker d_{n-2}$.
\end{enumerate}
We proceed by induction on the order of leading terms. Now we note that because of isomorphism from $C_nN$ to $C_n \otimes N$, the following partial order is defined:
\[  f \otimes t < g \otimes s \iff ft < gs.   \]
We have
\[ d_n(i_n(u)) = d_n( g \otimes c + i_n( u - \alpha d_n(g \otimes c)))   \]
\[  \hspace{10mm} = d_n(\alpha g \otimes c) + d_n(i_n(u - \alpha d_n( g \otimes c))). \]
Now we can verify that $u - \alpha d_n( g \otimes c)$ has a smaller leading term. This is not hard to check, but is slightly lengthy, and so would clutter this already complicated proof.
\[ d_n(i_n(u)) = d_n(\alpha g \otimes c) + d_n(i_n(u - \alpha d_n(g \otimes c)))  \]
\[ \hspace{30mm} = d_n(\alpha g \otimes c) + u - \alpha d_n(g \otimes c)   \]
\[  = u.  \]
Also
\[ d_n(d_{n+1}(gt \otimes 1) = d_n( g \otimes t - i_nd_n( g \otimes t))  \]
\[  = d_n(g \otimes t) - d_n(i_n(d_n(g \otimes t)))  \]
\[  = d_n(g \otimes t) - d_n( g \otimes t)  \]
\[  = 0 . \]
Here we need the fact that $d_n(u)$ is in $\ker d_{n-1}$, and we use the result we proved a moment ago.\\ 
Finally, $d_0$,$d_1$,$i_{-1}$ and $i_0$ satisfy the induction hypotheses, so by induction the two hypotheses hold true for all $n$. Together, they imply that the constructed sequence is a resolution.
\end{proof}

\subsection{Computation of Anick's resolution}
\begin{example}
We will compute Anick's resolution for algebra\\ 
\[A = \mathbb{K}\langle x,y \hspace{1mm}|\hspace{1mm}x^2=xy \rangle, \]
given DEGLEX ordering with $x>y$.\\
The Gr\"{o}bner basis of algebra $A$ is given by $\{ xy^nx - xy^{n+1}\hspace{2mm} \text{for all} \hspace{2mm} n \geq 0\}$. Hence the set of leading terms of the Gr\"{o}bner basis is given by $\{ xy^nx \hspace{2mm} \text{for all}\hspace{2mm} n \geq 0 \}$, which is our $F$ or $C_1$.
\begin{lemma}
The space spanned by $n$-chains, $C_n$ are of the form given in a tabular form:\\ 
\begin{center}
\begin{tabular}{|c|c|}
\hline
Chain   &  Form \\
$C_1$   &  $xy^{n_1}x$;\hspace{2mm}for all  $n_1 \geq 0$ \\
$C_2$   &  $xy^{n_1}xy^{n_2}x$;\hspace{2mm}for all  $n_1,n_2 \geq 0$ \\
$C_3$   &  $xy^{n_1}xy^{n_2}xy^{n_3}x$;\hspace{2mm}for all  $n_1,n_2,n_3 \geq 0$ \\
\vdots  &  \vdots \\
\vdots  &  \vdots \\
$C_k$   &  $xy^{n_1}xy^{n_2}x\dots y^{n_k}x$;\hspace{2mm}for all  $n_1,n_2,\dots,n_k \geq 0$ \\
\vdots  &  \vdots \\
\hline
\end{tabular}
\end{center}
\end{lemma}
Let us begin with the construction of maps.\\
As defined $d_0( x \otimes 1) = 1 \otimes x$ and $d_0( y \otimes 1) = 1 \otimes y$.\\ 
Now $d_1: C_1 \otimes A \rightarrow C_0 \otimes A$ given by
\begin{align}
d_1( xy^{n_1}x \otimes 1) & = x \otimes y^{n_1}x - i_0d_0( x \otimes y^{n_1}x) \nonumber \\ 
                          & = x \otimes y^{n_1}x - i_0( 1 \otimes \overline{xy^{n_1}x}) \nonumber \\ 
                          & = x \otimes y^{n_1}x - i_0( 1 \otimes xy^{n_1+1}) \nonumber \\
                          & = x \otimes y^{n_1}x - x \otimes y^{n_1 +1}
\end{align}
$d_2 : C_2 \otimes A \rightarrow C_1 \otimes A$ is given by
\begin{align}
d_2(xy^{n_1}xy^{n_2}x \otimes 1) & = xy^{n_1}x \otimes y^{n_2}x - i_1d_1(xy^{n_1}x \otimes y^{n_2}x) \nonumber \\
                                 & = xy^{n_1}x \otimes y^{n_2}x - i_1( x \otimes \overline{y^{n_1}xy^{n_2}x} - x \otimes \overline{y^{n_1+n_2+1}}x) \nonumber \\
                                 & = xy^{n_1}x \otimes y^{n_2}x - i_1( x \otimes y^{n_1}xy^{n_2+1} - x \otimes y^{n_1+n_2+1}x)  
\end{align}
%\newpage
We now deal with $i_1( x \otimes y^{n_1}xy^{n_2+1} - x \otimes y^{n_1+n_2+1}x)$. We see that $x \otimes y^{n_1}xy^{n_2+1}$ is our leading term in $( x \otimes y^{n_1}xy^{n_2+1} - x \otimes y^{n_1+n_2+1}x)$, and $x$ is a $0$-chain. We require a $1$-chain for our formula for $i_1$ and indeed it is $xy^{n_1}x$, hence proceeding we get,
\[ i_1( x \otimes y^{n_1}xy^{n_2+1} - x \otimes y^{n_1+n_2+1}x)  =\] 
\[ xy^{n_1}x \otimes y^{n_2+1}  + \] 
            \[     i_1(x \otimes y^{n_1}xy^{n_2+1} - x \otimes y^{n_1+n_2+1}x - d_1(xy^{n_1}x \otimes y^{n_2+1})) \] 
 \begin{equation}
 = xy^{n_1}x \otimes y^{n_2+1} + i_1( - x \otimes y^{n_1+n_2+1}x + x \otimes y^{n_1+n_2+2})   
\end{equation}
(Now we observe that $x \otimes y^{n_1+n_2+1}x$ is the leading term in $( - x \otimes y^{n_1+n_2+1}x + x \otimes y^{n_1+n_2+2})$ so we will deal with that term and as discussed earlier $x$ is $0$-chain so $xy^{n_1+n_2+1}$ is our $1$-chain and proceeding with the formula for $i_1$ we have obtained). Continuing with equation (12), 
\[  xy^{n_1}x \otimes y^{n_2+1} - xy^{n_1+n_2+1}x \otimes 1 +     \]
\[   \hspace{15mm} i_1( - x \otimes y^{n_1+n_2+1}x + x \otimes y^{n_1+n_2+2} + d_1( xy^{n_1+n_2+1}x \otimes 1))   \]
Now
\[ d_1( xy^{n_1+n_2+1}x \otimes 1) =  x \otimes y^{n_1+n_2+1}x - x \otimes y^{n_1+n_2+2}    \]
(obtained from the formula for $d_1$ in equation (10)), hence putting it in our main equation we get 
\[ i_1( - x \otimes y^{n_1+n_2+1}x + x \otimes y^{n_1+n_2+2} + d_1( xy^{n_1+n_2+1}x \otimes 1)) = 0.  \]
Hence 
\[ i_1( x \otimes y^{n_1}xy^{n_2+1} - x \otimes y^{n_1+n_2+1}x) = xy^{n_1}x \otimes y^{n_2+1} - xy^{n_1+n_2+1}x \otimes 1   \]
putting it in equation (11) we get,
\begin{equation}
d_2(xy^{n_1}xy^{n_2}x \otimes 1) \\
 = xy^{n_1}x \otimes y^{n_2}x - xy^{n_1}x \otimes y^{n_2+1} + xy^{n_1+n_2+1}x \otimes 1
\end{equation} 
We now work out the formula for  
\[d_3 : C_3 \otimes A \rightarrow C_2 \otimes A.\]
\[ d_3(xy^{n_1}xy^{n_2}xy^{n_3}x \otimes 1) \]
\begin{equation}
= xy^{n_1}xy^{n_2}x \otimes y^{n_3}x - i_2d_2(xy^{n_1}xy^{n_2}x \otimes y^{n_3}x)   
\end{equation}
We will deal with $i_2d_2(xy^{n_1}xy^{n_2}x \otimes y^{n_3}x)$. Now
\[ i_2d_2(xy^{n_1}xy^{n_2}x \otimes y^{n_3}x) =  \]
\[  i_2( xy^{n_1}x \otimes \overline{y^{n_2}xy^{n_3}x} - xy^{n_1}x \otimes \overline{y^{n_2+n_3+1}x} + xy^{n_1+n_2+1}x \otimes y^{n_3}x)   \]
\begin{equation} = i_2( xy^{n_1}x \otimes y^{n_2}xy^{n_3+1} - xy^{n_1}x \otimes y^{n_2+n_3+1}x + xy^{n_1+n_2+1}x \otimes y^{n_3}x)   \end{equation}
We see that $xy^{n_1}xy^{n_2}xy^{n_3+1}$ is the leading term and $xy^{n_1}x$ is a $1$-chain, so $xy^{n_1}xy^{n_2}x$ will be a $2$-chain satisfying our requirement, hence proceeding with equation (15) we get,
\[  xy^{n_1}xy^{n_2}x \otimes y^{n_3+1} +   \]
\[ i_2( xy^{n_1}x \otimes y^{n_2}xy^{n_3+1} - xy^{n_1}x \otimes y^{n_2+n_3+1}x + xy^{n_1+n_2+1}x \otimes y^{n_3}x \]
\[- d_2(xy^{n_1}xy^{n_2}x \otimes y^{n_3+1}))    \]
Now 
\[ d_2(xy^{n_1}xy^{n_2}x \otimes y^{n_3+1}) = xy^{n_1}x \otimes y^{n_2}xy^{n_3+1} - xy^{n_1}x \otimes y^{n_2+n_3+2}   \]
\[ + xy^{n_1+n_2+1}x \otimes y^{n_3+1}    \]
putting the value of $d_2$ back in $i_2$ and cancelling the term $xy^{n_1}x \otimes y^{n_2}xy^{n_3+1}$ we get,
\[ i_2( - xy^{n_1}x \otimes y^{n_2+n_3+1}x + xy^{n_1+n_2+1}x \otimes y^{n_3}x     \]
\begin{equation}  + xy^{n_1}x \otimes y^{n_2+n_3+2} - xy^{n_1+n_2+1}x \otimes y^{n_3+1})  \end{equation}
So $xy^{n_1}xy^{n_2+n_3+1}x$ from $xy^{n_1}x \otimes y^{n_2+n_3+1}x$ is our leading term and so we proceed with equation (16) and get,
\[  -xy^{n_1}xy^{n_2+n_3+1}x \otimes 1   \]
\[ + i_2(- xy^{n_1}x \otimes y^{n_2+n_3+1}x + xy^{n_1+n_2+1}x \otimes y^{n_3}x   \]
\[  + xy^{n_1}x \otimes y^{n_2+n_3+2} - xy^{n_1+n_2+1}x \otimes y^{n_3+1} + d_2(xy^{n_1}xy^{n_2+n_3+1}x \otimes 1 )) \]
We have
\[  d_2(xy^{n_1}xy^{n_2+n_3+1}x \otimes 1 ) = xy^{n_1}x \otimes y^{n_2+n_3+1}x - xy^{n_1}x \otimes y^{n_2+n_3+2}  \]
\[  + xy^{n_1+n_2+n_3+2}x \otimes 1   \]
so putting the value of $d_2$ we get terms $xy^{n_1}x \otimes y^{n_2+n_3+1}x$ and $xy^{n_1}x \otimes y^{n_2+n_3+2}$ get cancelled and we have, 
\[ i_2( xy^{n_1+n_2+1}x \otimes y^{n_3}x - xy^{n_1+n_2+1}x \otimes y^{n_3+1} + xy^{n_1+n_2+n_3+2}x \otimes 1)  \]
\[ = xy^{n_1+n_2+1}xy^{n_3}x \otimes 1      \]
\[ + i_2(xy^{n_1+n_2+1}x \otimes y^{n_3}x - xy^{n_1+n_2+1}x \otimes y^{n_3+1} + xy^{n_1+n_2+n_3+2}x \otimes 1   \]
\[  - d_2(xy^{n_1+n_2+1}xy^{n_3}x \otimes 1 ))   \]
Now
\[ d_2(xy^{n_1+n_2+1}xy^{n_3}x \otimes 1 ) = -xy^{n_1+n_2+1}x \otimes y^{n_3}x + xy^{n_1+n_2+1}x \otimes y^{n_3+1}    \]
\[  - xy^{n_1+n_2+n_3+2}x \otimes 1   \]
so all terms get cancelled and we left with $xy^{n_1+n_2+1}xy^{n_3}x \otimes 1 $. Putting it in equation (14) we get,
\[ d_3(xy^{n_1}xy^{n_2}xy^{n_3}x \otimes 1) =    \]
\[ xy^{n_1}xy^{n_2}x \otimes y^{n_3}x - xy^{n_1}xy^{n_2}x \otimes y^{n_3+1}      \]
\begin{equation}
 + xy^{n_1}xy^{n_2+n_3+1}x \otimes 1 - xy^{n_1+n_2+1}xy^{n_3}x \otimes 1
\end{equation}
\begin{proposition}
\[ d_k : C_k \otimes A \rightarrow C_{k-1} \otimes A    \]
for $k \geq 2$ is given by
\[ d_k(xy^{n_1}xy^{n_2}x \dots y^{n_k}x \otimes 1) =   \]
\[ xy^{n_1}xy^{n_2}x \dots y^{n_{k-1}}x \otimes y^{n_k}x  -  xy^{n_1}x \dots y^{n_{k-1}}x \otimes y^{n_k+1}    \]
\[ xy^{n_1}xy^{n_2}x \dots y^{n_{k-1}+n_k+1}x \otimes 1 - xy^{n_1}xy^{n_2}x \dots y^{n_{k-2}+n_{k-1}+1}xy^{n_k}x \otimes 1   \]
\begin{equation}
+  \dots + (-1)^{k-1} xy^{n_1+n_2+1}xy^{n_3}x \dots y^{n_k}x \otimes 1
\end{equation}
\end{proposition}
\begin{proof}
We prove the general formula given by equation (18) by method of induction.\\ 
We have already done for $k = 2,3$ given in equation (13) and (17) respectively. We assume that up to $k-1$ the formula is valid i.e. 
\[ d_{k-1}: C_{k-1} \otimes A \rightarrow C_{k-2} \otimes A    \]
is given by 
\[ d_{k-1}( xy^{n_1}xy^{n_2}\dots xy^{n_{k-1}}x \otimes 1) =    \]
\[  xy^{n_1}xy^{n_2} \dots xy^{n_{k-2}}x \otimes y^{n_{k-1}}x - xy^{n_1}xy^{n_2}\dots xy^{n_{k-2}}x \otimes y^{n_{k-1}+1}   \]
\[  + xy^{n_1}x \dots y^{n_{k-2}+n_{k-1}+1}x \otimes 1 - xy^{n_1}x \dots y^{n_{k-3}+n_{k-2}+1}xy^{n_{k-1}}x \otimes 1    \]
\[ + \dots + (-1)^{k-2} xy^{n_1+n_2+1}xy^{n_3} \dots y^{n_{k-1}}x \otimes 1   \]
We will find the formula for $d_k: C_k \otimes A \rightarrow C_{k-1} \otimes A$ using the previous data.
\[  d_k(xy^{n_1}xy^{n_2}x \dots y^{n_k}x \otimes 1)  \]
\begin{equation}
= xy^{n_1}x \dots y^{n_{k-1}}x \otimes y^{n_k}x - i_{k-1}d_{k-1}(xy^{n_1}x \dots y^{n_{k-1}}x \otimes y^{n_k}x)   
\end{equation}
As by our previous formula for $d_{k-1}$ we have,
\[ d_{k-1}(xy^{n_1}x \dots y^{n_{k-1}}x \otimes y^{n_k}x) =    \]
\[ xy^{n_1}xy^{n_2}\dots y^{n_{k-2}}x \otimes y^{n_{k-1}}xy^{n_k+1} - xy^{n_1}xy^{n_2}\dots y^{n_{k-2}}x \otimes y^{n_{k-1}+n_k+1}x     \]
\[ + xy^{n_1}x \dots xy^{n_{k-2}+n_{k-1}+1}x \otimes y^{n_k}x + \dots     \]
\[ \dots + (-1)^{k-2} xy^{n_1+n_2+1}xy^{n_3}\dots xy^{n_{k-1}}x \otimes y^{n_k}x     \]
putting this value in equation (19) we have
\[ i_{k-1}( xy^{n_1}xy^{n_2}\dots y^{n_{k-2}}x \otimes y^{n_{k-1}}xy^{n_k+1}    \]
\[  - xy^{n_1}xy^{n_2}\dots y^{n_{k-2}}x \otimes y^{n_{k-1}+n_k+1}x + xy^{n_1}x \dots xy^{n_{k-2}+n_{k-1}+1}x \otimes y^{n_k}x + \dots      \]
\[ + (-1)^{k-2} xy^{n_1+n_2+1}xy^{n_3}\dots xy^{n_{k-1}}x \otimes y^{n_k}x )    \]
we have $xy^{n_1}xy^{n_2}\dots y^{n_{k-2}}xy^{n_{k-1}}xy^{n_k+1}$ as our leading term and so,
\[ = xy^{n_1}xy^{n_2} \dots y^{n_{k-2}}xy^{n_{k-1}}x \otimes y^{n_k+1}     \]
\begin{equation}
 + i_{k-1}( \text{all previous terms of}\hspace{1mm} i_{k-1} - d_{k-1}(xy^{n_1}xy^{n_2} \dots y^{n_{k-2}}xy^{n_{k-1}}x \otimes y^{n_k+1} ))
\end{equation}
Now
\[ d_{k-1}(xy^{n_1}xy^{n_2} \dots y^{n_{k-2}}xy^{n_{k-1}}x \otimes y^{n_k+1} ) =     \]
\[ xy^{n_1}xy^{n_2}\dots y^{n_{k-2}}x \otimes y^{n_{k-1}}xy^{n_k+1} - xy^{n_1}x \dots y^{n_{k-2}}x \otimes y^{n_{k-1}+n_k+2}   \]
\[  + xy^{n_1}x \dots xy^{n_{k-2}+n-{k-1}+1}x \otimes y^{n_k+1} - \dots    \]
\[ \dots + (-1)^{k-2} xy^{n_1+n_2+1}xy^{n_3}\dots y^{n_{k-1}}x \otimes y^{n_k+1}   \]
putting this value of $d_{k-1}$ in equation (20) we see that $xy^{n_1}xy^{n_2}\dots y^{n_{k-2}}x \otimes y^{n_{k-1}}xy^{n_k+1}$ term is getting cancelled and we proceed with remaining terms of $i_{k-1}$.\\ 
So we have noticed that when we apply $i_{k-1}$ for the first time total $2$ terms get cancelled and we left with $(k-1) +(k-1)= 2(k-1)$ terms.
Now
\[ i_{k-1}( -xy^{n_1}x \dots  y^{n_{k-2}}x \otimes y^{n_{k-1}+n_k+1}x + xy^{n_1}x \dots y^{n_{k-2}+n_{k-1}+1}x \otimes y^{n_k}x   \]
\[ - \dots  + (-1)^{k-2} xy^{n_1+n_2+1}x \dots y^{n_{k-1}}x \otimes y^{n_k}x \]
\[  + xy^{n_1}\dots xy^{n_{k-2}}x \otimes y^{n_{k-1}+n_k+2} - \dots    \]
\[  + \dots + (-1)^{k-2} xy^{n_1+n_2+1}x \dots y^{n_{k-1}}x \otimes y^{n_k+1} )    \]
We have $-xy^{n_1}x \dots  y^{n_{k-2}}xy^{n_{k-1}+n_k+1}x$ as our leading term and hence
\[  = - xy^{n_1}x \dots y^{n_{k-1}+n_k+1}x \otimes 1   \]
\begin{equation}
  + i_{k-1}( \text{all terms from previous}\hspace{1mm}i_{k-1} + d_{k-1}(xy^{n_1}x \dots y^{n_{k-1}+n_k+1}x \otimes 1 ))
\end{equation}
We have
\[ d_{k-1}(xy^{n_1}x \dots y^{n_{k-1}+n_k+1}x \otimes 1 ) = xy^{n_1}\dots y^{n_{k-2}}x \otimes y^{n_{k-1}+n_k+1}x    \]
\[ - xy^{n_1}\dots y^{n_{k-2}}x \otimes y^{n_{k-1}+n_k+2} + \dots     \]
\[  \dots + (-1)^{k-2} xy^{n_1+n_2+1}x \dots y^{n_{k-1}+n_k+1}x \otimes 1  \]
putting the value in equation (21) we get that $xy^{n_1}\dots y^{n_{k-2}}x \otimes y^{n_{k-1}+n_k+1}x$ and $xy^{n_1}\dots y^{n_{k-2}}x \otimes y^{n_{k-1}+n_k+2}$ get cancelled and we left with remaining terms to proceed.\\ 
So here after second application of $i_{k-1}$ we observe that total $4$ terms get canceled and we left with $(3k-2)$ number of terms.\\ 
We proceed with the rest of the structure of $i_{k-1}$ and here after 3rd application of it observe that total $6$ terms get canceled and we left with $4(k-3)$ number of terms. We will make a tabular form which will easily show us how many terms get cancelled after each $i_{k-1}$ operation and the remaining terms.
\begin{center}
\begin{tabular}{|c|c|c|}
\hline
Number of use of $i_{k-1}$    &  Total canceled terms      &  remaining terms for next $i_{k-1}$   \\ 
1                             &        2                   &  $2(k-1)$ \\
2                             &        4                   &  $3(k-2)$ \\
3                             &        6                   &  $4(k-3)$ \\
4                             &        8                   &  $5(k-4)$ \\
\vdots                        &      \vdots                &  \vdots   \\
$k$                           &       $2k$                 &  $k+1(k-k) = 0$ \\
\hline
\end{tabular}
\end{center}
So we observe that after $k$ steps of using $i_{k-1}$ we actually cancel all terms in the last used $i_{k-1}$ and so left with no terms. In each of these $k$ steps we get $1$ term (like in equation (19),(20),(21) etc) which will go to our main formula of $d_k$ and so
\[ d_k(xy^{n_1}xy^{n_2}x \dots y^{n_k}x \otimes 1) =   \]
\[ xy^{n_1}xy^{n_2}x \dots y^{n_{k-1}}x \otimes y^{n_k}x  -  xy^{n_1}x \dots y^{n_{k-1}}x \otimes y^{n_k+1}    \]
\[ xy^{n_1}xy^{n_2}x \dots y^{n_{k-1}+n_k+1}x \otimes 1 - xy^{n_1}xy^{n_2}x \dots y^{n_{k-2}+n_{k-1}+1}xy^{n_k}x \otimes 1   \]
\[ +  \dots + (-1)^{k-1} xy^{n_1+n_2+1}xy^{n_3}x \dots y^{n_k}x \otimes 1. \] 
 Our induction is complete and hence the proof.
\end{proof}
\end{example}
\subsection{Computation of Anick's resolution for algebra $B_n$}
Conner and Goetz in their paper \textit{$A_\infty$-algebra structures associated to $K_2$ algebras}[8] define algebra $B_n$ for each $n \in \mathbb{N}$. Algebra $B_n$ for each $n$ is a $K_2$ algebra which was recently introduced by Cassidy and Shelton[10] as a generalization of notion of Koszul algebra[7]  .\\ 
Algebra $B_n$ consists of a set of generators $\{a_0,b_0,c_0,a_1,b_1,c_1,\dots,a_n,b_n,c_n \}$ and a set of relations given by
\[ \{a_nb_nc_n, c_0a_0 \} \cup \{a_ib_ic_i + c_{i+1}a_{i+1}b_{i+1},b_{i+1}c_{i+1}a_{i+1},c_ic_{i+1},b_{i+1}a_i\}_{0 \leq i < n}    \]
We take the DEGLEX order with
\[ a_n > b_n > c_n > a_{n-1} > b_{n-1} > c_{n-1} > \dots > a_1 > b_1 > c_1 > a_0 > b_0 > c_0   \]
\begin{lemma}
The reduced Gr\"{o}bner basis of $B_n$ is given by the set of it's relations described above.
\end{lemma}
\begin{proof}
Let us denote the set of generators of algebra $B_n$ by $X$ and set of relations by $R$. The basis $X$ generates a free semigroup $\langle X\rangle$ with $1$ under the multiplication in $T(V)$, where $V$ is the $\mathbb{K}$-vector space with basis $X$ and $T(V)$ is the tensor algebra. We have the canonical quotient map $\pi_{B_n} : T(V) \rightarrow B_n$ of graded $\mathbb{K}$-algebras. Let $R' \subset \langle X \rangle$ be the set
\[ \{ a_nb_nc_n, c_0a_0 \} \cup \{c_{i+1}a_{i+1}b_{i+1}, b_{i+1}c_{i+1}a_{i+1}, c_ic_{i+1}, b_{i+1}a_i \}_{0 \leq i < n}  \]
and
\[ \langle R' \rangle = \{ AWB \hspace{2mm}|\hspace{2mm} A,B \in \langle X \rangle, W \in R'     \} \]
Then the following proposition is a straightforward application of Bergman's Diamond Lemma [11].
\begin{proposition}
The image under $\pi_{B_n}$ of $\langle X \rangle - \langle R' \rangle$, the set of tensors which do not contain any element of $R'$, gives a monomial $\mathbb{K}$-basis for $B_n$.
\end{proposition}
Proposition 4.4.1 implies that $R$ is the reduced Gr\"{o}bner basis of algebra $B_n$ [9]. For more details of the proof we refer to [8].
\end{proof}
Using the Gr\"{o}bner basis for $B_n$ we now start computing Anick's resolution for $B_1$, $B_2$ and will later reach the general formula for $B_n$.\\ 
Let us begin with the computation for $B_1$:
\subsubsection{Case $n=1$}
Algebra $B_1$ is given by the set of generators $\{a_0,b_0,c_0,a_1,b_1,c_1\}$ and the set of relations 
\[ \{a_1b_1c_1, c_0a_0, a_0b_0c_0 + c_1a_1b_1, b_1c_1a_1, c_0c_1, b_1a_0\}.   \]
Lemma 4.4.1 suggests that the Gr\"{o}bner basis for $B_1$ is given by the set of relations, and so the set of leading words $F$ or $1$-chain $C_1$ are
\[  a_1b_1c_1;\hspace{2mm} c_0a_0;\hspace{2mm} c_1a_1b_1;\hspace{2mm} b_1c_1a_1;\hspace{2mm} c_0c_1;\hspace{2mm} b_1a_0.  \]
We cannot construct higher chains from $c_0a_0$ and $b_1a_0$ and so we proceed with the rest to construct 2-chains. So elements of $C_2$ are 
\[ a_1b_1c_1a_1;\hspace{2mm} c_1a_1b_1c_1;\hspace{2mm} c_1a_1b_1a_0;\hspace{2mm} b_1c_1a_1b_1;\hspace{2mm} c_0c_1a_1b_1.   \]
We observe that we cannot produce higher chains from $c_1a_1b_1a_0$ and so we will form $C_3$ from the remaining elements of $C_2$. So elements of $C_3$ are,
\[ a_1b_1c_1a_1b_1c_1;\hspace{2mm} c_1a_1b_1c_1a_1b_1;\hspace{2mm} b_1c_1a_1b_1c_1a_1;\hspace{2mm} b_1c_1a_1b_1a_0;   \]
\[  c_0c_1a_1b_1c_1;\hspace{2mm} c_0c_1a_1b_1a_0   \]
Elements of $C_4$ are,
\[ a_1b_1c_1a_1b_1c_1a_1;\hspace{2mm} c_1a_1b_1c_1a_1b_1c_1;\hspace{2mm} c_1a_1b_1c_1a_1b_1a_0;  \]
\[ b_1c_1a_1b_1c_1a_1b_1;\hspace{2mm} c_0c_1a_1b_1c_1a_1b_1    \]
Elements of $C_5$ are,
\[ a_1b_1c_1a_1b_1c_1a_1b_1c_1;\hspace{2mm} c_1a_1b_1c_1a_1b_1c_1a_1b_1;\hspace{2mm} b_1c_1a_1b_1c_1a_1b_1a_0;   \]
\[ b_1c_1a_1b_1c_1a_1b_1c_1a_1;\hspace{2mm} c_0c_1a_1b_1c_1a_1b_1a_0;\hspace{2mm} c_0c_1a_1b_1c_1a_1b_1c_1.   \]
Looking at the constructions of $C_2$, $C_3$, $C_4$, $C_5$ we construct the $n$-chain.
\begin{lemma}
Elements of $C_n$ are of the form:\\ 
when $n$ is odd i.e. $n=2m+1$:
\[ a_1b_1c_1a_1b_1c_1\dots a_1b_1c_1; \hspace{2mm}\text{appearance of}\hspace{2mm} a_1b_1c_1 = m+1\hspace{2mm} \text{times}   \]
\[ c_1a_1b_1c_1a_1b_1 \dots c_1a_1b_1; \hspace{2mm}\text{appearance of}\hspace{2mm} c_1a_1b_1 = m+1\hspace{2mm} \text{times}    \]
\[ b_1c_1a_1b_1c_1a_1\dots b_1c_1a_1; \hspace{2mm}\text{appearance of}\hspace{2mm} b_1c_1a_1 = m+1\hspace{2mm} \text{times}   \]
\[ b_1c_1a_1b_1c_1a_1\dots b_1c_1a_1b_1a_0;\hspace{2mm}\text{appearance of}\hspace{2mm} b_1c_1a_1 = m\hspace{2mm} \text{times}   \]
\[ c_0c_1a_1b_1c_1a_1b_1\dots c_1a_1b_1c_1;\hspace{2mm}\text{appearance of}\hspace{2mm} c_1a_1b_1 = m\hspace{2mm} \text{times}  \]
\[ c_0c_1a_1b_1\dots c_1a_1b_1a_0; \hspace{2mm}\text{appearance of}\hspace{2mm} c_1a_1b_1 = m\hspace{2mm} \text{times}.   \]
When $n$ is even i.e. $n=2m$:
\[ a_1b_1c_1a_1b_1c_1 \dots a_1b_1c_1a_1;\hspace{2mm}\text{appearance of}\hspace{2mm} a_1b_1c_1 = m\hspace{2mm} \text{times}  \]
\[ c_1a_1b_1c_1a_1b_1 \dots c_1a_1b_1c_1;\hspace{2mm}\text{appearance of}\hspace{2mm} c_1a_1b_1 = m\hspace{2mm} \text{times}  \]
\[ c_1a_1b_1c_1a_1b_1 \dots c_1a_1b_1a_0;\hspace{2mm}\text{appearance of}\hspace{2mm} c_1a_1b_1 = m\hspace{2mm} \text{times}  \]
\[ b_1c_1a_1b_1c_1a_1 \dots b_1c_1a_1b_1;\hspace{2mm}\text{appearance of}\hspace{2mm} b_1a_1c_1 = m\hspace{2mm} \text{times}  \]
\[ c_0c_1a_1b_1 \dots c_1a_1b_1;\hspace{2mm}\text{appearance of}\hspace{2mm} c_1a_1b_1 = m\hspace{2mm} \text{times}.  \]
\end{lemma}
We have the chains, so we begin calculating $d_n$. We start with\\ 
$d_1: C_1 \otimes B_1 \rightarrow C_0 \otimes B_1$,
\[ d_1(a_1b_1c_1 \otimes 1) = a_1 \otimes b_1c_1  \]
\[ d_1(c_0a_0 \otimes 1) = c_0 \otimes a_0 \]
\[ d_1(c_1a_1b_1 \otimes 1) = c_1 \otimes a_1b_1 - a_0 \otimes b_0c_0  \]
\[ d_1(b_1c_1a_1 \otimes 1) = b_1 \otimes c_1a_1   \]
\[ d_1(c_0c_1 \otimes 1) = c_0 \otimes c_1   \]
\[ d_1(b_1a_0 \otimes 1) = b_1 \otimes a_0. \]
$d_2: C_2 \otimes B_1 \rightarrow C_1 \otimes B_1$ is given by,
\[ d_2(a_1b_1c_1a_1 \otimes 1) = a_1b_1c_1 \otimes a_1 - i_1d_1(a_1b_1c_1 \otimes a_1)   \]
\[ = a_1b_1c_1 \otimes a_1 - i_1(a_1 \otimes \overline{b_1c_1a_1})   \]
\[ = a_1b_1c_1 \otimes a_1, \hspace{2mm}\text{as}\hspace{1mm}b_1c_1a_1\hspace{1mm}\text{cannot be reduced to a normal word.} \]
\[ d_2(c_1a_1b_1c_1 \otimes 1) = c_1a_1b_1 \otimes c_1 - i_1d_1(c_1a_1b_1 \otimes c_1) \]
\[ = c_1a_1b_1 \otimes c_1 - i_1( c_1 \otimes \overline{a_1b_1c_1} - a_0 \otimes \overline{b_0c_0c_1})   \]
\[ = c_1a_1b_1 \otimes c_1 \]
\[ d_2(c_1a_1b_1a_0 \otimes 1) = c_1a_1b_1 \otimes a_0 - i_1d_1( c_1a_1b_1 \otimes a_0)   \]
\[ = c_1a_1b_1 \otimes a_0 - i_1( c_1 \otimes \overline{a_1b_1a_0} - a_0 \otimes \overline{b_0c_0a_0})    \]
\[ = c_1a_1b_1 \otimes a_0        \]
\[ d_2(b_1c_1a_1b_1 \otimes 1) = b_1c_1a_1 \otimes b_1 - i_1d_1(b_1c_1a_1 \otimes b_1)   \]
\[  = b_1c_1a_1 \otimes b_1 - i_1(b_1 \otimes \overline{c_1a_1b_1})    \]
\[ = b_1c_1a_1 \otimes b_1 + i_1(b_1 \otimes a_0b_0c_0)  \]
\[ = b_1c_1a_1 \otimes b_1 + b_1a_0 \otimes b_0c_0 + i_1( b_1 \otimes a_0b_0c_0 - d_1(b_1a_0 \otimes b_0c_0))   \]
\[ = b_1c_1a_1 \otimes b_1 + b_1a_0 \otimes b_0c_0   \]
\[ d_2(c_0c_1a_1b_1 \otimes 1) = c_0c_1 \otimes a_1b_1 - i_1d_1(c_0c_1 \otimes a_1b_1)  \]
\[ = c_0c_1 \otimes a_1b_1 - i_1( c_0 \otimes \overline{c_1a_1b_1})  \]
\[ = c_0c_1 \otimes a_1b_1 + i_1(c_0 \otimes a_0b_0c_0)   \]
\[ = c_0c_1 \otimes a_1b_1 + c_0a_0 \otimes b_0c_0 + i_1(c_0 \otimes a_0b_0c_0 - d_1(c_0a_0 \otimes b_0c_0))   \]
\[ = c_0c_1 \otimes a_1b_1 + c_0a_0 \otimes b_0c_0   \]
$d_3: C_3 \otimes B_1 \rightarrow C_2 \otimes B_1$ is given by,
\[ d_3(a_1b_1c_1a_1b_1c_1 \otimes 1) = a_1b_1c_1a_1 \otimes b_1c_1 - i_2d_2(a_1b_1c_1a_1 \otimes b_1c_1)    \]
\[ = a_1b_1c_1a_1 \otimes b_1c_1 - i_2(a_1b_1c_1 \otimes \overline{a_1b_1c_1})    \]
\[ = a_1b_1c_1a_1 \otimes b_1c_1   \]
\[ d_3(c_1a_1b_1c_1a_1b_1 \otimes 1) = c_1a_1b_1c_1 \otimes a_1b_1 - i_2d_2(c_1a_1b_1c_1 \otimes a_1b_1)    \]
\[ = c_1a_1b_1c_1 \otimes a_1b_1 - i_2( c_1a_1b_1 \otimes \overline{c_1a_1b_1})    \]
\[ = c_1a_1b_1c_1 \otimes a_1b_1 + i_2(c_1a_1b_1 \otimes a_0b_0c_0)    \]
\[ = c_1a_1b_1c_1 \otimes a_1b_1 + c_1a_1b_1a_0 \otimes b_0c_0 + i_2( c_1a_1b_1 \otimes a_0b_0c_0 - d_2(c_1a_1b_1a_0 \otimes b_0c_0))   \]
\[ = c_1a_1b_1c_1 \otimes a_1b_1 + c_1a_1b_1a_0 \otimes b_0c_0   \]
\[ d_3(b_1c_1a_1b_1c_1a_1 \otimes 1) = b_1c_1a_1b_1 \otimes c_1a_1 - i_2d_2(b_1c_1a_1b_1 \otimes c_1a_1)  \]
\[ = b_1c_1a_1b_1 \otimes c_1a_1 - i_2(b_1c_1a_1 \otimes \overline{b_1c_1a_1} + b_1a_0 \otimes \overline{b_0c_0c_1a_1})   \]
\[ = b_1c_1a_1b_1 \otimes c_1a_1   \]
\[ d_3(b_1c_1a_1b_1a_0 \otimes 1) = b_1c_1a_1b_1 \otimes a_0 - i_2d_2(b_1c_1a_1b_1 \otimes a_0)   \]
\[ = b_1c_1a_1b_1 \otimes a_0 - i_2( b_1c_1a_1 \otimes \overline{b_1a_0} + b_1a_0 \otimes \overline{b_0c_0a_0})  \]
\[ = b_1c_1a_1b_1 \otimes a_0   \]
\[ d_3(c_0c_1a_1b_1c_1 \otimes 1) = c_0c_1a_1b_1 \otimes c_1 - i_2d_2( c_0c_1a_1b_1 \otimes c_1)  \]
\[ = c_0c_1a_1b_1 \otimes c_1 - i_2(c_0c_1 \otimes \overline{a_1b_1c_1} + c_0a_0 \otimes \overline{b_0c_0c_1}   \]
\[ = c_0c_1a_1b_1 \otimes c_1   \]
\[ d_3(c_0c_1a_1b_1a_0 \otimes 1) = c_0c_1a_1b_1 \otimes a_0 - i_2d_2(c_0c_1a_1b_1 \otimes a_0)   \]
\[ = c_0c_1a_1b_1 \otimes a_0 - i_2( c_0c_1 \otimes \overline{a_1b_1a_0} + c_0a_0 \otimes \overline{b_0c_0a_0})    \]
\[ = c_0c_1a_1b_1 \otimes a_0  \]
\newpage 
$d_4: C_4 \otimes B_1 \rightarrow C_3 \otimes B_1$ is given by,
\[ d_4(a_1b_1c_1a_1b_1c_1a_1 \otimes 1) = a_1b_1c_1a_1b_1c_1 \otimes a_1 - i_3d_3(a_1b_1c_1a_1b_1c_1 \otimes a_1)   \]
\[ = a_1b_1c_1a_1b_1c_1 \otimes a_1 - i_3( a_1b_1c_1a_1 \otimes \overline{b_1c_1a_1})   \]
\[ = a_1b_1c_1a_1b_1c_1 \otimes a_1   \]
\[ d_4(c_1a_1b_1c_1a_1b_1c_1 \otimes 1) = c_1a_1b_1c_1a_1b_1 \otimes c_1 - i_3d_3(c_1a_1b_1c_1a_1b_1 \otimes c_1)   \]
\[ = c_1a_1b_1c_1a_1b_1 \otimes c_1 - i_3( c_1a_1b_1c_1 \otimes \overline{a_1b_1c_1} + c_1a_1b_1a_0 \otimes \overline{b_0c_0c_1})   \]
\[ = c_1a_1b_1c_1a_1b_1 \otimes c_1   \]
\[ d_4(b_1c_1a_1b_1c_1a_1b_1 \otimes 1) = b_1c_1a_1b_1c_1a_1 \otimes b_1 - i_3d_3(b_1c_1a_1b_1c_1a_1 \otimes b_1)     \]
\[ = b_1c_1a_1b_1c_1a_1 \otimes b_1 - i_3( b_1c_1a_1b_1 \otimes \overline{c_1a_1b_1})     \]
\[ =  b_1c_1a_1b_1c_1a_1 \otimes b_1 - i_3( b_1c_1a_1b_1 \otimes a_0b_0c_0)   \]
\[ = b_1c_1a_1b_1c_1a_1 \otimes b_1 + b_1c_1a_1b_1a_0 \otimes b_0c_0 - i_3(b_1c_1a_1b_1 \otimes a_0b_0c_0 - d_3(b_1c_1a_1b_1a_0 \otimes b_0c_0))     \]
\[ =b_1c_1a_1b_1c_1a_1 \otimes b_1 + b_1c_1a_1b_1a_0 \otimes b_0c_0  \]
\[ d_4(c_1a_1b_1c_1a_1b_1a_0 \otimes 1) = c_1a_1b_1c_1a_1b_1 \otimes a_0 - i_3d_3(c_1a_1b_1c_1a_1b_1 \otimes a_0)   \]
\[ = c_1a_1b_1c_1a_1b_1 \otimes a_0 - i_3(c_1a_1b_1c_1 \otimes \overline{a_1b_1a_0} + c_1a_1b_1a_0 \otimes \overline{b_0c_0a_0})   \]
\[ = c_1a_1b_1c_1a_1b_1 \otimes a_0    \]
\[ d_4(c_0c_1a_1b_1c_1a_1b_1 \otimes 1) = c_0c_1a_1b_1c_1 \otimes a_1b_1 - i_3d_3(c_0c_1a_1b_1c_1 \otimes a_1b_1)    \]
\[ = c_0c_1a_1b_1c_1 \otimes a_1b_1 - i_3( c_0c_1a_1b_1 \otimes \overline{c_1a_1b_1})   \]
\[ = c_0c_1a_1b_1c_1 \otimes a_1b_1 + c_0c_1a_1b_1a_0 \otimes b_0c_0 + i_3( c_0c_1a_1b_1 \otimes a_0b_0c_0 - d_3(c_0c_1a_1b_1a_0 \otimes b_0c_0))  \]
\[ = c_0c_1a_1b_1c_1 \otimes a_1b_1 + c_0c_1a_1b_1a_0 \otimes b_0c_0 .    \]
After observing the formula for $d_1$, $d_2$, $d_3$ and $d_4$ we can directly write the formula for $d_n$ which is as follows:
\begin{proposition}
\[ d_n : C_n \otimes B_1 \rightarrow C_{n-1} \otimes B_1   \]
When $n$ is odd i.e. $n = 2m+1$:
\begin{equation}
d_n(a_1b_1c_1a_1b_1c_1 \dots a_1b_1c_1 \otimes 1) = a_1b_1c_1 \dots a_1b_1c_1a_1 \otimes b_1c_1
\end{equation}
\begin{equation}
d_n(c_1a_1b_1 \dots c_1a_1b_1 \otimes 1) = c_1a_1b_1 \dots c_1a_1b_1c_1 \otimes a_1b_1 + c_1a_1b_1 \dots c_1a_1b_1a_0 \otimes b_0c_0
\end{equation}
\begin{equation}
d_n(b_1c_1a_1 \dots b_1c_1a_1 \otimes 1) = b_1c_1a_1 \dots b_1c_1a_1b_1 \otimes c_1a_1
\end{equation}
\begin{equation}
d_n(b_1c_1a_1 \dots b_1c_1a_1b_1a_0 \otimes 1) = b_1c_1a_1 \dots b_1c_1a_1b_1 \otimes a_0
\end{equation}
\begin{equation}
d_n(c_0c_1a_1b_1 \dots c_1a_1b_1c_1 \otimes 1) = c_0c_1a_1b_1c_1 \dots a_1b_1 \otimes c_1
\end{equation}
\begin{equation}
d_n(c_0c_1a_1b_1 \dots c_1a_1b_1a_0 \otimes 1) = c_0c_1a_1b_1 \dots c_1a_1b_1 \otimes a_0.
\end{equation}
\newpage 
When $n$ is even i.e. $n = 2m$:
\begin{equation}
d_n(a_1b_1c_1 \dots a_1b_1c_1a_1 \otimes 1) = a_1b_1c_1 \dots a_1b_1c_1 \otimes a_1 
\end{equation}
\begin{equation}
d_n(c_1a_1b_1 \dots c_1a_1b_1c_1 \otimes 1) = c_1a_1b_1 \dots c_1a_1b_1 \otimes c_1 
\end{equation}
\begin{equation}
d_n(c_1a_1b_1 \dots c_1a_1b_1a_0 \otimes 1) = c_1a_1b_1 \dots c_1a_1b_1 \otimes a_0 
\end{equation}
\[ d_n(b_1c_1a_1 \dots b_1c_1a_1b_1 \otimes 1) = b_1c_1a_1 \dots b_1c_1a_1 \otimes b_1 +    \]
\begin{equation}
b_1c_1a_1 \dots b_1c_1a_1b_1a_0 \otimes b_0c_0
\end{equation}
\[  d_n(c_0c_1a_1b_1 \dots c_1a_1b_1 \otimes 1) = c_0c_1a_1b_1 \dots c_1 \otimes a_1b_1 +  \]
\begin{equation}
c_0c_1a_1b_1 \dots c_1a_1b_1a_1 \otimes b_0c_0.
\end{equation}
\end{proposition}
\begin{proof}
For both even and odd case we can easily prove the general formula $d_n$ with the help of induction and combinatorics.
\end{proof}

\subsubsection{Case $n=2$}
Algebra $B_2$ is given by 9 generators and 10 relations and so the Gr\"{o}bner basis is given by:
\[ \{ a_2b_2c_2, c_0a_0, a_0b_0c_0 + c_1a_1b_1, a_1b_1c_1 + c_2a_2b_2, b_1c_1a_1   \}    \]
\[ \cup \{ b_2c_2a_2, c_0c_1, c_1c_2, b_1a_0, b_2a_1   \}   \]
So the elements of chain $C_1$ are 
\[ a_2b_2c_2,\hspace{2mm} c_0a_0,\hspace{2mm} c_1a_1b_1,\hspace{2mm}c_2a_2b_2,\hspace{2mm}b_1c_1a_1,  \]
\[ b_2c_2a_2,\hspace{2mm} c_0c_1,\hspace{2mm} c_1c_2,\hspace{2mm} b_1a_0,\hspace{2mm} b_2a_1   \]
Elements of chain $C_2$ are
\[ a_2b_2c_2a_2,\hspace{2mm} c_1a_1b_1a_0,\hspace{2mm} c_1a_1b_1c_1a_1,\hspace{2mm} c_2a_2b_2c_2, \]
\[ c_2a_2b_2a_1,\hspace{2mm} b_1c_1a_1b_1,\hspace{2mm} b_2c_2a_2b_2,\hspace{2mm} c_0c_1a_1b_1,\hspace{2mm} c_1c_2a_2b_2    \]
Similarly the elements of $C_3$ which are constructed from $C_2$ are,
\[ a_2b_2c_2a_2b_2c_2,\hspace{2mm} c_1a_1b_1c_1a_1b_1,\hspace{2mm} c_2a_2b_2c_2a_2b_2,\hspace{2mm} b_1c_1a_1b_1a_0,   \]
\[ b_1c_1a_1b_1c_1a_1,\hspace{2mm} b_2c_2a_2b_2c_2a_2,\hspace{2mm} b_2c_2a_2b_2a_1,\hspace{2mm} c_0c_1a_1b_1a_0,   \]
\[ c_0c_1a_1b_1c_1a_1,\hspace{2mm} c_1c_2a_2b_2a_1,\hspace{2mm} c_1c_2a_2b_2c_2   \]
The elements of $C_4$ are
\[ a_2b_2c_2a_2b_2c_2a_2,\hspace{2mm} c_1a_1b_1c_1a_1b_1a_0,\hspace{2mm} c_1a_1b_1c_1a_1b_1c_1a_1,   \]
\[ c_2a_2b_2c_2a_2b_2a_1,\hspace{2mm} c_2a_2b_2c_2a_2b_2c_2,\hspace{2mm} b_1c_1a_1b_1c_1a_1b_1,\hspace{2mm} b_2c_2a_2b_2c_2a_2b_2,    \]
\[ c_1c_2a_2b_2c_2a_2b_2,\hspace{2mm} c_0c_1a_1b_1c_1a_1b_1   \]
\newpage 
\begin{remark}
In this subsection of computation of Anick's resolution of $B_2$ as well as in remaining subsections while computing Anick's resolution of $B_n$ for $n=3, \dots$; when we will write an element of $C_n$ like some $b_{k+1}a_kb_kc_k \dots a_kb_kc_ka_{k-1}$, by '$\dots$', it will imply the consecutive appearance of $a_kb_kc_k$, like one we described in algebra $B_1$. 
\end{remark}
\begin{lemma}
We now write the general formula for elements of chain $C_n$. When $n$ is odd elements of $C_n$ are,
\[ a_2b_2c_2a_2b_2c_2 \dots a_2b_2c_2,\hspace{2mm} c_1a_1b_1c_1a_1b_1 \dots c_1a_1b_1,\hspace{2mm} c_2a_2b_2c_2a_2b_2 \dots c_2a_2b_2,   \]
\[ b_1c_1a_1 \dots b_1c_1a_1b_1a_0,\hspace{2mm} b_1c_1a_1b_1c_1a_1 \dots b_1c_1a_1,\hspace{2mm} b_2c_2a_2 \dots b_2c_2a_2,   \]
\[ b_2c_2a_2 \dots b_2c_2a_2b_2a_1,\hspace{2mm} c_0c_1a_1b_1c_1a_1b_1 \dots c_1a_1b_1a_0,\hspace{2mm} c_0c_1a_1b_1 \dots c_1a_1b_1c_1a_1,   \]
\[ c_1c_2a_2b_2 \dots c_2a_2b_2a_1,\hspace{2mm} c_1c_2a_2b_2 \dots c_2a_2b_2c_2.   \]
When $n$ is even, elements of $C_n$ are,
\[ a_2b_2c_2 \dots a_2b_2c_2a_2,\hspace{2mm} c_1a_1b_1 \dots c_1a_1b_1a_0,\hspace{2mm} c_1a_1b_1 \dots c_1a_1b_1c_1a_1,   \]
\[ b_1c_1a_1 \dots b_1c_1a_1b_1,\hspace{2mm} b_2c_2a_2 \dots b_2c_2a_2b_2,\hspace{2mm} c_0c_1a_1b_1 \dots c_1a_1b_1,  \]
\[ c_1c_2a_2b_2 \dots c_2a_2b_2,\hspace{2mm} c_2a_2b_2 \dots c_2a_2b_2c_2,\hspace{2mm} c_2a_2b_2 \dots c_2a_2b_2a_1.  \]
\end{lemma}
Now we will start our computations of Anick's resolution.
\[ d_1: C_1 \otimes B_2 \rightarrow C_0 \otimes B_2  \]
\[ d_1(a_2b_2c_2 \otimes 1) = a_2 \otimes b_2c_2 - i_0d_0(a_2 \otimes b_2c_2)   \]
\[ = a_2 \otimes b_2c_2 - i_0( 1 \otimes \overline{a_2b_2c_2})   \]
\[ = a_2 \otimes b_2c_2  \]
\[ d_1(c_0a_0 \otimes 1) = c_0 \otimes a_0  \]
\[ d_1(c_1a_1b_1 \otimes 1) = c_1 \otimes a_1b_1 - i_0d_0( c_1 \otimes a_1b_1)   \]
\[ = c_1 \otimes a_1b_1 - i_0(1 \otimes \overline{c_1a_1b_1})   \]
\[ = c_1 \otimes a_1b_1 + a_0 \otimes b_0c_0  \]
\[ d_1(c_2a_2b_2 \otimes 1) = c_2 \otimes a_2b_2 + a_1 \otimes b_1c_1  \]
\[ d_1(b_1c_1a_1 \otimes 1) = b_1 \otimes c_1a_1 - i_0d_0(b_1 \otimes c_1a_1)  \]
\[ = b_1 \otimes c_1a_1 - i_0(1 \otimes \overline{b_1c_1a_1})  \]
\[ = b_1 \otimes c_1a_1   \]
\[ d_1(b_2c_2a_2 \otimes 1) = b_2 \otimes c_2a_2   \]
\[ d_1(c_0c_1 \otimes 1) = c_0 \otimes c_1  \]
\[ d_1(c_1c_2 \otimes 1) = c_1 \otimes c_2 \]
\[ d_1(b_1a_0 \otimes 1) = b_1 \otimes a_0   \]
\[ d_1(b_2a_1 \otimes 1) = b_2 \otimes a_1.  \]
\newpage 
\[ d_2: C_2 \otimes B_2 \rightarrow C_1 \otimes B_2  \]
are given by,
\[  d_2(a_2b_2c_2a_2 \otimes 1) = a_2b_2c_2 \otimes a_2 - i_1d_1(a_2b_2c_2 \otimes a_2) \]
\[ = a_2b_2c_2 \otimes a_2 - i_1(a_2 \otimes \overline{b_2c_2a_2})   \]
\[ = a_2b_2c_2 \otimes a_2  \]
\[ d_2(c_1a_1b_1a_0 \otimes 1) = c_1a_1b_1 \otimes a_0 - i_1d_1(c_1a_1b_1 \otimes a_0)  \]
\[ = c_1a_1b_1 \otimes a_0 - i_1( c_1 \otimes \overline{a_1b_1a_0} + a_0 \otimes \overline{b_0c_0a_0})   \]
\[ = c_1a_1b_1 \otimes a_0  \]
\[ d_2(c_1a_1b_1c_1a_1 \otimes 1) = c_1a_1b_1 \otimes c_1a_1 - i_1d_1(c_1a_1b_1 \otimes c_1a_1)  \]
\[ = c_1a_1b_1 \otimes c_1a_1 - i_1(c_1 \otimes \overline{a_1b_1c_1a_1} + a_0 \otimes \overline{b_0c_0c_1a_1})   \]
\[ = c_1a_1b_1 \otimes c_1a_1  \]
\[ d_2(c_2a_2b_2c_2 \otimes 1) = c_2a_2b_2 \otimes c_2 - i_1d_1(c_2a_2b_2 \otimes c_2)  \]
\[ = c_2a_2b_2 \otimes c_2 - i_1( c_2 \otimes \overline{a_2b_2c_2} + a_1 \otimes \overline{b_1c_1c_2})   \]
\[ = c_2a_2b_2 \otimes c_2  \]
\[ d_2(c_2a_2b_2a_1 \otimes 1) = c_2a_2b_2 \otimes a_1 - i_1d_1(c_2a_2b_2 \otimes a_1)   \]
\[ = c_2a_2b_2 \otimes a_1 - i_1( c_2 \otimes \overline{a_2b_2a_1} + a_1 \otimes \overline{b_1c_1a_1})    \]
\[ = c_2a_2b_2 \otimes a_1   \]
\[ d_2(b_1c_1a_1b_1 \otimes 1) = b_1c_1a_1 \otimes b_1 - i_1d_1(b_1c_1a_1 \otimes b_1)   \]
\[ = b_1c_1a_1 \otimes b_1  - i_1(b_1 \otimes \overline{c_1a_1b_1})   \]
\[ = b_1c_1a_1 \otimes b_1 + i_1(b_1 \otimes a_0b_0c_0)   \]
\[ = b_1c_1a_1 \otimes b_1 + b_1a_0 \otimes b_0c_0 + i_1(b_1 \otimes a_0b_0c_0 - d_1(b_1a_0 \otimes b_0c_0))  \]
\[ = b_1c_1a_1 \otimes b_1 + b_1a_0 \otimes b_0c_0   \]
Similarly we can formulate for $d_2(b_2c_2a_2b_2 \otimes 1)$ like the previous one,
\[ d_2(b_2c_2a_2b_2 \otimes 1) = b_2c_2a_2 \otimes b_2 + b_2a_1 \otimes b_1c_1   \]
\[ d_2(c_0c_1a_1b_1 \otimes 1) = c_0c_1 \otimes a_1b_1 - i_1d_1(c_0c_1 \otimes a_1b_1)  \]
\[ = c_0c_1 \otimes a_1b_1 - i_1(c_0 \otimes \overline{c_1a_1b_1})   \]
\[ = c_0c_1 \otimes a_1b_1 + c_0a_0 \otimes b_0c_0 + i_1( c_0 \otimes a_0b_0c_0 - d_1(c_0a_0 \otimes b_0c_0))   \]
\[ = c_0c_1 \otimes a_1b_1 + c_0a_0 \otimes b_0c_0   \]
\newpage 
\[ d_2(c_1c_2a_2b_2 \otimes 1) = c_1c_2 \otimes a_2b_2 - i_1d_1(c_1c_2 \otimes a_2b_2)   \]
\[ = c_1c_2 \otimes a_2b_2 - i_1(c_1 \otimes \overline{c_2a_2b_2})   \]
\[ = c_1c_2 \otimes a_2b_2 + i_1(c_1 \otimes a_1b_1c_1)   \]
\[ = c_1c_2 \otimes a_2b_2 + c_1a_1b_1 \otimes c_1 + i_1(c_1 \otimes a_1b_1c_1 - d_1(c_1a_1b_1 \otimes c_1))    \]
\[ = c_1c_2 \otimes a_2b_2 + c_1a_1b_1 \otimes c_1.   \]
Now 
\[ d_3: C_3 \otimes B_2 \rightarrow C_2 \otimes B_2  \]
are given by,
\[ d_3(a_2b_2c_2a_2b_2c_2 \otimes 1) = a_2b_2c_2a_2 \otimes b_2c_2 - i_2d_2(a_2b_2c_2a_2 \otimes b_2c_2)  \]
\[ = a_2b_2c_2a_2 \otimes b_2c_2 - i_2( a_2b_2c_2 \otimes \overline{a_2b_2c_2})   \]
\[ = a_2b_2c_2a_2 \otimes b_2c_2   \]
\[ d_3(c_1a_1b_1c_1a_1b_1 \otimes 1) = c_1a_1b_1c_1a_1 \otimes b_1 - i_2d_2(c_1a_1b_1c_1a_1 \otimes b_1)   \]
\[ = c_1a_1b_1c_1a_1 \otimes b_1 - i_2(c_1a_1b_1 \otimes \overline{c_1a_1b_1})   \]
\[ = c_1a_1b_1c_1a_1 \otimes b_1 + i_2(c_1a_1b_1 \otimes a_0b_0c_0)   \]
\[ = c_1a_1b_1c_1a_1 \otimes b_1 + c_1a_1b_1a_0 \otimes b_0c_0   \]
\[ d_3(c_2a_2b_2c_2a_2b_2 \otimes 1) = c_2a_2b_2c_2 \otimes a_2b_2 - i_2d_2(c_2a_2b_2c_2 \otimes a_2b_2)   \]
\[ =  c_2a_2b_2c_2 \otimes a_2b_2 - i_2(c_2a_2b_2 \otimes \overline{c_2a_2b_2})  \]
\[ = c_2a_2b_2c_2 \otimes a_2b_2 + i_2(c_2a_2b_2 \otimes a_1b_1c_1)  \]
\[ = c_2a_2b_2c_2 \otimes a_2b_2 + c_2a_2b_2a_1 \otimes b_1c_1 + i_2(c_2a_2b_2 \otimes a_1b_1c_1 - d_1(c_2a_2b_2a_1 \otimes b_1c_1))    \]
\[ = c_2a_2b_2c_2 \otimes a_2b_2 + c_2a_2b_2a_1 \otimes b_1c_1   \]
\[ d_3(b_1c_1a_1b_1a_0 \otimes 1) = b_1c_1a_1b_1 \otimes a_0 - i_2d_2(b_1c_1a_1b_1 \otimes a_0)   \]
\[ = b_1c_1a_1b_1 \otimes a_0 - i_2(b_1c_1a_1 \otimes \overline{b_1a_0} + b_1a_0 \otimes \overline{b_0c_0a_0})   \]
\[ = b_1c_1a_1b_1 \otimes a_0   \]
\[ d_3(b_1c_1a_1b_1c_1a_1 \otimes 1) = b_1c_1a_1b_1 \otimes c_1a_1 - i_2d_2(b_1c_1a_1b_1 \otimes c_1a_1)  \]
\[ = b_1c_1a_1b_1 \otimes c_1a_1 - i_2( b_1c_1a_1 \otimes \overline{b_1c_1a_1} + b_1a_0 \otimes \overline{b_0c_0c_1a_1})   \]
\[ = b_1c_1a_1b_1 \otimes c_1a_1   \]
\[ d_3(b_2c_2a_2b_2c_2a_2 \otimes 1) = b_2c_2a_2b_2 \otimes c_2a_2   \]
\[  d_3(b_2c_2a_2b_2a_1 \otimes 1) = b_2c_2a_2b_2 \otimes a_1 - i_2d_2(b_2c_2a_2b_2 \otimes a_1)  \]
\[ = b_2c_2a_2b_2 \otimes a_1 - i_2( b_2c_2a_2 \otimes \overline{b_2a_1} + b_2a_1 \otimes \overline{b_1c_1a_1})   \]
\[ = b_2c_2a_2b_2 \otimes a_1   \]
\newpage 
\[ d_3(c_0c_1a_1b_1a_0 \otimes 1) = c_0c_1a_1b_1 \otimes a_0 - i_2d_2(c_0c_1a_1b_1 \otimes a_0)   \]
\[ = c_0c_1a_1b_1 \otimes a_0 - i_2(c_0c_1 \otimes \overline{a_1b_1a_0} + c_0a_0 \otimes \overline{b_0c_0a_0})   \]
\[ = c_0c_1a_1b_1 \otimes a_0   \]
\[ d_3(c_0c_1a_1b_1c_1a_1 \otimes 1) = c_0c_1a_1b_1 \otimes c_1a_1 - i_2d_2(c_0c_1a_1b_1 \otimes c_1a_1)  \]
\[ = c_0c_1a_1b_1 \otimes c_1a_1 - i_2(c_0c_1 \otimes \overline{a_1b_1c_1a_1} + c_0a_0 \otimes \overline{b_0c_0c_1a_1})  \]
\[ = c_0c_1a_1b_1 \otimes c_1a_1  \]
\[ d_3(c_1c_2a_2b_2a_1 \otimes 1) = c_1c_2a_2b_2 \otimes a_1 - i_2d_2(c_1c_2a_2b_2 \otimes a_1)  \]
\[ = c_1c_2a_2b_2 \otimes a_1 - i_2(c_1c_2 \otimes \overline{a_2b_2a_1} + c_1a_1b_1 \otimes \overline{c_1a_1})   \]
\[ =  c_1c_2a_2b_2 \otimes a_1 - c_1a_1b_1c_1a_1 \otimes 1 - i_2( c_1a_1b_1 \otimes c_1a_1 - d_2(c_1a_1b_1c_1a_1 \otimes 1)) \]
\[ = c_1c_2a_2b_2 \otimes a_1 - c_1a_1b_1c_1a_1 \otimes 1  \]
\[ d_3(c_1c_2a_2b_2c_2 \otimes 1) = c_1c_2a_2b_2 \otimes c_2 - i_2d_2(c_1c_2a_2b_2 \otimes c_2) \]
\[ = c_1c_2a_2b_2 \otimes c_2 - i_2(c_1c_2 \otimes \overline{a_2b_2c_2} + c_1a_1b_1 \otimes \overline{c_1c_2})  \]
\[ = c_1c_2a_2b_2 \otimes c_2   \]
So after observing the formulas for $d_1$, $d_2$ and $d_3$, we can easily guess the general formula for $d_n$ which is given by:
\begin{proposition}
\[ d_n: C_n \otimes B_2 \rightarrow C_{n-1} \otimes B_2  \]
when $n$ is even,
\begin{equation}
d_n(a_2b_2c_2 \dots a_2b_2c_2a_2 \otimes 1) = a_2b_2c_2 \dots a_2b_2c_2 \otimes a_2
\end{equation}
\begin{equation}
d_n(c_1a_1b_1 \dots c_1a_1b_1a_0 \otimes 1) = c_1a_1b_1 \dots c_1a_1b_1 \otimes a_0
\end{equation}
\begin{equation}
d_n(c_1a_1b_1 \dots c_1a_1b_1c_1a_1 \otimes 1) = c_1a_1b_1 \dots c_1a_1b_1 \otimes c_1a_1
\end{equation}
\begin{equation}
d_n(c_2a_2b_2 \dots c_2a_2b_2c_2 \otimes 1) = c_2a_2b_2 \dots c_2a_2b_2 \otimes c_2
\end{equation}
\begin{equation}
d_n(c_2a_2b_2 \dots c_2a_2b_2a_1 \otimes 1) = c_2a_2b_2 \dots c_2a_2b_2 \otimes a_1
\end{equation}
\begin{equation}
d_n(b_1c_1a_1 \dots b_1c_1a_1b_1 \otimes 1) = b_1c_1a_1 \dots b_1c_1a_1 \otimes b_1 + b_1c_1a_1b_1 \dots c_1a_1b_1a_0 \otimes b_0c_0
\end{equation}
\begin{equation}
d_n(b_2c_2a_2 \dots b_2c_2a_2b_2 \otimes 1) = b_2c_2a_2 \dots b_2c_2a_2 \otimes b_2 + b_2c_2a_2b_2 \dots c_2a_2b_2a_1 \otimes b_1c_1
\end{equation}
\begin{equation}
d_n(c_0c_1a_1b_1 \dots c_1a_1b_1 \otimes 1) = c_0c_1a_1b_1 \dots c_1a_1 \otimes b_1 + c_0c_1a_1b_1 \dots c_1a_1b_1a_0 \otimes b_0c_0
\end{equation}
except,
\begin{equation}
d_2(c_0c_1a_1b_1 \otimes 1) = c_0c_1 \otimes a_1b_1 + c_0a_0 \otimes b_0c_0
\end{equation}
\newpage 
\begin{equation}
d_n(c_1c_2a_2b_2 \dots c_2a_2b_2 \otimes 1) = c_1c_2a_2b_2 \dots c_2 \otimes a_2b_2
\end{equation}
except
\begin{equation}
d_2(c_1c_2a_2b_2 \otimes 1) = c_1c_2 \otimes a_2b_2 + c_1a_1b_1 \otimes c_1.
\end{equation}
When $n$ is odd,
\begin{equation}
d_n(a_2b_2c_2 \dots a_2b_2c_2 \otimes 1) = a_2b_2c_2 \dots a_2b_2c_2a_2 \otimes b_2c_2
\end{equation}
\begin{equation}
d_n(c_1a_1b_1 \dots c_1a_1b_1 \otimes 1) = c_1a_1b_1 \dots c_1a_1 \otimes b_1 + c_1a_1b_1 \dots c_1a_1b_1a_0 \otimes b_0c_0
\end{equation}
\begin{equation}
d_n(c_2a_2b_2 \dots c_2a_2b_2 \otimes 1) = c_2a_2b_2 \dots c_2 \otimes a_2b_2 + c_2a_2b_2 \dots c_2a_2b_2a_1 \otimes b_1c_1
\end{equation}
\begin{equation}
d_n(b_1c_1a_1 \dots b_1c_1a_1b_1a_0 \otimes 1) = b_1c_1a_1 \dots b_1c_1a_1b_1 \otimes a_0
\end{equation}
\begin{equation}
d_n(b_1c_1a_1 \dots b_1c_1a_1 \otimes 1) = b_1c_1a_1 \dots b_1 \otimes c_1a_1
\end{equation}
\begin{equation}
d_n(b_2c_2a_2 \dots b_2c_2a_2 \otimes 1) = b_2c_2a_2 \dots b_2 \otimes c_2a_2
\end{equation}
\begin{equation}
d_n(b_2c_2a_2 \dots b_2c_2a_2b_2a_1 \otimes 1) = b_2c_2a_2 \dots b_2c_2a_2b_2 \otimes a_1
\end{equation}
\begin{equation}
d_n(c_0c_1a_1b_1 \dots c_1a_1b_1a_0 \otimes 1) = c_0c_1a_1b_1 \dots c_1a_1b_1 \otimes a_0 
\end{equation}
\begin{equation}
d_n(c_0c_1a_1b_1 \dots c_1a_1b_1c_1a_1 \otimes 1) = c_0c_1a_1b_1 \dots c_1a_1b_1 \otimes c_1a_1
\end{equation}
\begin{equation}
d_n(c_1c_2a_2b_2 \dots c_2a_2b_2c_2 \otimes 1) = c_1c_2a_2b_2 \dots c_2a_2b_2 \otimes c_2
\end{equation}
\begin{equation}
d_n(c_1c_2a_2b_2 \dots c_2a_2b_2a_1 \otimes 1) = c_1c_2a_2b_2 \dots c_2a_2b_2 \otimes a_1
\end{equation}
except,
\begin{equation}
d_n(c_1c_2a_2b_2a_1 \otimes 1) = c_1c_2a_2b_2 \otimes a_1 - c_1a_1b_1c_1a_1 \otimes 1.
\end{equation}
\end{proposition}
We can easily verify these formulas with the help of induction.\\ 
We just write the $n$-chains and the general formula, $d_n$ for $B_3$ before proceeding to write the formula of Anick's resolution structure for $B_n$.
\subsubsection{Case $n=3$}
\begin{lemma}
The elements of chain $C_n$, when $n$ is odd are,
\[ a_3b_3c_3 \dots a_3b_3c_3,\hspace{2mm} c_1a_1b_1 \dots c_1a_1b_1,\hspace{2mm} c_2a_2b_2 \dots c_2a_2b_2,   \]
\[ c_3a_3b_3 \dots c_3a_3b_3,\hspace{2mm} b_1c_1a_1 \dots b_1c_1a_1,\hspace{2mm} b_1c_1a_1 \dots b_1c_1a_1b_1a_0, \]
\[ b_2c_2a_2 \dots b_2c_2a_2,\hspace{2mm} b_2c_2a_2 \dots b_2c_2a_2b_2a_1,\hspace{2mm} b_3c_3a_3 \dots b_3c_3a_3,   \]
\[ b_3c_3a_3 \dots b_3c_3a_3b_3a_2,\hspace{2mm} c_0c_1a_1b_1 \dots c_1a_1b_1a_0,\hspace{2mm} c_0c_1a_1b_1 \dots c_1a_1b_1c_1a_1,    \]
\[ c_1c_2a_2b_2 \dots c_2a_2b_2c_2a_2,\hspace{2mm} c_1c_2a_2b_2 \dots c_2a_2b_2a_1,   \]
\[ c_2c_3a_3b_3 \dots c_3a_3b_3c_3,\hspace{2mm} c_2c_3a_3b_3 \dots c_3a_3b_3a_2.    \]
\end{lemma}
\newpage 
\begin{proposition}
 When $n$ is odd, $d_n: C_n \otimes B_3 \rightarrow C_{n-1} \otimes B_3$ is given by,
\begin{equation}
d_n(a_3b_3c_3 \dots a_3b_3c_3 \otimes 1) = a_3b_3c_3 \dots a_3 \otimes b_3c_3
\end{equation}
\begin{equation}
d_n(c_1a_1b_1 \dots c_1a_1b_1 \otimes 1) = c_1a_1b_1 \dots c_1a_1 \otimes b_1 + c_1a_1b_1 \dots c_1a_1b_1a_0 \otimes b_0c_0
\end{equation}
\begin{equation}
d_n(c_2a_2b_2 \dots c_2a_2b_2 \otimes 1) = c_2a_2b_2 \dots c_2a_2 \otimes b_2 + c_2a_2b_2 \dots c_2a_2b_2a_1 \otimes b_1c_1
\end{equation}
\begin{equation}
d_n(c_3a_3b_3 \dots c_3a_3b_3 \otimes 1) = c_3a_3b_3 \dots c_3a_3 \otimes b_3 + c_3a_3b_3 \dots c_3a_3b_3a_2 \otimes b_2c_2
\end{equation}
\begin{equation}
d_n(b_1c_1a_1 \dots b_1c_1a_1 \otimes 1) = b_1c_1a_1 \dots b_1 \otimes c_1a_1
\end{equation}
\begin{equation}
d_n(b_2c_2a_2 \dots b_2c_2a_2 \otimes 1) = b_2c_2a_2 \dots b_2 \otimes c_2a_2
\end{equation}
\begin{equation}
d_n(b_3c_3a_3 \dots b_3c_3a_3 \otimes 1) = b_3c_3a_3 \dots b_3 \otimes c_3a_3
\end{equation}
\begin{equation}
d_n(b_1c_1a_1 \dots b_1c_1a_1b_1a_0 \otimes 1) = b_1c_1a_1 \dots b_1c_1a_1b_1 \otimes a_0
\end{equation}
\begin{equation}
d_n(b_2c_2a_2 \dots b_2c_2a_2b_2a_1 \otimes 1) = b_2c_2a_2 \dots b_2c_2a_2b_2 \otimes a_1
\end{equation}
\begin{equation}
d_n(b_3c_3a_3 \dots b_3c_3a_3b_3a_2 \otimes 1) = b_3c_3a_3 \dots b_3c_3a_3b_3 \otimes a_2
\end{equation}
\begin{equation}
d_n(c_0c_1a_1b_1 \dots c_1a_1b_1a_0 \otimes 1) = c_0c_1a_1b_1 \dots a_1b_1 \otimes a_0
\end{equation}
\begin{equation}
d_n(c_0c_1a_1b_1 \dots c_1a_1b_1c_1a_1 \otimes 1) = c_0c_1a_1b_1 \dots c_1a_1b_1 \otimes c_1a_1
\end{equation}
\begin{equation}
d_n(c_1c_2a_2b_2 \dots c_2a_2b_2a_1 \otimes 1) = c_1c_2a_2b_2 \dots c_2a_2b_2 \otimes a_1
\end{equation}
\begin{equation}
d_n(c_1c_2a_2b_2 \dots c_2a_2b_2c_2a_2 \otimes 1) = c_1c_2a_2b_2 \dots c_2a_2b_2 \otimes c_2a_2
\end{equation}
\begin{equation}
d_n(c_2c_3a_3b_3 \dots c_3a_3b_3c_3 \otimes 1) = c_2c_3a_3b_3 \dots c_3a_3b_3 \otimes c_3
\end{equation}
\begin{equation}
d_n(c_2c_3a_3b_3 \dots c_3a_3b_3a_2 \otimes 1) = c_2c_3a_3b_3 \dots c_3a_3b_3 \otimes a_2.
\end{equation}
\end{proposition}
\begin{lemma}
When $n$ is even the elements of $C_n$ are,
\[ a_3b_3c_3 \dots a_3b_3c_3a_3,\hspace{2mm} c_1a_1b_1 \dots c_1a_1b_1a_0,\hspace{2mm} c_1a_1b_1 \dots c_1a_1b_1c_1a_1,  \]
\[ c_2a_2b_2 \dots c_2a_2b_2a_1,\hspace{2mm} c_2a_2b_2 \dots c_2a_2b_2c_2a_2,\hspace{2mm} c_3a_3b_3 \dots c_3a_3b_3c_3a_3,  \]
\[ c_3a_3b_3 \dots c_3a_3b_3a_2,\hspace{2mm} b_1c_1a_1 \dots b_1c_1a_1b_1,\hspace{2mm} b_2c_2a_2 \dots b_2c_2a_2b_2,   \]
\[ b_3c_3a_3 \dots b_3c_3a_3b_3,\hspace{2mm} c_0c_1a_1b_1 \dots c_1a_1b_1,\hspace{2mm} c_1c_2a_2b_2 \dots c_2a_2b_2,   \]
\[ c_2c_3a_3b_3 \dots c_3a_3b_3.   \]
\end{lemma}
\begin{proposition} 
\[ d_n: C_n \otimes B_3 \rightarrow C_{n-1} \otimes B_3  \]
are given by,
\begin{equation}
d_n(a_3b_3c_3 \dots a_3b_3c_3a_3 \otimes 1) = a_3b_3c_3 \dots a_3b_3c_3 \otimes a_3
\end{equation}
\begin{equation}
d_n(c_1a_1b_1 \dots c_1a_1b_1a_0 \otimes 1) = c_1a_1b_1 \dots c_1a_1b_1 \otimes a_0
\end{equation}
\newpage 
\begin{equation}
d_n(c_1a_1b_1 \dots c_1a_1b_1c_1a_1 \otimes 1) = c_1a_1b_1 \dots c_1a_1b_1 \otimes c_1a_1
\end{equation}
\begin{equation}
d_n(c_2a_2b_2 \dots c_2a_2b_2a_1 \otimes 1) = c_2a_2b_2 \dots c_2a_2b_2 \otimes a_1
\end{equation}
\begin{equation}
d_n(c_2a_2b_2 \dots c_2a_2b_2c_2a_2 \otimes 1) = c_2a_2b_2 \dots c_2a_2b_2 \otimes c_2a_2
\end{equation}
\begin{equation}
d_n(c_3a_3b_3 \dots c_3a_3b_3a_2 \otimes 1) = c_3a_3b_3 \dots c_3a_3b_3 \otimes a_2
\end{equation}
\begin{equation}
d_n(c_3a_3b_3 \dots c_3a_3b_3c_3a_3 \otimes 1) = c_3a_3b_3 \dots c_3a_3b_3 \otimes c_3a_3
\end{equation}
\begin{equation}
d_n(b_1c_1a_1 \dots b_1c_1a_1b_1 \otimes 1) = b_1c_1a_1 \dots b_1c_1a_1 \otimes b_1 + b_1c_1a_1b_1 \dots c_1a_1b_1a_0 \otimes b_0c_0
\end{equation}
\begin{equation}
d_n(b_2c_2a_2 \dots b_2c_2a_2b_2 \otimes 1) = b_2c_2a_2 \dots b_2c_2a_2 \otimes b_2 + b_2c_2a_2b_2 \dots c_2a_2b_2a_1 \otimes b_1c_1
\end{equation}
\begin{equation}
d_n(b_3c_3a_3 \dots b_3c_3a_3b_3 \otimes 1) = b_3c_3a_3 \dots b_3c_3a_3 \otimes b_3 + b_3c_3a_3b_3 \dots c_3a_3b_3 \otimes b_2c_2
\end{equation}
\end{proposition}
Again we can easily prove the formula of $d_n$ using the methods of induction. 

\subsubsection{The general formula}
We recall algebra $B_n$, which consists of a set of generators\\  $\{a_0,b_0,c_0,a_1,b_1,c_1,\dots,a_n,b_n,c_n \}$ and a set of relations given by
\[ \{a_nb_nc_n, c_0a_0 \} \cup \{a_ib_ic_i + c_{i+1}a_{i+1}b_{i+1},b_{i+1}c_{i+1}a_{i+1},c_ic_{i+1},b_{i+1}a_i\}_{0 \leq i < n}    \]
According to lemma 4.4.1, the reduced Gr\"{o}bner basis of algebra $B_n$ is given by the set of relations defined above. The set of leading terms, denoted by $F$, are:
\[ a_nb_nc_n,\hspace{3mm} c_0a_0,\hspace{3mm} \{c_ia_ib_i\}_{1 \leq i \leq n},\hspace{3mm}  \{b_ic_ia_i\}_{1 \leq i \leq n},\hspace{3mm} \{c_ic_{i+1}\}_{0 \leq i \leq n-1},\hspace{3mm} \{b_ia_{i-1}\}_{1 \leq i \leq n}.   \]
\begin{lemma}
Total number of elements in $C_n$, when $n$ is odd, are $5n+1$. The elements are:
\[ a_nb_nc_n\dots a_nb_nc_n,\hspace{3mm} \{c_ia_ib_i \dots c_ia_ib_i\}_{1 \leq i \leq n},\hspace{3mm} \{b_ic_ia_i \dots b_ic_ia_i\}_{1 \leq i \leq n} \]
\[  \{b_ic_ia_i \dots b_ic_ia_ib_ia_{i-1}\}_{1 \leq i \leq n},\hspace{3mm} \{c_{i-1}c_ia_ib_i \dots c_ia_ib_ia_{i-1}\}_{1 \leq i \leq n},\] \[ \{c_{i-1}c_ia_ib_i \dots c_ia_ib_ic_ia_i\}_{1 \leq i \leq n-1},\hspace{3mm}   c_{n-1}c_na_nb_n \dots c_na_nb_nc_n. \]
\end{lemma}
\begin{lemma}
When $n$ is even, the elements of $C_n$ are:
\[ a_nb_nc_n \dots a_nb_nc_na_n,\hspace{3mm} \{c_ia_ib_i \dots c_ia_ib_ic_ia_i\}_{1 \leq i \leq n-1},\hspace{3mm} c_na_nb_n \dots c_na_nb_nc_n,   \]
\[ \{c_ia_ib_i \dots c_ia_ib_ia_{i-1}\}_{1 \leq i \leq n},\hspace{3mm} \{b_ic_ia_i \dots b_ic_ia_ib_i\}_{1 \leq i \leq n},\] \[ \{c_{i-1}c_ia_ib_i \dots c_ia_ib_i\}_{1 \leq i \leq n}.   \]
Total number of elements in $C_n$ are $4n+1$.
\end{lemma}
\begin{proposition}
When $n$ is odd,
\[ d_n : C_n \otimes B_n \rightarrow C_{n-1} \otimes B_n  \]
are given by
\[ d_n(a_nb_nc_n \dots a_nb_nc_n \otimes 1) = a_nb_nc_n \dots a_nb_nc_na_n \otimes b_nc_n.   \]
For $1 \leq i \leq n$,
\[ d_n(c_ia_ib_i \dots c_ia_ib_i \otimes 1) = c_ia_ib_i \dots c_ia_i \otimes b_i + c_ia_ib_i \dots c_ia_ib_ia_{i-1} \otimes b_{i-1}c_{i-1}   \]
\[ d_n(b_ic_ia_i \dots b_ic_ia_i \otimes 1) = b_ic_ia_i \dots b_i \otimes c_ia_i  \]
\[ d_n(b_ic_ia_i \dots b_ic_ia_ib_ia_{i-1} \otimes 1) = b_ic_ia_i \dots b_ic_ia_ib_i \otimes a_{i-1} \]
\[ d_n(c_{i-1}c_ia_ib_i \dots c_ia_ib_ia_{i-1} \otimes 1) = c_{i-1}c_ia_ib_i \dots c_ia_ib_i \otimes a_{i-1},  \]
except for $2 \leq i \leq n$,
\[ d_3(c_{i-1}c_ia_ib_ia_{i-1} \otimes 1) = c_{i-1}c_ia_ib_i \times a_{i-1} - c_{i-1}a_{i-1}b_{i-1}c_{i-1}a_{i-1} \otimes 1.  \]
For $1 \leq i \leq n-1$,
\[ d_n(c_{i-1}c_ia_ib_i \dots c_ia_ib_ic_ia_i \otimes 1) = c_{i-1}c_ia_ib_i \dots c_ia_ib_i \otimes c_ia_i  \]
\[ d_n(c_{n-1}c_na_nb_n \dots c_na_nb_nc_n \otimes 1) = c_{n-1}c_na_nb_n \dots c_na_nb_n \otimes c_n.  \]
\end{proposition}
\begin{proposition}
When $n$ is even,
\[ d_n : C_n \otimes B_n \rightarrow C_{n-1} \otimes B_n  \]
are given by
\[ d_n(a_nb_nc_n \dots a_nb_nc_na_n \otimes 1) = a_nb_nc_n \dots a_nb_nc_n \otimes a_n  \]
For $1 \leq i \leq n-1$,
\[ d_n(c_ia_ib_i \dots c_ia_ib_ic_ia_i \otimes 1) = c_ia_ib_i \dots c_ia_ib_i \otimes c_ia_i.  \]
\[ d_n(c_na_nb_n \dots c_na_nb_nc_n \otimes 1) = c_na_nb_n \dots c_na_nb_n \otimes c_n  \]
For $1 \leq i \leq n$,
\[ d_n(b_ic_ia_i \dots b_ic_ia_ib_i \otimes 1) = b_ic_ia_i \dots b_ic_ia_i \otimes b_i + b_ic_ia_ib_i \dots c_ia_ib_ia_{i-1} \otimes b_{i-1}c_{i-1}.   \]
For $1 \leq i \leq n-2$,
\[ d_n(c_{i-1}c_ia_ib_i \dots c_ia_ib_i \otimes 1) = c_{i-1}c_ia_ib_i \dots c_ia_i \otimes b_i\] 
\[+ c_{i-1}c_ia_ib_i \dots c_ia_ib_ia_{i-1} \otimes b_{i-1}c_{i-1}. \]
\[ d_n(c_{n-1}c_na_nb_n \dots c_na_nb_n \otimes 1) = c_{n-1}c_na_nb_n \dots c_n \otimes a_nb_n\] 
\[ + c_{n-1}c_na_nb_n \dots c_na_nb_na_{n-1} \otimes b_{n-1}c_{n-1}. \]
\end{proposition}
\begin{proof}
We prove the above propositions using the method of induction. Proposition 4.4.6 and 4.4.7 are influenced by proposition 4.4.2, 4.4.3, 4.4.4 and 4.4.5 where we construct the formula of $d_n$ for algebra $B_1$, $B_2$, $B_3$ and prove them using the method of induction.
\end{proof}

\newpage 
\section{Categories, Functors and Natural Transformations}
\subsection{Metacategory}
\subsubsection{Axioms for Categories}
  Without using any set theoretic approach and only using axioms we are going to describe categories viz \textit{metacategories}. For that we first require to define \textit{metagraph}.\\
  \begin{definition} A \textbf{metagraph} consists of objects say $a,b,c,\dots$ and arrows $f,g,h,\dots$ along with following two operations: \\
  \textit{Domain}, which assigns to each arrow $f$ an object $a$,\hspace{1mm}say $a=\dom f$\\
  \textit{Codomain}, which assigns to each arrow $f$ an object $b$,\hspace{1mm}say $b=\codom f$\\
  \end{definition}
  We can display these operations by means of an arrow diagram as follows\\
\[  \begin{tikzcd}
     a \arrow{r}{f} & b
  \end{tikzcd}\]
  Now we define a \textbf{metacategory}.\\
 \begin{definition} A \textbf{metacategory} is a metagraph with additional two operations viz.\\
  \textit{Identity}, which assigns to each object $a$ an arrow $Id_a = 1_a : a \rightarrow a$\\
  \textit{Composition}, which assigns to each pair $\langle g,f \rangle$ of arrows,\hspace{1mm}an arrow $g \circ f:\dom f \rightarrow \codom g$, provided $\dom g = \codom f$\end{definition} 
  We can represent this operation by means of a diagram,\\  
  \[
  \begin{tikzcd}
   & b \arrow{dr}{g}     \\
   a \arrow{ur}{f} \arrow{rr}{g \circ f}  && c
   \end{tikzcd}
   \]
   These operations in a \textit{metacategory} are subject to the following axioms:\\
   \textit{Associativity}, for given objects and arrows in the following configuration\\
   \[  \begin{tikzcd}
    a \arrow{r}{f} & b \arrow{r}{g} & c \arrow{r}{h} & d
    \end{tikzcd}
   \]
  one always has the equality\\
  \begin{equation}
            h \circ( g \circ f) = (h \circ g) \circ f 
  \end{equation}
  This axiom assures that the \textit{associative law} holds for the operation of composition whenever it makes sense.\\
%  \[ \begin{tikzcd}
 %  a \arrow{r}{h \circ(g \circ f)=(h \circ g) \circ f} \arrow{rd}{g \circ f} \arrow{d}{f} & d \arrow{u}{h} \\
  % b \arrow{ur}{h \circ g} \arrow{r}{g}  &c 
   %\end{tikzcd} 
%   \]
\textit{Unit law}, for all arrows $f : a \rightarrow b$ and $g:b \rightarrow c$ composition with the identity arrow gives\\
\begin{equation} 1_b \circ f=f  \hspace{10mm} \text{and} \hspace{10mm}  g \circ 1_b=g       \end{equation}
  This axiom assures that the identity arrow $1_b$ of each object $b$ acts as an identity for the operation of composition whenever it makes sense.
  
  %\[ \begin{tikzcd}
  %a \arrow{r}{f} \arrow {f}  & b \arrow{d}{1_b} \searrow{g}\\
   %                               b \arrow{r}{g}   &c                  
  %\end{tikzcd}
  % \]

\vspace{10mm} 

\begin{example} Let us give an example of \textit{metacategory}. We take \textit{metacategory of sets} where \textit{objects} are all sets and \textit{arrows} are the function defined on them (set functions) along with identity function and composition of functions (Usual composition). For example a function $f: X \rightarrow Y$ consists of $X$ as domain, $Y$ as codomain and with a suitable set of ordered pairs $\langle x,fx \rangle$ which assigns for each $x \in X$, an element $fx \in Y$.\\ 
 Similarly we can define \textit{metacategory of all groups} where objects are all groups $G,H,K,\dots$ and arrows are those functions $f:G \rightarrow H$ which are homomorphism of groups.\end{example}
 % Definition of matecategories through arrows only and stating three axioms based on arrows to get a metacategory.
 \subsection{Category}
 Now a category will mean any interpretation of category axioms (Associativity and Unital law) within Set theory thus making it distinguished from \textit{metacategories}.\\
 \begin{definition} A \textbf{category} $\mathcal{C}$ consists of three ingredients: a class obj($\mathcal{C}$) of \textit{objects}, a set of \textit{morphisms} $Hom(A,B)$ for every ordered pair $(A,B)$ of objects and \textit{composition} $Hom(A,B) \times Hom(B,C) \rightarrow Hom(A,C)$ denoted by $(f,g) \mapsto g \circ f=gf$, for every ordered triple $A,B,C$ of objects.
 \end{definition}
 These ingredients are subject to following axioms:\\
 \begin{itemize}
\item the Hom sets are pairwise disjoint i.e. each $f \in Hom(A,B)$ has a unique \textit{Domain} $A$ and a unique \textit{Codomain} $B$.
\item for each object $A$,there exists an \textit{identity morphism} $1_A \in Hom(A,A)$ such that $f \circ 1_A=f1_A=f$ and $1_B \circ f=1_Bf=f$; for each $f : A \rightarrow B$
\item composition is associative i.e. for given morphisms\\ 
\begin{tikzcd} A \arrow{r}{f} & B \arrow{r}{g} & C \arrow{r}{h} & D \end{tikzcd}, we have $h(gf)=(hg)f$.
\end{itemize}
\vspace{5mm}

\begin{example}
 The category $\mathcal{C(X)}$ of partially ordered set $\mathcal{X}$, where $\obj \mathcal{C(X)}$ =\hspace{1mm} $\mathcal{X}$,\hspace{1mm}
 $Hom_\mathcal{C(X)}(i,j)$ consists of a unique element say $l_j^i$ if $i \leq j$ and is $\emptyset$, otherwise.\hspace{1mm}
 Composition is given by $l_k^jl_j^i=l_k^i$ when $ i \leq j \leq k$, because of reflexivity we have $1_i=l_i^i$ and composition makes sense because $\leq$ is transitive.
 \end{example}
 \begin{example}
 Monoid is a category of one object. Thus each monoids is determined by the set of all it's arrows, by the identity arrow and by the rule of composition of arrows. Since any two arrows have a composite so monoids may be described as a set $X$ with a binary operation $X \times X \rightarrow X$ which is associative and has an identity.
 \end{example}
 \begin{example}
  A group is a category of one object and morphisms are homomorphisms and every morphism has a (two-sided) inverse under composition.
  \end{example}
  \begin{example}
  \textbf{Grp}, category of all groups with morphism as group homomorphisms.
  \end{example}
  \begin{example}
  \textbf{Ab}, category of all abelian groups with morphism as group homomorphisms. 
  \end{example}
  \begin{example}
  \textbf{Rng}, category of rings with object as all rings with morphism as homomorphism (preserving units) between two rings.
  \end{example}
  \begin{example}
 \textbf{Top}, category of topological spaces where \textit{$\obj$(\textbf{Top})} is all  topological spaces and morphism are continuous maps between them.
 \end{example}
 \begin{example}
  \textbf{Toph}, where \textit{$\obj$(\textbf{Toph})} is all topological spaces and morphism given by \\
  $Hom_{Toph}(X,Y)$ \hspace{2mm}=\hspace{2mm}\{the set of homotopy classes of continuous mapping from $X$ to $Y$\} .
  \end{example}
  \begin{example}
 \textbf{Top$_\ast$}, in this category\hspace{1mm} \textit{$\obj$(\textbf{$Top_\ast$})} consists of all ordered pairs $(X,x_0)$ where $X$ is a non-empty topological spaces and $x_0 \in X$ and morphisms are continuous maps $f:(X,x_0) \rightarrow (Y,y_0)$ where $f:X \rightarrow Y$ with $f(x_0)=y_0$ i.e. base point preserving maps.
\end{example}

\subsection{Functor}
\begin{definition}
If $\mathcal{C}$ and $\mathcal{D}$ are two categories,then a \textbf{functor} $T:\mathcal{C} \rightarrow \mathcal{D}$ is a function such that:
\begin{enumerate}
\item if $a \in \obj(\mathcal{C})$, then $T(a) \in \obj(\mathcal{D})$.\\
\item if $f: a \rightarrow a'$ in $\mathcal{C}$, then $T(f): T(a) \rightarrow T(a')$ in $\mathcal{D}$.\\
\item if \begin{tikzcd} a \arrow{r}{f} & a' \arrow{r}{g} & a'' \end{tikzcd} in $\mathcal{C}$, then \begin{tikzcd} T(a) \arrow{r}{T(f)} & T(a') \arrow{r}{T(g)} & T(a'') \end{tikzcd} in $\mathcal{D}$ and
\[ T(gf)= T(g)T(f).  \]
\item $T(1_a)=1_{T(a)}$ \hspace{2mm} for each $a \in \obj(\mathcal{C})$.\\
We call this functor a \textbf{covariant functor}.
\end{enumerate}
\end{definition}
\begin{example}
 Power set functor $\mathcal{P}:\textbf{Set} \rightarrow \textbf{Set}$. The object function assigns to each set $X$ the usual power set $\mathcal{P}X$, where elements are all subsets $S \subset X$; it's arrow function assigns to each $f: X \rightarrow Y$ the map $\mathcal{P}f: \mathcal{P}X \rightarrow \mathcal{P}Y$ which sends each $S \subset X$ to its image $fS \subset Y$. Since both $\mathcal{P}(1_X)=1_{\mathcal{P}X}$ and $\mathcal{P}(gf)=\mathcal{P}(g)\mathcal{P}(f)$, this defines a functor $\mathcal{P}:\textbf{Set} \rightarrow \textbf{Set}$.
 \end{example}
 \begin{example}
 If $\mathcal{C}$ is a category, then the \textbf{identity functor} $1_\mathcal{C}: \mathcal{C} \rightarrow \mathcal{C}$ is defined by $1_\mathcal{C}A=A$ for all objects $A$ and $1_\mathcal{C}f=f$ for all morphisms $f$. It is the simplest example of a functor.
 \end{example}
 \begin{example}
 If $\mathcal{C}$ is a category and $A \in \obj(\mathcal{C})$, then the \textbf{$\Hom$ functor}\hspace{1mm} $T_A:\mathcal{C}\rightarrow \textbf{Sets}$, denoted by $\Hom(A,-)$ is defined by
\[ T_A(B)=\Hom(A,B) \hspace{2mm}\text{for all}\hspace{2mm} B \in \obj(\mathcal{C}) \]
and if $f:B \rightarrow B'$ in $\mathcal{C}$, then  $T_A(f): \Hom(A,B) \rightarrow \Hom(A,B')$ is given by 
\[  T_A(f): h \mapsto fh\]
We call $T_A(f)=\Hom(A,f)$ the induced map and denoted it by $f_\ast$.
\end{example}
\begin{example}
Singular homology in a given dimension $n \in \mathbb{N}$ assigns to each topological space $X$ an abelian group $H_n(X)$, the $n$-th homology group of $X$ and also to each continuous map $f:X \rightarrow Y$ of spaces a corresponding homomorphism $H_n(f): H_n(X) \rightarrow H_n(Y)$ of groups thus making $H_n$ a functor \textbf{Top} $\rightarrow$ \textbf{Ab}.
\end{example}
\begin{example}
 Continuing with the previous example homotopic maps $f,g:X \rightarrow Y$ yield the same homomorphism $H_n(X) \rightarrow H_n(Y)$, so we can also regard $H_n$ as a functor \textbf{Toph} $\rightarrow$ \textbf{Grp}, defined on homotopy category.
 \end{example}
 \begin{example}
 For any commutative ring $K$, we can attach the general linear group of $n \times n$ non-singular matrices with entries from $K$ viz $\GL_n(K)$. For any ring homomorphism $f:K \rightarrow K'$ we can associate a homomorphism $\GL_n(f):\GL_n(K) \rightarrow \GL_n(K')$ of groups and thus defines the functor $\GL_n$:\textbf{CRng} $\rightarrow$ \textbf{Grp}, where \textbf{CRng} is the category of commutative rings.
\end{example}
\begin{definition}
 A functor which simply forgets some or all of the structure of an algebraic object is known as \textbf{forgetful functor}. Let us give an example to clarify this:
\end{definition}
\begin{example}
A forgetful functor $F$:\textbf{Grp} $\rightarrow$ \textbf{Set} assigns to each group $G$, the set $F\hspace{1mm}G$ of its elements (thus forgetting the multiplication and hence the group structure) and assigns to each morphism $f:G \rightarrow G'$ of groups the same function $f$, regarded just as a function between sets.
\end{example}
\begin{definition}
A functor $T:A \rightarrow B$ is  \textbf{full} when to every pair $a,a'$ of objects of $A$ and to every arrow $g:Ta \rightarrow Ta'$ of $B$, there is an arrow $f:a \rightarrow a'$ of $A$ with $g=T\hspace{1mm}f$.\\
It is very clear that composite of two full functors is also a full functor.
\end{definition}
\begin{definition}
A functor $T:\mathcal{A} \rightarrow \mathcal{B}$ is \textbf{faithful} if for all $a,b \in \obj(\mathcal{A})$, the functions $\Hom_\mathcal{A}(a,b) \rightarrow \Hom_\mathcal{B}(Ta,Tb)$ given by $f \mapsto T\hspace{1mm}f$ are injections. \\
Again composite of two faithful functors is a faithful functor.
\end{definition}
\begin{example}
The previous forgetful functor \textbf{Grp} $\rightarrow$ \textbf{Set} is a faithful functor. However it is not a full functor.
\end{example}
\begin{definition}
If $\mathcal{C}$ and $\mathcal{D}$ are two categories, then a \textbf{contravariant functor} $T:\mathcal{C} \rightarrow \mathcal{D}$ is a function such that:
\begin{enumerate}
\item if $a \in \obj(\mathcal{C})$, then $T(a) \in \obj(\mathcal{D})$.\\
\item if $f: a \rightarrow a'$ in $\mathcal{C}$, then $T(f): T(a') \rightarrow T(a)$ in $\mathcal{D}$.\\
\item if \begin{tikzcd} a \arrow{r}{f} & a' \arrow{r}{g} & a'' \end{tikzcd} in $\mathcal{C}$, then \begin{tikzcd} T(a'') \arrow{r}{T(g)} & T(a') \arrow{r}{T(f)} & T(a) \end{tikzcd} in $\mathcal{D}$ and
\[ T(gf)= T(f)T(g).  \]
\item $T(1_a)=1_{T(a)}$ \hspace{2mm} for each $a \in \obj(\mathcal{C})$.
\end{enumerate}
\end{definition}
\begin{example}
Let us give an example of a contravariant functor.\\ 
 If $\mathcal{C}$ is a category and $B \in \obj(\mathcal{C})$, then the \textbf{contravariant $\Hom$ functor} $T^B:\mathcal{C}\rightarrow \textbf{Sets}$ denoted by $\Hom(-,B)$, defined as
\[  T^B(C)=\Hom(C,B) \hspace{3mm} \text{for all} \hspace{2mm} C \in \obj(\mathcal{C})\]
and if $f:C \rightarrow C'$ in $\mathcal{C}$, then $T^B(f):\Hom(C',B) \rightarrow \Hom(C,B)$ is given by
\[  T^B(f): h \mapsto hf\]
\end{example}
\subsection{Natural Transformation}
\begin{definition}
Given two functors $S,T:\mathcal{C} \rightarrow \mathcal{B}$, a \textbf{natural transformation} is a function $\tau : S \rightarrow T$ which assigns to each object $c$ in $\mathcal{C}$, an arrow $\tau_c=\tau \hspace{1mm}c: Sc \rightarrow Tc$ of $\mathcal{B}$ in such a way that every arrow $f: c \rightarrow c'$ in $\mathcal{C}$ yields a commutative diagram,\\
\[\begin{tikzcd}
 S \hspace{1mm}c \arrow{r}{\tau \hspace{1mm}c} \arrow{d}{S \hspace{1mm}f}
  & T \hspace{1mm}c \arrow{d}{T \hspace{1mm}f}\\
  S \hspace{1mm}c' \arrow{r}{\tau \hspace{1mm}c'} & T \hspace{1mm}c'
\end{tikzcd} \]
\end{definition}
When it holds, we also say that $\tau_c: Sc \rightarrow Tc$ is \textit{natural} in $c$.\hspace{1mm}A natural transformation is also called a \textit{morphism of functors}.
\begin{example}
Let $\det_K M$ be the determinant of the $n \times n$ matrix $M$ with entries from the commutative ring $K$, while $K^\ast$ denotes the group of units of $K$. Thus $M$ is non-singular when $\det_K M$ is a unit and $\det_K$ is a morphism $\GL_n K \rightarrow K^\ast$ of groups. As the determinant is defined by the same formula for all rings $K$, each morphism $f:K \rightarrow K'$ of rings leads to a commutative diagram
\[ \begin{tikzcd}
  \GL_nK \arrow{r}{\det_K} \arrow{d}{\GL_n f}
  & K^\ast \arrow{d}{f^\ast}\\
 \GL_n K' \arrow{r}{\det_{K'}} & K'^\ast
\end{tikzcd}
\]
this states that the transformation $ \det :\GL_n \rightarrow (\hspace{1mm})^\ast $ is natural between two functors \textbf{CRng} $\rightarrow$ \textbf{Grp}.
\end{example}
\begin{example}
For each group $G$, the projection $p_G: G \rightarrow G/[G,G]$ to the factor commutator group defines a transformation $p$ from the identity functor on \textbf{Grp} to the factor commutator functor \textbf{Grp} $\rightarrow$ \textbf{Ab} $\rightarrow$ \textbf{Grp}. Moreover, $p$ is natural because each group homomorphism $f: G \rightarrow H$ defines the evident homomorphism $f'$ for which we get the following commutative diagram
\[ \begin{tikzcd}
G \arrow{r}{p_G} \arrow{d}{f}
  & \frac{G}{[G,G]} \arrow{d}{f'}\\
  H \arrow{r}{p_H} & \frac{H}{[H,H]}
  \end{tikzcd}
  \] 
\end{example}
\subsection{Product of Categories}
We are going to construct a new category $\mathcal{B} \times \mathcal{C}$ from given categories $\mathcal{B}$ and $\mathcal{C}$ in the following way:\\
An object of  $\mathcal{B} \times \mathcal{C}$ is a pair $\langle b,c \rangle$ where $b\in \mathcal{B}$ and $c \in \mathcal{C}$; an arrow $\langle b,c \rangle \rightarrow \langle b',c' \rangle$ of  $\mathcal{B} \times \mathcal{C}$ is a pair $\langle f,g \rangle$ of arrows $f:b \rightarrow b'$ and $g:c \rightarrow c'$ and the composite of two such arrows
\[ \begin{tikzcd}
 \langle b,c \rangle \arrow{r}{\langle f,g \rangle} & \langle b',c' \rangle \arrow{r}{\langle f',g' \rangle} & \langle b'',c''\rangle
\end{tikzcd}\]
is defined by \begin{equation}\langle f',g'\rangle \circ \langle f,g \rangle = \langle f' \circ f , g' \circ g \rangle  \end{equation}
\begin{definition}
Functors $P$,$Q$\\
\[\begin{tikzcd} \mathcal{B} \times \mathcal{C} \arrow{r}{P} & \mathcal{B}  \end{tikzcd} ;\begin{tikzcd}  \mathcal{B} \times \mathcal{C} \arrow{r}{Q} & \mathcal{C} \end{tikzcd} \]
called the \textit{projections} of product and are defined on arrows by,
\[ P\langle f,g\rangle=f ,\hspace{3mm} Q\langle f,g\rangle=g. \]
\end{definition}
\vspace{2mm}
\textbf{Universality of $P$ and $Q$}:\\
Given any category $\mathcal{D}$ and two functors,
\[\begin{tikzcd} \mathcal{D} \arrow{r}{R} & \mathcal{B}  \end{tikzcd} ;\begin{tikzcd}  \mathcal{D} \arrow{r}{T} & \mathcal{C} \end{tikzcd} \]
there is a unique functor $F:\mathcal{D} \rightarrow \mathcal{B} \times \mathcal{C}$ with $P\hspace{1mm}F=R$ and $Q\hspace{1mm}F=T$; these two conditions require that $Fh$ for any arrow $h \in \mathcal{D}$, must be $\langle Rh,Th \rangle$.\\ 
  Conversely, this value for $Fh$ makes $F$ a functor with the required properties.\\
 %\[\begin{tikzcd} 
 % \mathcal{D} \arrow{dl}{R}               \arrow{d}{F}                                  \arrow{dr}{T}\\
 %             \mathcal{B} \arrow{l}{P} &  \mathcal{B}\times \mathcal{C} \arrow{r}{Q} &  \mathcal{C}
 %end{tikzcd}\]

 Two functors $U:\mathcal{B} \rightarrow \mathcal{B'}$ and $V: \mathcal{C} \rightarrow \mathcal{C'}$ have a product $ U \times V: \mathcal{B} \times \mathcal{C} \rightarrow \mathcal{B'} \times \mathcal{C'}$ which can be described explicitely on objects and arrows as 
 \[ (U \times V)\langle b,c \rangle=\langle Ub,Vc \rangle ;\hspace{2mm} (U \times V)\langle f,g \rangle= \langle Uf,Vg\rangle\]
 The uniqueness of the functor $U \times V$ can be seen from the following commutative diagrams,
 \[ \begin{tikzcd} 
 \mathcal{B} \times \mathcal{C} \arrow{r}{P} \arrow{d}{U \times V} & \mathcal{B} \arrow{d}{U}\\
 \mathcal{B'} \times \mathcal{C'} \arrow{r}{P'} & \mathcal{B'} 
 \end{tikzcd} \hspace{8mm}  \begin{tikzcd} 
 \mathcal{B} \times \mathcal{C} \arrow{r}{Q} \arrow{d}{U \times V} & \mathcal{C} \arrow{d}{V}\\
 \mathcal{B'} \times \mathcal{C'} \arrow{r}{Q'} & \mathcal{C'} 
 \end{tikzcd}\]
 Therefore the product $\times$ is a pair of functions: To each pair $\langle \mathcal{B},\mathcal{C}\rangle$ of categories a new category $\mathcal{B} \times \mathcal{C}$; to each pair of functors $\langle U,V\rangle$ a new functor $U \times V$, also as $(U' \times V')\circ (U \times V)=U'U \times V'V$ thus the composite $U' \circ U$ and $V' \circ V$ are defined. Hence the operation $\times$ itself is a functor $\times:\textbf{Cat} \rightarrow \textbf{Cat}$.
 \begin{definition}
 Functor $G: \mathcal{B} \times \mathcal{C} \rightarrow \mathcal{D}$ from a product category is called a \textbf{bifunctor} on $\mathcal{B}$ and $\mathcal{C}$.
 \end{definition}
 \begin{example}
 The cartesian product $X \times Y$ of two sets $X$ and $Y$ is a bifunctor \textbf{Set} $\times$ \textbf{Set} $\rightarrow$ \textbf{Set}.
 \end{example}
 We now state a proposition regarding the product categories and bifunctors:\\
 \begin{proposition}
 Let $\mathcal{B}$, $\mathcal{C}$ and $\mathcal{D}$ be categories. For all objects $c \in \mathcal{C}$ and $b \in \mathcal{B}$, let
 \[ L_c:\mathcal{B} \rightarrow \mathcal{D},\hspace{3mm} M_b: \mathcal{C} \rightarrow \mathcal{D}\]
 be functors such that $M_b(c)=L_c(b)$ for all $b,c$. Then there exists a bifunctor $S:\mathcal{B} \times \mathcal{C} \rightarrow \mathcal{D}$ with $S(-,c)=L_c$ for all $c$ and $S(b,-)=M_b$ for all $b$ if and only if for every pair of arrows $f:b \rightarrow b'$ and $g:c \rightarrow c'$ one can obtain
  \begin{equation} 
  M_{b'}g \circ L_cf=L_{c'}f \circ M_bg
 \end{equation} 
 \end{proposition}
 \begin{proof}
 For the necessary part what we need to show is that
 \[ S(b',g)S(f,c)=S(f,c')S(b,g)\]
 and this condition can be viewed as a commutative diagram as follows
 \[\begin{tikzcd} 
  S(b,c) \arrow{r}{S(b,g)} \arrow{d}{S(f,c)} & S(b,c') \arrow{d}{S(f,c')} \\
  S(b',c) \arrow{r}{S(b',g)}  & S(b',c')
      \end{tikzcd}\]
     and this is just condition (85) rewritten, and we can achieve the above part simply applying functor $S$ to the following equation which can be found by writing $b$, $c$ for their corresponding arrows and from definition (84) of composite in $\mathcal{B} \times \mathcal{C}$,
     \[ \langle b',g \rangle \circ \langle f,c \rangle=\langle b'f,gc \rangle = \langle f,g \rangle =\langle fb,c'g \rangle=\langle f,c'\rangle \circ \langle b,g \rangle \]
  Conversely, one can verify that given all $L_c$ and $M_b$ condition (85) defines $S(f,g)$ for every pair $f,g$ as well as the definition (84) yields a bifunctor $S$ with its required properties.
  \end{proof}
  Our next aim is to define a natural transformation between bifunctors $S,S':\mathcal{B} \times \mathcal{C} \rightarrow \mathcal{D}$.\hspace{2mm}Let $\alpha$ be that function which assigns to each pair of objects $b \in \mathcal{B}$, $c \in \mathcal{C}$ an arrow
  \begin{equation} \alpha(b,c): S(b,c) \rightarrow S'(b,c) \end{equation}
  in $\mathcal{D}$. We call $\alpha$ \textit{natural} in $b$ if for each $c \in \mathcal{C}$ the components $\alpha(b,c)$ for all $b$ define
  \[ \alpha(-,c): S(-,c) \rightarrow S'(-,c)\]
  a natural transformation of functors $\mathcal{B} \rightarrow \mathcal{D}$.\\ 
  Based on such definition we state another proposition which can be verified easily.
  \begin{proposition}
  For bifunctors $S,S'$, the function $\alpha$ written in (86) is a natural transformation $\alpha:S \rightarrow S'$ of bifunctors if and only if $\alpha(b,c)$ is natural in $b$ for each $c \in \mathcal{C}$ and natural in $c$ for each $b \in \mathcal{B}$.
\end{proposition}

\newpage
\section{Some Topics in Homological Algebra}
\subsection{Some useful definitions }
\begin{definition}[Chain complex]
A \textbf{chain complex} $(C_.,d_.)$ is a sequence of abelian groups (or modules) and homomorphisms
\[ \begin{tikzcd}
 \dots \arrow{r}{d_{n+1}} & C_n \arrow{r}{d_n} & C_{n-1} \arrow{r}{d_{n-1}} & C_{n-2} \arrow{r} & \dots
\end{tikzcd} \]
with the property $d_n \circ d_{n+1}= d_nd_{n+1} = 0$ for all $n$. Homomorphisms $d_n$ are called \textit{differentials}.
\end{definition}
\begin{definition}[Exact sequence]
A chain complex defined above is called \textbf{exact} if
\[ \ker d_n = \Ima d_{n+1} \hspace{4mm}\text{for all}\hspace{1mm}n.  \]
\end{definition}
\begin{definition}[Co-chain complex]
A \textbf{co-chain complex} $(C^.,d^.)$ is a sequence of abelian groups (or modules) and homomorphisms
\[ \begin{tikzcd}
 \dots \arrow{r}{d^{n-1}} & C^n \arrow{r}{d^n} & C^{n+1} \arrow{r}{d^{n+1}} & C^{n+2} \arrow{r} & \dots
\end{tikzcd} \]
with the property $d^n \circ d^{n-1} = 0$ for all $n$.
\end{definition}
\begin{definition}[Homology and Cohomology]
A \textbf{homology} of a chain complex is the group $H_n(C_.) = \ker d_n / \Ima d_{n+1}$.\\ 
Similarly a \textbf{cohomology} of a co-chain complex is the group\\   $H^n(C^.) = \ker d^n / \Ima d^{n-1}$.
\end{definition}
\begin{definition}
$C_n$ are called \textbf{n-dimensional chains}. Similarly $C^n$ are called \textbf{n-dimensional co-chains}. $\ker d_n = Z_n$ are called \textbf{n-dimensional cycles}. Similarly $\ker d^n = Z^n$ are called \textbf{n-dimensional co-cycles}.\\  $\Ima d_{n+1} = B_n$ are called \textbf{boundaries} and $\Ima d^{n-1} = B^n$ are called \textbf{co-boundaries}.
\end{definition}
\begin{definition}[Chain homotopy]
A chain map $\phi : (C_.,d_.^C) \rightarrow (D_.,d_.^D)$ is \textbf{null homotopic} if there are maps $h_n : C_n \rightarrow D_{n+1}$ such that $\phi = d_{n+1}^D h_n + h_{n-1}d_n^C$. The maps $\{h_n\}$ are called a chain contraction.\\ 
We say that two chain maps $\phi, \psi : C_. \rightarrow D_.$ are \textbf{chain homotopic} if their difference is null homotopic.
\end{definition}
Next we will state an important theorem without the proof.
\begin{theorem}
Let 
\[ \begin{tikzcd}
0 \arrow{r} & K_. \arrow{r}{f} & L_. \arrow{r}{g} & M_. \arrow{r} & 0
\end{tikzcd}   \]
be a short exact sequence of chain complexes $K_.,L_.,M_.$. Then the long sequence of their homologies
\[ \begin{tikzcd}
\dots \arrow{r} & H_n(K_.) \arrow{r}{H_n(f)} & H_n(L_.) \arrow{r}{H_n(g)} & H_n(M_.) \arrow{r}{\delta} & H_{n-1}(K_.) \\
             \arrow{r}{H_{n-1}(f)} & H_{n-1}(L_.) \arrow{r}{H_{n-1}(g)} & H_{n-1}(M_.) \arrow{r} & \dots
\end{tikzcd}\]
where $\delta$ is the connecting homomorphism, is exact.
\end{theorem}
\begin{proof}
For proof of the above theorem, we refer to theorem 6.10 of [6].
\end{proof}

\subsection{Free resolution}
\begin{definition}
A \textbf{resolution} of a module $M$ is an exact sequence of modules $M_n$ with differentials of the form
\[ \begin{tikzcd}
\dots \arrow{r}{d_{n+2}} & M_{n+1} \arrow{r}{d_{n+1}} & M_n \arrow{r}{d_n} & M_{n-1} \arrow{r}{d_{n-1}} & \dots \\
 \hspace{36mm}\dots  \arrow{r}{d_2} & M_1 \arrow{r}{d_1} & M_0 \arrow{r} & M \arrow{r} & 0.
\end{tikzcd} \]
We call a resolution \textbf{free} if each of $M_n$ are free, and \textbf{projective} if each of $M_n$ are projective.
\end{definition}
\begin{remark}
As each free module is a projective module[21], so any free resolution is a projective resolution.
\end{remark}
\subsection{One example of free resolution}
\begin{example}[Koszul complex in commutative algebra]
Let $R$ be a commutative ring. Let $M = R^k$ be a finitely generated $R$-module. Let $K_n = \Lambda _R^n M$ be its exterior algebra. Let us fix an element $m \in M$. Left exterior multiplication by $m$ produce a differential on $K_n$ since it's square is equal to $0$. We call this complex \textit{Koszul complex} [20] and denote it by $K(m)_.$.\\ 
\begin{remark}
If $e_1,e_2,\dots,e_k$ is a basis of $M$ then the set of exterior products $e_{i_1}\wedge e_{i_2} \wedge \dots \wedge e_{i_n}$ with $1 \leq i_1 < i_2 < \dots < i_n \leq n$ form a basis of $\Lambda _R^n M$.
\end{remark}
Let $R= \mathbb{K}[x,y,z]$. Our claim is that Koszul complex for $m = (x,y,z) \in M = R^3$ is a free resolution.\\ 
Now,
\[ \Lambda _R^0(R^3) = R\cdot 1 = R    \]
\[ \Lambda _R^1(R^3) = Re_1 \oplus Re_2 \oplus Re_3 \simeq R^3      \]
\[ \Lambda _R^2(R^3) = R(e_1 \wedge e_2) \oplus R(e_2 \wedge e_3) \oplus R(e_1 \wedge e_3) \simeq R^3  \]
\[ \Lambda _R^3(R^3) = R(e_1 \wedge e_2 \wedge e_3) \simeq R.    \]
So our chain complex $K(m)_.$ is of the form
\[ \begin{tikzcd} 
0 \arrow{r}{h_{-1}} & R \arrow{r}{h_0} & R^3 \arrow{r}{h_1} & R^3 \arrow{r}{h_2} & R \arrow{r}{h_3} & 0.
\end{tikzcd} \]
Next
\[ h_0 : R \rightarrow R^3   \]
is given by
\[  r \mapsto m \wedge r   \]
\[   =(xe_1+ye_2+ze_3) \wedge r  \]
\[   = xre_1 +yre_2 + zre_3     \]
Hence $\ker h_0 = \{ r \hspace{2mm} |\hspace{2mm} h_0(r) = 0 \}$ is given by conditions $xr=yr=zr=0$, which implies that $\ker h_0 = \{ 0\}$ and it is equal to $\Ima h_{-1}$.\\ 
We also observed that $\Ima h_0 = \langle (x,y,z) \rangle$ i.e. the elements of $\Ima h_0$ are generated by $x,y,z$.\\
Next,
\[ h_1 : R^3 \rightarrow R^3    \]
given by
\[ r_1e_1 + r_2e_2 + r_3e_3 \mapsto m \wedge (r_1e_1 + r_2e_2 + r_3e_3)   \]
\[ =(xe_1 +ye_2 +ze_3) \wedge (r_1e_1 + r_2e_2 + r_3e_3)    \]
\[ = (xr_2-yr_1)e_1 \wedge e_2 + (yr_3-zr_2)e_2 \wedge e_3 + (xr_3-zr_1)e_1 \wedge e_3.  \]
Our job is to find $\ker h_1$ which is given by the relations,
\[  xr_2 - yr_1 = 0  \]
\[  yr_3 - zr_2 = 0    \]
\[  xr_3 - zr_1 = 0   \]
from these three relations we get $(r_1 : r_2) = (x :y)$; $(r_2:r_3)=(y:z)$; $(r_1:r_3)=(x:z)$ which all together implies $(r_1:r_2:r_3)=(x:y:z)$. Hence $\ker h_1 = \langle (x,y,z) \rangle = \Ima h_0$.\\ 
We get that $\Ima h_1 = \{ r_1 : r_2 : r_3 \}$ where $r_1 = \langle (x,y) \rangle$, $r_2 = \langle (y,z) \rangle$, $r_3 = \langle (x,z) \rangle$.\\ 
Next we have
\[  h_2 : R^3 \rightarrow R   \]
given by
\[  r_1(e_1 \wedge e_2) + r_2(e_2 \wedge e_3) + r_3(e_1 \wedge e_3) \mapsto m \wedge (r_1(e_1 \wedge e_2) + r_2(e_2 \wedge e_3) + r_3(e_1 \wedge e_3))  \]
\[ = (xr_2 - yr_3 + zr_1)e_1 \wedge e_2 \wedge e_3.  \]
So $\ker h_2$ is given by the relation $xr_2 - yr_3 + zr_1 = 0$. So we have $yr_3=xr_2+zr_1$. From this relation it is very clear that $r_3$ can not depend on $y$ as otherwise we will have $y^2$ term in left hand side but in right hand side there is no $y^2$ term. If $r_3$ only depends on $x$ that implies that $r_1$ must be only $0$, later we will see it's impossible. Similarly $r_3$ can not only depends on $z$ as otherwise $r_2$ must be only $0$ and we will reach the impossibility. The only option is that $r_3$ is depending on both $x,z$.\\ 
Next from the relation $zr_1 = yr_3 - xr_2$ using the similar argument we get that $r_1$ can not depend on $z$. Similarly if $r_1$ only depends on $x$ and only depends on $y$ then $r_3$ and $r_2$ must be only $0$ always respectively which is not possible. Also it makes clear that in the previous argument why $r_1$ is not only $0$. So only suitable option is that $r_1$ depends on both $x,y$.\\ 
Similarly from the relation $xr_2 = yr_3 - zr_1$ and using the similar arguments we get $r_2$ is depending on both $y,z$. Also none of $r_1,r_2,r_3$ depend on all of $x,y,z$ as otherwise we will have $x^2,y^2,z^2$ terms which can not be matched with other terms inside the relation.\\ 
All together it implies that $\ker h_2 = \{ r_1 :r_2:r_3 \}$ where $r_1 =\langle (x,y)\rangle$, $r_2=\langle (y,z)\rangle$ and $r_3 = \langle (x,z) \rangle$ and hence $\ker h_2 = \Ima h_1$.\\ 
Also from the map $h_2$ we understand that $\Ima h_2 = R$. Our $\ker h_3 = R$ and thus we get $\ker h_3 = \Ima h_2$. Hence we have obtained a short exact sequence
\[ \begin{tikzcd} 
0 \arrow{r}{h_{-1}} & R \arrow{r}{h_0} & R^3 \arrow{r}{h_1} & R^3 \arrow{r}{h_2} & R \arrow{r}{h_3} & 0.
\end{tikzcd} \]
implies it is a free resolution.
\end{example}
\newpage  
\section{Further Research Directions}
In this section, we mention about two applications of Anick resolution which will link Gr\"{o}bner bases with core homological algebra as well as with homotopical algebra.

\subsection{Construction of derived functor from Anick resolution}
\begin{definition}
Let $\mathcal{P}$, $\mathcal{Q}$ be abelian categories[5] and let $\mathcal{F} : \mathcal{P} \rightarrow \mathcal{Q}$ be a covariant functor. Let
\[ \begin{tikzcd}  
0 \arrow{r} & A \arrow{r} & B \arrow{r} & C \arrow{r} & 0
 \end{tikzcd}\]
 be a short exact sequence of objects in $\mathcal{P}$. Then $\mathcal{F}$ is \textbf{right exact} if
 \[ \begin{tikzcd} 
 \mathcal{F}(A) \arrow{r} & \mathcal{F}(B) \arrow{r} & \mathcal{F}(C) \arrow{r} & 0
 \end{tikzcd}\]
 is exact.
\end{definition}
\begin{example}
Let $\mathcal{C}$, $\mathcal{D}$ be categories of abelian groups. Let us fix an abelian group $H$. Then 
\[ \mathcal{F}_H : G \rightarrow G \otimes_{\mathbb{Z}} H \]
is an example of a right exact functor from \textbf{Ab} to \textbf{Ab}.
\end{example}
\begin{definition}
Let $\mathcal{F}$ be a right exact functor from category $\mathcal{C}$ to category $\mathcal{D}$. Then \textbf{higher derived functor} of $\mathcal{F}$ is a collection of functors $\mathcal{F} = \mathcal{F}_0$, $\mathcal{F}_1$, $\mathcal{F}_2$, $\dots$ such that if
\[ \begin{tikzcd}  
0 \arrow{r} & A \arrow{r} & B \arrow{r} & C \arrow{r} & 0
 \end{tikzcd}\]
 is a short exact sequence, where $A, B, C \in \mathcal{C}$, then
 \[ \begin{tikzcd}
 \dots \arrow{r} & \mathcal{F}_1(A) \arrow{r} & \mathcal{F}_1(B) \arrow{r} & \mathcal{F}_1(C) \arrow{r} & \mathcal{F}_0(A)\\ 
  \arrow{r} & \mathcal{F}_0(B) \arrow{r} & \mathcal{F}_0(C) \arrow{r} & 0
 \end{tikzcd}\]
 is a long exact sequence.
\end{definition}
Now we will show how to compute higher derived functors. Given $A \in \mathcal{C}$, our aim is to compute $\mathcal{F}_i(A)$. We find a projective resolution $(P_., d_.)$ of $A$.
\[ \dots \rightarrow P_2 \rightarrow P_1 \rightarrow P_0 \rightarrow A \rightarrow 0   \]
Applying functor $\mathcal{F}$ on the above resolution gives us a chain complex $( \mathcal{F}(P_.), \mathcal{F}(d_.))$. By definition,
\[ \mathcal{F}_i(A) = H_i(\mathcal{F}(P_.), \mathcal{F}(d_.))  \]
\[  \mathcal{F}_0(A) = \mathcal{F}(A). \]
\begin{lemma}
The functor $\mathcal{F}_i(A)$ doesn't depend on the resolution $(P_.,d_.)$.
\end{lemma}
\begin{proof}
For proof, we refer to [6].
\end{proof}
\begin{remark}
Short exact sequence gives rise to long exact sequence of homologies.
\end{remark}
This is stated in theorem 6.1.1. Hence $\mathcal{F}_i$ are derived functor of $\mathcal{F}$. 
\begin{definition}
Let $\mathcal{C}$ be a category of right $R$ modules where $R$ is an algebra over a field $\mathbb{K}$ and let $\mathcal{D}$ be category of vector spaces over $\mathbb{K}$. Let $N$ is a given left $R$-module. Then $\mathcal{F}_N : M \rightarrow M \otimes_R N$ is a right exact functor from $\mathcal{C}$ to $\mathcal{D}$. We compute derived functors of $\mathcal{F}_N$. They are called
\[ \mathcal{F}_{N,i}(M) = \Tor_i^R(M,N)  \]
\end{definition}
Now we will show how to compute some derived functors from our Anick resolution. Let $A$ be an augmented algebra over $\mathbb{K}$. This makes $\mathbb{K}$ a left and right module. Our aim is to construct the structure of 
\[ \Tor_i^A(\mathbb{K},\mathbb{K}) \]
We have our Anick resolution of $\mathbb{K}$
\begin{equation} \begin{tikzcd} 
\dots \arrow{r}{d_2} & C_1 \otimes_{\mathbb{K}} A \arrow{r}{d_1} & C_0 \otimes_{\mathbb{K}} A \arrow{r}{d_0} & \mathbb{K} \arrow{r} & 0
\end{tikzcd} \end{equation}
We tensor each element and differentials of the above resolution with $\mathbb{K}$ over $A$. Hence we get 
\[ ( C_i \otimes_{\mathbb{K}} A ) \otimes_A \mathbb{K} = C_i \otimes_{\mathbb{K}} \mathbb{K} \simeq C_i. \]
Hence we obtain the following sequence
\[ \begin{tikzcd} 
\dots \arrow{r}{d_3 \otimes_A \mathbb{K}} & C_2 \arrow{r}{d_2 \otimes_A \mathbb{K}} & C_1 \arrow{r}{d_1 \otimes_A \mathbb{K}} & C_0.
\end{tikzcd}\]
Let $c \in C_i \otimes 1 \simeq C_i$. Then $d_i(c)$ is of the form
\[ d_i(c) = \sum_l c^l \otimes \sigma ^l  \]
for some $c^l$ and $\sigma^l$. Then
\[ d_i \otimes_A \mathbb{K} (c) = \sum_l v^l \cdot \epsilon ( r^l)  \]
for some $v^l, r^l$ where $\epsilon: A \rightarrow \mathbb{K}$ is the augmentation map. In particular if $A$ is graded i.e. $A = \oplus_i A_i$ with $A_0 = \mathbb{K}$ then
\[ \epsilon|_{A_i} = 0 \hspace{2mm} \text{for} \hspace{2mm} i \neq 0.  \]
\begin{definition}
A resolution is said to be \textbf{minimal} if 
\[ d_i \otimes_A \mathbb{K} = 0 \hspace{2mm}\text{for all}\hspace{1mm}i.  \]
\end{definition}
\begin{corollary}
If we have a resolution of the form given in equation (87) and if it is minimal, then 
\[   \Tor_i^A(\mathbb{K},\mathbb{K}) \simeq C_i. \]
\end{corollary}
\begin{remark}
For algebra $B_1$, its Anick resolution is minimal. But for algebra $B_n$ for $n >1$, its Anick resolution is not minimal as $d_3 \otimes_A \mathbb{K} \neq 0$. 
\end{remark}
For more details about derived functors and minimal resolution we refer to [5], [6] and [17].
 
\subsection{Computation of $A_\infty$ algebra from Anick resolution}
In 1960's Jim Stasheff introduced the concept of $A_\infty$ spaces and $A_\infty$ algebras to study 'group like' topological spaces. In this section we give a brief introduction about $A_\infty$ algebras and computation of $A_\infty$ structures on $\Ext$-algebras. For more details we refer to [7], [8] and [22]. 
\begin{definition}
An $A_\infty$-algebra over a field $\mathbb{K}$ (also called a 'strongly homotopy associative algebra') is a $\mathbb{Z}$-graded vector space
\[ A = \bigoplus_{n \in \mathbb{Z}} A^n  \]
endowed with graded maps
\[ m_n : A^{\otimes n} \rightarrow A, \hspace{2mm} n \geq 1,  \]
of degree $2-n$ satisfying the following relations.
\begin{itemize}
\item $m_1m_1 = 0$, i.e. $(A,m_1)$ is a differential complex.
\item We have 
\[ m_1m_2 = m_2( m_1 \otimes \id + \id \otimes m_1) \]
as maps $A^{\otimes2} \rightarrow A$. Here $\id$ denotes the identity map of the space $A$. So $m_1$ is a graded derivation with respect to multiplication $m_2$.
\item We have
\[ m_2(\id \otimes m_2 - m_2 \otimes \id) = m_1m_3 +   \]
\[ m_3( m_1 \otimes \id \otimes \id + \id \otimes m_1 \otimes \id + \id \otimes \id \otimes m_1)  \]
as maps $A^{\otimes3} \rightarrow A$. This implies that $m_2$ is associative up to homotopy.
\item More generally, for $n \geq 1$, we have
\[ \sum (-1)^{r+st} m_u ( \id^{\otimes r} \otimes m_s \otimes \id^{\otimes t}) = 0  \]
where the sum runs over all decompositions $n = r+s+t$ and $u = r+t+1$.
\end{itemize}
\end{definition}
Following is an immediate consequence of the definition 7.2.1:
\begin{corollary}
If $m_n$ vanishes for all $n \geq 3$, then $A$ is an associative differential graded algebra and conversely each differential graded algebra yields an $A_\infty$-algebra structure with $m_n = 0$ for all $n \geq 3$.
\end{corollary}
\begin{theorem}[Kadeishvili]
Let $C$ be an $A_\infty$-algebra and let $E = HC$ be the cohomology ring of $C$. There exists an $A_\infty$ structure on $E$ with $m_1=0$ and $m_2$ induced by multiplication on $C$. 
\end{theorem}
\begin{proof}
For proof and more details about the theorem we refer to [23].
\end{proof}
\begin{remark}
Furthermore the $A_\infty$ structure on $E$, mentioned in theorem 7.2.1 is unique up to quasi-isomorphism[22] of $A_\infty$ algebras.
\end{remark}
\subsubsection{Merkulov's construction}
Let $A$ be a $\mathbb{K}$-algebra and let $(Q_.,d_.)$ be a minimal graded projective resolution of $\mathbb{K}$ by $A$-modules with $Q_0=A$. Here the resolution can be taken as Anick resolution and it is possible to construct a minimal resolution from Anick resolution. Let $U = \Hom_A(Q,Q)$ be the differential graded endomorphism algebra of $(Q_.,d_.)$ with multiplication given by composition. The minimality of the resolution implies that $\Hom_A(Q_.,\mathbb{K})$ is equal to its cohomology.\\
We compute the structure of an $A_\infty$  algebra on $\Hom_A(Q_.,\mathbb{K})$ by specifying the data of a strong deformation retraction from $U$ to $\Hom_A(Q_.,\mathbb{K})$. We choose maps $i : \Hom_A(Q_.,\mathbb{K}) \rightarrow U$, $p : U \rightarrow \Hom_A(Q_.,\mathbb{K})$ and a chain homotopy map $G: U \rightarrow U$ such that $i$ and $p$ are degree $0$ morphisms of complexes and $G$ is a homogeneous $\mathbb{K}$-linear map of degree $-1$ such that $pi = \id_{\Hom(Q,d)}$ and $\id_U - ip = \delta G - G\delta$, where $\delta$ is the differential on $U$ given by $\delta(f) = df - (-1)^{|f|}fd$ ($|f|$ is the degree of $f$).\\ 
We define a family of homogeneous $\mathbb{K}$-linear maps $\{ \lambda_n : U^{\otimes n} \rightarrow U \}_{n \geq 2}$ with $\deg(\lambda_n) = 2-n$ as follows:
\begin{itemize}
\item $\lambda_2$ is multiplication on $U$.
\item $G\lambda_1 = -\id_U$.
\item $\lambda_n = \sum_{\substack{s+t=n \\ s,t \geq 1}} (-1)^{s+1} \lambda_2(G\lambda_s \otimes G\lambda_t)$.
\end{itemize}
\begin{theorem}[Merkulov]
Let $m_1 =0$ and for $n \geq 2$, let $m_n = p\lambda_n i^{\otimes n}$. Then $(\Hom(Q_.,\mathbb{K}), m_1, m_2, m_3, \dots )$ is an $A_\infty$-algebra satisfying the conditions of theorem 7.2.1.
\end{theorem}
Merkulov's theorem applies to the subcomplex $i(\Hom_A(Q_.,\mathbb{K})$ of $U$ and gives structure maps $m_n = ip\lambda_n$ which together with the subcomplex forms an $A_\infty$ structure.
\begin{lemma}
$ E = \Ext_A( \mathbb{K}, \mathbb{K})$ is isomorphic to $HU$.
\end{lemma}
So we have seen how to get an $A_\infty$ algebra structure on $\Ext$-algebras.
\begin{remark}
If one use the Merkulov's construction to compute $A_\infty$-algebra structure associated to algebra $\mathbb{K}\langle x \rangle /(x^2)$ by constructing its Anick resolution then one can find that all higher maps in $A_\infty$ structure vanish for $n \geq 3$. A Koszul algebra canonically carries an $A_\infty$ structure will all higher maps vanish for $n \geq 3$ [7]. Algebra $\mathbb{K}\langle x \rangle /(x^2)$ is Koszul and hence the above computation through Merkulov's construction verifies this fact. 
\end{remark}

\end{document}